\pdfoutput=1

\documentclass[11pt]{article}

\usepackage{fullpage}
\usepackage{amsmath}
\usepackage{accents}
\usepackage{amsthm}
\usepackage{amssymb}
\usepackage[hidelinks]{hyperref}
\usepackage{cleveref}
\usepackage{textcomp}
\usepackage{enumerate}
\usepackage{graphicx}
\usepackage{caption}
\usepackage{subcaption}
\usepackage{float}

\newcommand{\mbb}[1]{\mathbb{#1}}
\newcommand{\mbf}[1]{\mathbf{#1}}
\newcommand{\bs}{\boldsymbol}
\newcommand{\tr}{\textup{tr}}
\newcommand{\wt}{\widetilde}
\newcommand{\wh}{\widehat}
\newcommand{\mc}{\mathcal}

\newcommand{\Tr}[1]{\left\langle#1\right\rangle}

\renewcommand{\det}{\textup{det}}
\renewcommand{\Re}{\textup{Re}}
\renewcommand{\Im}{\textup{Im}}

\numberwithin{equation}{section}

\newtheorem{theorem}{Theorem}
\newtheorem{lemma}{Lemma}[section]
\newtheorem{proposition}{Proposition}[section]

\numberwithin{theorem}{section}
\numberwithin{proposition}{section}
\numberwithin{corollary}{section}

\theoremstyle{remark}

\theoremstyle{definition}
\newtheorem{definition}{Definition}[section]

\title{Universality for Diagonal Eigenvector Overlaps of non-Hermitian Random Matrices}
\date{}
\author{Mohammed Osman\footnote{mohammed.osman@qmul.ac.uk}\\~\\\small Queen Mary, University of London}

\begin{document}

\maketitle

\abstract{We prove the universality of the joint distribution of an eigenvalue and the corresponding diagonal eigenvector overlap, in the bulk and at the edge, for eigenvalues of complex matrices and real eigenvalues of real matrices. As part of the proof we obtain a bound for the least non-zero singular value of $X-z$ when $z$ is an edge eigenvalue and a bound for the inner product between left and right singular vectors of $X-z$ when $|z|=1+O(N^{-1/2})$.} 

\section{Introduction}
\subsection{Overview}
Let $X$ be an $N\times N$ non-Hermitian matrix. The left and right eigenvectors $\mbf{l}_{n},\,\mbf{r}_{n}$ form a bi-orthogonal set, and by scaling the eigenvectors appropriately, we can express the bi-orthogonality by the equation
\begin{align*}
    \mbf{l}_{n}^{*}\mbf{r}_{m}&=\delta_{nm}.
\end{align*}
The remaining inner products between eigenvectors can be collected in the overlap matrix $O_{nm}$:
\begin{align}
    O_{nm}&:=\mbf{l}_{n}^{*}\mbf{l}_{m}\mbf{r}_{m}^{*}\mbf{r}_{n}.
\end{align}
The diagonal elements $O_{nn}=\|\mbf{l}_{n}\|^{2}\cdot\|\mbf{r}_{n}\|^{2}$ have received particular attention, due in part to their relevance to perturbation theory. Indeed, one can show that the maximal rate of change of an eigenvalue $z_{n}(t)$ under all possible norm 1 perturbations of the form $X+tY$ is directly related to $O_{nn}$:
\begin{align*}
    \sup_{\|Y\|=1}\frac{dz_{n}(t)}{dt}\Big|_{t=0}&=\sqrt{O_{nn}}.
\end{align*}
In view of this relation, $O_{nn}$ is sometimes called the eigenvalue condition number. By Cauchy-Schwarz we have the lower bound
\begin{align*}
    O_{nn}&\geq1,
\end{align*}
which is achieved if $X$ is normal. An immediate question that arises is: what is the typical size of $O_{nn}$ when $X$ is a random matrix? For random matrices from the complex Ginibre ensemble ($N\times N$ matrices whose entries are i.i.d. complex Gaussians with mean 0 and variance $1/N$), the first work studying the overlaps was that of Chalker--Mehlig \cite{chalker_eigenvector_1998}, who computed the following expectation values
\begin{align*}
    O(z)&=\frac{1}{N}\mbb{E}\left[\sum_{n=1}^{N}O_{nn}\delta(z-z_{n})\right],\\
    O(z_{1},z_{2})&=\frac{1}{N}\mbb{E}\left[\sum_{n\neq m}^{N}O_{nm}\delta(z_{1}-z_{n})\delta(z_{2}-z_{m})\right],
\end{align*}
where $\delta(z)=\delta_{0}(\mathrm{d}^{2}z)$ is the Dirac delta mass at 0. In the large $N$ limit they found
\begin{align*}
    N^{-1}O(z)&=\frac{1-|z|^{2}}{\pi},\\
    N^{-2}O(z_{1},z_{2})&=-\frac{1-|z|^{2}}{\pi^{2}|w|^{4}}\left[1-(1+|w|^{2})e^{-|w|^{2}}\right],
\end{align*}
where in the second line $z=\frac{1}{2}(z_{1}+z_{2})$ and $w=\sqrt{N}(z_{1}-z_{2})$. A rigourous proof of the first result for any $z$ in the bulk was later provided by Walters and Starr \cite{walters_note_2015}. We observe that the typical sizes of the overlaps correpsonding to bulk eigenvalues is $N$ and increases in the distance from the unit circle.

The full distribution of $O_{nn}$ when the associated eigenvalue $z_{n}$ is near a bulk point $z_{0}$ (i.e. $|z_{0}|<1-\delta$ for $\delta>N^{-1/2+\epsilon}$) was calculated by Bourgade--Dubach \cite{bourgade_distribution_2020} and Fyodorov \cite{fyodorov_statistics_2018}. Defining the rescaled overlap $S_{n}=\frac{O_{nn}}{N(1-|z_{n}|^{2})}$, the limiting density of $(z_{n},S_{n})$ is given by
\begin{align}
    \rho_{2,bulk}(z,s)&=\frac{1}{\pi s^{3}}e^{-\frac{1}{s}},
\end{align}
in the sense that
\begin{align*}
    \lim_{N\to\infty}\left[\mbb{E}\sum_{n=1}^{N}f(N^{1/2}(z_{n}-z_{0}),S_{n})-\int_{\mbb{C}\times[0,\infty]}Nf(N^{1/2}(z-z_{0}),s)\rho_{2,bulk}(z,s)\,\mathrm{d}m(z,s)\right]&=0,
\end{align*}
for any compactly supported $f$, where $\mathrm{d}m(z,s)$ is the Lebesgue measure on $\mbb{C}\times[0,\infty]$. The characteristic feature of this density is the $s^{-3}$ tail, implying that the variance of diverges. The same distribution has been obtained by Dubach \cite{dubach_eigenvector_2021} for the spherical (i.e. $X_{1}X^{-1}_{2}$ where $X_{1}$ and $X_{2}$ are independent complex Ginibre matrices) and truncated unitary ensembles of random matrices.

Beyond the complex Ginibre ensemble, Fyodorov \cite{fyodorov_statistics_2018} also obtained the distribution for overlaps associated to real eigenvalues of real Ginibre matrices, as well as the distribution at the edge (where $|z_{n}|=1+O(N^{-1/2})$). These results have been extended to the real elliptic ensemble (defined by $X+i\sqrt{\tau_{N}}Y$ for independent real Ginibre matrices $X$ and $Y$) by Fyodorov--Tarnowski \cite{fyodorov_condition_2021}, where the limiting distribution in the bulk of the weakly asymmetric regime (for which $\tau_{N}=O(N^{-1})$) was also derived. The corresponding result at the edge of the weakly asymmetric regime, for which $\tau_{N}=O(N^{-1/3})$, was obtained by Tarnowski \cite{tarnowski_condition_2024}. The distribution for complex eigenvalues of real matrices is still unknown, although the mean $O(z)$ has been calculated by Crumpton--Fyodorov--W\"{u}rfel \cite{wurfel_mean_2024} (see also \cite{crumpton_mean_2024}, where the same authors calculate $O(z)$ in the elliptic Ginibre ensembles). Dubach \cite{dubach_symmetries_2020} obtains the distribution at the origin for symplectic Ginibre matrices. Belinschi--Nowak--Speicher--Tarnowski \cite{belinschi_squared_2017} show that $O(z)$ can be calculated using the methods of free probability, which greatly simplifies the calculation of $O(z)$ in ensembles of matrix products. A determinantal structure for the mean overlap conditioned on $k\geq1$ eigenvalues was uncovered in the complex Ginibre ensemble by Akemann--Tribe--Tsareas--Zaboronski \cite{akemann_determinantal_2020} (see also \cite{akemann_determinantal_2020-1} for the result away from the origin), and extended to a Pfaffian structure in the symplectic Ginibre ensemble by Akemann--Byun--Noda \cite{akemann_pfaffian_2024}, building on earlier work of Akemann--F\"{o}rster--Kieburg \cite{akemann_universal_2020}.

Beyond the integrable models (i.e. Ginibre, spherical, truncated unitary), the only result of the kind described above is the recent calculation of the mean $O(z)$ for finite rank perturbations of the complex Ginibre ensemble when $z$ is near an eigenvalue of the perturbation by Zhang \cite{zhang_mean_2024}. Cipolloni--Erd\H{o}s--Henheik--Schr\"{o}der \cite[Theorem 2.4]{cipolloni_optimal_2024} prove a lower bound for overlaps corresponding to bulk eigenvalues $z_{n}$ of the form
\begin{align*}
    O_{nn}\geq N^{1-\xi},
\end{align*}
with probability $1-N^{-D}$ for any $\xi,D>0$, which holds under the assumption of real or complex matrices with i.i.d. entries $x_{ij}$ such that $\mbb{E}|N^{1/2}x_{ij}|^{p}\leq C_{p}$ for any $p>0$. The analogous lower bound for edge eigenvalues follows from \cite[Theorem 3.5]{cipolloni_universality_2024} as mentioned in a footnote. Erd\H{o}s--Ji \cite[Theorem 2.9]{erdos_wegner_2024} prove a corresponding upper bound for the expectation of the sums of overlaps corresponding to eigenvalues in any Borel set $\mc{D}\subset\mbb{C}$ of the form
\begin{align}
    \mbb{E}\sum_{z_{n}\in\mc{D}}O_{nn}&\leq CN^{1+\delta}(N|\mc{D}|)\label{eq:boundedDensity}
\end{align}
for any $\delta>0$ and complex matrices with i.i.d. entries having bounded densities, with an analogous result for real matrices. In view of Chalker--Mehlig's result, this bound is optimal up to the factor $N^{\delta}$. The results in \cite{erdos_wegner_2024} are in fact more general and allow for a deterministic shift $X+A$. Previous, slightly weaker, results of a similar nature had been obtained by Banks--Garza-Vargas--Kulkarni--Srivastava \cite{banks_overlaps_2020}, \cite{banks_gaussian_2021} and Jain--Sah--Sawhney \cite{jain_real_2021}. The motivation for results of this kind comes from the problem of finding a perturbation of a given matrix whose eigenvalue condition numbers satisfy a universal (dimension dependent) upper bound, and originates in Davies' idea of ``approximate diagonalization" \cite{davies_approximate_2008}. 

\subsection{Main Results}
\paragraph{Notation} $\mbb{M}_{n}(\mbb{F})$ and $\mbb{M}^{H}_{n}(\mbb{F})$ denote the spaces of general and Hermitian matrices with entries in $\mbb{F}$. For $M\in\mbb{M}_{n}(\mbb{F})$, $|M|=\sqrt{M^{*}M}$, $\|M\|$ denotes the operator norm, $\|M\|_{2}$ the Frobenius norm and $\Tr{M}:=n^{-1}\tr M$ the normalised trace. The real and imaginary parts of $M$ are defined by $\Re M=\frac{1}{2}(M+M^{*})$ and $\Im M=\frac{1}{2i}(M-M^{*})$ respectively. $U(n)$ and $O(n)$ denote the unitary and orthogonal groups respectively. For $N\in\mbb{N}$, we write $[N]:=[1,...,N]$ and use the shorthand $|n|\in[N]$ to mean $n=\pm1,...,\pm N$. We denote by $\mbb{C}_{+}$ the open upper half-plane, $\mbb{D}\subset\mbb{C}$ the open unit disk and $\mbb{T}\subset\mbb{C}$ the unit circle. When $x$ belongs to a coset space of a compact Lie group (e.g. $U(n),\,O(n)/O(n-m)$), we denote by $\,\mathrm{d}_{H}x$ the Haar measure with total mass equal to the volume of the corresponding space. We denote generic positive constants by $C,c$ and use the notation
\begin{align*}
    x&\lesssim y\quad\text{if } x\leq Cy,\\
    x&=O(y)\quad\text{if }|x|\lesssim |y|,\\
    x&\simeq y\quad\text{if } cy\leq x\leq Cy.
\end{align*}
If $\mc{D}$ is a probability distribution, we write $X\sim\mc{D}$ to mean that $X$ has distribution $\mc{D}$. We use $\prec$ to denote stochastic domination: $X\prec Y$ or $X=O_{\prec}(Y)$ if, for any $\xi,D>0$, $|X|\leq N^{\xi}|Y|$ with probability at least $1-N^{-D}$ for sufficiently large $N>N_{0}(D,\xi)$. If $X$ and $Y$ are deterministic then $X\prec Y$ means that $|X|\leq N^{\xi}|Y|$ for any fixed $\xi>0$. By ``very high probability" we mean probability at least $1-N^{-D}$ for any $D>0$ and sufficiently large $N>N_{0}(D)$.

\begin{definition}\label{def1}
We say that $A=(a_{jk})_{j,k=1}^{N}$ is a \textbf{non-Hermitian Wigner matrix} if $a_{jk}$ are independent complex random variables such that $\Re a_{jk}$ is independent of $\Im a_{jk}$ and
\begin{align}
    \mbb{E}\left[a_{jk}\right]&=0,\label{cond1}\\
    \mbb{E}\left[N|a_{jk}|^{2}\right]&=1,\label{cond2}\\
    \mbb{E}\left[N^{p/2}(\Re a_{jk})^{p-q}(\Im a_{jk})^{q}\right]&\lesssim C_{p},\qquad p>2\text{ and }0\leq q\leq p.\label{cond3}
\end{align}
This includes the case when $a_{jk}$ is real, i.e. $\Im a_{jk}$ is identically zero.
\end{definition}

To unify the presentation of real and complex matrices, we define the symmetry parameter $\beta$ which is equal to 1 for real matrices and 2 for complex matrices. $Gin_{\beta}(N)$ denotes the ensemble of matrices whose entries are real ($\beta=1$) or complex ($\beta=2$) i.i.d. Gaussian random variables with mean zero and variance $1/N$. We define $\mbb{F}_{1}=\mbb{R},\,\mbb{F}_{2}=\mbb{C},\,\mbb{D}_{\beta}=\mbb{D}\cap\mbb{F}_{\beta}$, $\mbb{T}_{\beta}=\mbb{T}\cap\mbb{F}_{\beta}$ and $S^{n}_{\beta}=S^{n}\cap\mbb{F}_{\beta}^{n+1}$, where $S^{n}$ is the complex $n$-sphere. Given a matrix $X\in\mbb{M}_{N}(\mbb{F}_{\beta})$, we set $N_{1}=N_{\mbb{R}}$ and $N_{2}=N$, where $N_{\mbb{R}}$ is the number of real eigenvalues and we order the eigenvalues of real matrices such that $z_{n}\in\mbb{R},\,n=1,...,N_{\mbb{R}}$.

We define the constant
\begin{align}
    v_{\beta}&=\frac{1}{\beta(2\pi)^{\beta/2-1}}\times\begin{cases}
    2&\quad\beta=1\\
    2\pi&\quad\beta=2
    \end{cases}.
\end{align}
In the bulk limit we consider the rescaled overlap 
\begin{align}
    S_{n}&:=\frac{O_{nn}}{N(1-|z_{n}|^{2})},\label{eq:Sbulk}
\end{align}
in order to obtain a universal form for the limiting density that is independent of $z$. In the Gaussian case, the following density of $(z_{n},S_{n})$
\begin{align}
    \rho_{\beta,bulk}(z,s)&:=\frac{1}{v_{\beta}s^{\beta+1}}e^{-\frac{\beta}{2s}},\label{eq:rhoBulk}
\end{align}
has been obtained by Bourgade--Dubach \cite[Theorem 1.1]{bourgade_distribution_2020} for $\beta=2$ and Fyodorov \cite[Eq. (2.7) and (2.24)]{fyodorov_statistics_2018} for $\beta=1,2$.

At the edge, we cannot remove the $z$-dependence by rescaling, so instead we consider 
\begin{align}
    S_{n}&:=\frac{O_{nn}}{N^{1/2}}.\label{eq:Sedge}
\end{align}
The following limit has been obtained by Fyodorov \cite[Eq. (2.10) and (2.25)]{fyodorov_statistics_2018} for Gaussian matrices:
\begin{align}
    \rho_{\beta,edge}(\delta_{z},s)&:=\frac{1}{v_{\beta}s^{2\beta+1}}e^{-\frac{\beta}{4s^{2}}(1-2\delta_{z}s)}\nonumber\\
    &\times\frac{1}{(2\pi)^{\beta/2}}\det\left[\int_{\delta_{z}}^{\infty}x^{j+k-2}e^{-\frac{1}{2}x^{2}}\left[1+s(x-\delta_{z})\right]\,\mathrm{d}x\right]_{j,k=1}^{\beta},\label{eq:rhoEdge}
\end{align}
where 
\begin{align}
    \delta_{z}&:=\lim_{N\to\infty}N^{1/2}(|z|^{2}-1).
\end{align}
The exact expression in \cite{fyodorov_statistics_2018} can be recovered after evaluating the integral and determinant, but we prefer to write it in this form to unify the cases $\beta=1$ and $\beta=2$. Note that when $\delta_{z}\to-\infty$, the edge density tends to the bulk density after rescaling $s\to -s/\delta_{z}$.

Our main result is the universality of these limits for non-Hermitian Wigner matrices.
\begin{theorem}\label{thm1}
Let $X$ be a non-Hermitian Wigner matrix with real ($\beta=1$) or complex ($\beta=2$) entries. Let $f\in C^{2}(\mbb{F}_{\beta})$ and $g\in C^{5}([0,\infty])$ have compact support and, for $z_{0}\in\mbb{C}$, define 
\begin{align}
    f_{\beta,z_{0}}(z)&:=N^{\beta/2}f(N^{1/2}(z-z_{0})).
\end{align}
Let $\mathrm{d}m(z,s)$ be the Lebesgue measure on $\mbb{F}_{\beta}\times[0,\infty]$. Then there is a $\tau>0$ such that:
\begin{enumerate}[i)]
\item for any fixed $z_{0}\in\mbb{D}_{\beta}$ we have
\begin{align}
    \mbb{E}\left[\frac{1}{N^{\beta/2}}\sum_{n=1}^{N_{\beta}}f_{\beta,z_{0}}(z_{n})g\left(S_{n}\right)\right]&=\int_{\mbb{F}_{\beta}\times[0,\infty]}f_{\beta,z_{0}}(z)g(s)\rho_{\beta,bulk}(z,s)\,\mathrm{d}m(z,s)+O(N^{-\tau}),
\end{align}
where $S_{n}$ is given by \eqref{eq:Sbulk};
\item for any fixed $z_{0}\in\mbb{T}_{\beta}$ we have
\begin{align}
    \mbb{E}\left[\frac{1}{N^{\beta/2}}\sum_{n=1}^{N_{\beta}}f_{\beta,z_{0}}(z_{n})g\left(S_{n}\right)\right]&=\int_{\mbb{F}_{\beta}\times[0,\infty]}f_{\beta,z_{0}}(z)g(s)\rho_{\beta,edge}(\delta_{z},s)\,\mathrm{d}m(z,s)+O(N^{-\tau}),
\end{align}
where $S_{n}$ is given by \eqref{eq:Sedge}.
\end{enumerate}
\end{theorem}

Part of the proof of Theorem \ref{thm1} uses the fact that, when $z_{n}$ is a bulk eigenvalue, the least non-zero singular value $s_{2}(z_{n})$ of $X-z_{n}$ satisfies $s_{2}(z_{n})>N^{-1-\epsilon}$ with probability $1-N^{-\epsilon}$ for sufficiently small $\epsilon$ \cite[Theorem 2.1]{osman_least_2024}, \cite[Lemma 2.4]{dubova_gaussian_2024}. Combined with an upper bound on the inner product between left and right singular vectors of $X-z$, this allows for the construction of an approximate overlap as a function of resolvents of the Hermitisation, to which we can apply Girko's formula and the moment matching method of Tao and Vu \cite{tao_random_2015}. At the edge we need the following bound.
\begin{theorem}\label{thm2}
Let $X$ be a non-Hermitian Wigner matrix with real ($\beta=1$) or complex ($\beta=2$) entries and $z_{0}\in\mbb{T}_{\beta}$. Let $s_{1}(z)\leq\cdots\leq s_{N}(z)$ be the singular values of $X-z$. Then for any fixed $r>0$ and $\epsilon\in(0,1/24)$ we have
\begin{align}
    P\left(\min_{N^{1/2}|z_{n}-z_{0}|<r}s_{2}(z_{n})<N^{-3/4-\epsilon}\right)&\prec N^{-\epsilon},
\end{align}
uniformly in $z_{0}\in\mbb{T}_{\beta}$.
\end{theorem}
Gauss-divisible matrices can be shown to satisfy the following much stronger bound:
\begin{align*}
    P\left(\min_{N^{1/2}|z_{n}-z_{0}|<r}s_{2}(z_{n})<\eta\right)&\lesssim N^{3/2}\eta^{2}|\log N^{3/4}\eta|^{2-\beta}+N^{-D},
\end{align*}
for any $\eta>0$ and $D>0$. Our use of Girko's formula and the Lindeberg method leads to the constraint $\epsilon<1/24$.

Given the above bound, we can also construct an approximate overlap for edge eigenvalues. To carry out the moment matching argument we need some a priori bounds on traces of resolvents. In the bulk these can be deduced from a suitable bound on the inner product between left and right singular vectors of $X-z$, or equivalently a bound on the quantity $\mbf{w}_{n}^{*}F\mbf{w}_{m}$, where
\begin{align*}
    F&=\begin{pmatrix}0&1_{N}\\0&0\end{pmatrix},
\end{align*}
and  $\{\mbf{w}_{n}\}_{|n|\in[N]}$ are the eigenvectors of 
\begin{align*}
    W_{z}&=\begin{pmatrix}0&X-z\\X^{*}-\bar{z}&0\end{pmatrix}.
\end{align*}
At the edge the analogous bound needs to decay in $||n|-|m||$. A large part of this paper is devoted to stengthening the bound
\begin{align}
    |\mbf{w}_{n}^{*}F\mbf{w}_{m}|&\prec N^{-1/4},\quad |n|,|m|<N^{\xi},
\end{align}
which is derived in \cite[Corollary 3.6]{cipolloni_universality_2024} as a corollary to a certain two-resolvent local law. The bound we obtain is the following.
\begin{theorem}\label{thm3}
Let $\big||z|-1\big|\lesssim N^{-1/2}$. Then there are (small) constants $c_{1},c_{2}>0$ such that
\begin{align}
    |\mbf{w}_{n}^{*}F\mbf{w}_{m}|&\prec\begin{cases}
    \left(\frac{|n|\wedge |m|}{N}\right)^{1/4}& \quad |n|\wedge|m|\geq c_{1}(|n|\vee |m|)\\
    \frac{1}{(N(|n|\vee |m|))^{1/4}}&\quad |n|\wedge|m|\leq c_{1}(|n|\vee|m|)
    \end{cases},\label{eq:svOverlapEdge}
\end{align}
for all $|n|,|m|<c_{2}N$.
\end{theorem}
The main point is that there are two regimes depending on whether $|n|$ and $|m|$ are comparable. Numerical evidence (see Section \ref{sec:numerics}) suggests that the decay of $|\mbf{w}_{n}^{*}F\mbf{w}_{m}|$ in $||n|-|m||$ becomes increasingly sharp as $|n|\vee|m|$ increases, but the bound we obtain can only capture this decay accurately for $|n|\vee|m|\lesssim1$. The lack of optimality stems from the use of the quantity $\tr \Im G_{z}(w_{1})F\Im G_{z}(w_{2})F^{*}$, where $G_{z}(w)=(W_{z}-w)^{-1}$, as an upper bound for $|\mbf{w}_{n}^{*}F\mbf{w}_{m}|^{2}$. This approach cannot give the optimal bound, since the former quantity decays smoothly in $||n|-|m||$. Nevertheless, the bound in \eqref{eq:svOverlapEdge} is strong enough for our purpose of proving comparison lemmas for certain functions of resolvents.

The outline of the rest of the paper is as follows. In Section \ref{sec:gaussDivisible} we recall the partial Schur decomposition and define the set of matrices for which, after adding a small Gaussian matrix, we can prove universality. The results in this section are the extension to edge eigenvalues of some results in \cite{maltsev_bulk_2024,osman_bulk_2024}. In Sections \ref{sec:thm1Gauss} and \ref{sec:thm2Gauss} we prove versions of Theorems \ref{thm1} and \ref{thm2} for Gauss-divisible matrices. The main idea in Section \ref{sec:thm1Gauss} is to adapt Fyodorov's calculation \cite{fyodorov_statistics_2018} to Gauss-divisible matrices. In the Gauss-divisible setting we do not have rotational invariance and our approach centres on the introduction of a certain measure on the unit sphere (see \eqref{eq:nu}) with respect to which we can prove concentration of quadratic forms. In Section \ref{sec:svOverlap} we prove Theorem \ref{thm3} using the dynamical method of Cipolloni--Erd\H{o}s--Xu \cite{cipolloni_universality_2024}. Since we want to capture the dependence of traces of resolvents $G_{z}(w)$ on the real parts of the arguments $w$, there are various points at which we cannot afford to use Cauchy--Schwarz and the resolvent identity, since this necessarily entails additional factors of $1/\Im w$. To circumvent this we need to first obtain bounds on the matrix entries of resolvent expressions and combine this with fluctuation averaging. In Section \ref{sec:comparison} we prove comparison lemmas for the least non-zero singular value and the diagonal eigenvector overlap via Girko's formula, which allow us to remove the Gaussian component and complete the proofs of Theorems \ref{thm1} and \ref{thm2}. In the appendix we prove some auxiliary results relating to local laws that we need for our arguments, and present some numerical simulations of the singular vector overlap.

\section{Partial Schur Decomposition}\label{sec:gaussDivisible}
We begin by recalling the partial Schur decomposition. For the purposes of this paper we only need the one-step decomposition. We define the space 
\begin{align}
    \Omega_{\beta}&=\mbb{F}_{\beta}\times S^{N-1}_{\beta}\times\mbb{F}_{\beta}^{N-1}\times\mbb{M}_{N-1}(\mbb{F}_{\beta}),
\end{align}
where we recall that $\mbb{F}_{1}=\mbb{R}$ and $\mbb{F}_{2}=\mbb{C}$, and consider the map
\begin{align}
    T\colon\Omega_{\beta}\ni(z,\mbf{v},\mbf{w},M)&\mapsto R(\mbf{v})\begin{pmatrix}z&\mbf{w}^{*}\\0&M\end{pmatrix}R(\mbf{v})\in\mbb{M}_{N}(\mbb{F}_{\beta}),
\end{align}
where $R(\mbf{v})$ is the Householder than exchanges $\mbf{v}$ with the first coordinate vector. We write $R(\mbf{v})=(\mbf{v},V)$, where $V\in\mbb{F}_{\beta}^{N\times(N-1)}$ is a partial isometry, i.e. $V^{*}V=1_{N-1}$. For $X\in\mbb{M}_{N}(\mbb{F}_{\beta})$ and $z\in\mbb{C}$, define the probability measure $\mu_{\beta,z}^{X}$ on $S^{N-1}_{\beta}$ by
\begin{align}
    d\mu^{X}_{\beta,z}(\mbf{v})&=\frac{1}{K_{\beta}(z)}\left(\frac{\beta N}{2 \pi t}\right)^{\beta N/2-1}e^{-\frac{\beta N}{2t}\|X_{z}\mbf{v}\|^{2}}\,\mathrm{d}_{H}\mbf{v},\label{eq:mu}
\end{align}
with normalisation $K_{\beta}(z)$. Let $\mbb{E}_{\beta,z}$ denote the expectation with respect to $\mu_{\beta,z}^{X}$. Let $\mbb{E}_{\mbf{w}}$ denote the expectation with respect to $\mbf{w}\in\mbb{F}^{N-1}_{\beta}$ distributed according to the $\mbf{v}$-dependent measure
\begin{align}
    \left(\frac{\beta N}{2\pi t}\right)^{\beta(N-1)/2}\exp\left\{-\frac{\beta N}{2t}\|\mbf{w}-\mbf{b}\|^{2}\right\}\,\mathrm{d}^{\beta}\mbf{w},
\end{align}
where $\mbf{b}=V^{*}X^{*}\mbf{v}$ and $\,\mathrm{d}^{\beta}\mbf{w}$ is the Lebesgue measure on $\mbb{F}_{\beta}^{N-1}$. Let $M_{N}=X+\sqrt{t}Y_{N}$, where $Y_{N}\sim Gin_{\beta}(N)$, and denote by $\mbb{E}_{N}$ the expectation with respect to $Gin_{\beta}(N)$. For a bounded and measurable $f:\Omega_{\beta}\to\mbb{C}$, we define the linear statistic
\begin{align*}
    \mc{L}_{\beta}(f)&:=\sum_{n=1}^{N_{\beta}}f(z_{n},\mbf{r}_{n},\mbf{w}_{n},M_{n}),
\end{align*}
where $N_{1}=N_{\mbb{R}},\,N_{2}=N$ and
\begin{align*}
    \left\{z_{n},\mbf{r}_{n},\mbf{w}_{n},M_{n}\right\}_{n\in[N]}=T^{-1}(M_{N}).
\end{align*}
Let $M_{N-1}:=X^{(\mbf{v})}+\sqrt{\frac{Nt}{N-1}}Y_{N-1}$, where $Y_{N-1}\sim Gin_{\beta}(N-1)$. Then we have
\begin{align}
    \mbb{E}_{N}\left[\mc{L}_{\beta}(f)\right]&=\frac{\beta N}{4\pi^{\beta}t}\int_{\mbb{F}_{\beta}}K_{\beta}(z)\mbb{E}_{\beta,z}\left[\mbb{E}_{\mbf{w}}\left[\mbb{E}_{N-1}\left[f\left(z,\mbf{v},\mbf{w},M_{N-1}\right)|\det(M_{N-1}-z)|^{\beta}\right]\right]\right]\,\mathrm{d}^{\beta}z,\label{eq:partialSchur}
\end{align}

To streamline the asymptotic analysis of this formula, we state some preparatory definitions and lemmas. Recall that $X_{z}=X-z$,
\begin{align*}
    W_{z}&=\begin{pmatrix}0&X-z\\X^{*}-\bar{z}&0\end{pmatrix},
\end{align*}
and $G_{z}(w)=(W_{z}-w)^{-1}$. We define the matrices $H_{z}(\eta)$ and $\wt{H}_{z}(\eta)$ by
\begin{align}
    G_{z}(i\eta)&=:\begin{pmatrix}i\eta \wt{H}_{z}(\eta)&X_{z}H_{z}(\eta)\\H_{z}(\eta)X^{*}_{z}&i\eta H_{z}(\eta)\end{pmatrix},
\end{align}
which is equivalent to
\begin{align*}
    H_{z}(\eta)&:=(\eta^{2}+|X_{z}|^{2})^{-1},\\
    \wt{H}_{z}(\eta)&:=(\eta^{2}+|X^{*}_{z}|^{2})^{-1}.
\end{align*}

\begin{definition}\label{def:Ebulk}
We say that $X\in\mc{E}_{bulk}(t)$ if there is an $\epsilon>0$ such that
\begin{itemize}
\item[] \begin{equation}\tag{A1}\|X\|\leq e^{\log^{2}N};\label{A1}\end{equation}
\item[] \begin{equation}\tag{A2}\eta\Tr{H_{z}(\eta)}\simeq1;\label{A2}\end{equation}
\item[] \begin{equation}\tag{A3}\eta^{3}\Tr{H^{2}_{z}(\eta)}\simeq1;\label{A3}\end{equation}
\item[] \begin{equation}\tag{A4}\eta^{2}\Tr{H_{z}(\eta)\wt{H}_{z}(\eta)}\lesssim1;\label{A4}\end{equation}
\item[] \begin{equation}\tag{A5}\eta\Tr{H_{z}(\eta)X_{z}H_{z}(\eta)}\lesssim1,\label{A5}\end{equation}
\end{itemize}
for $N^{-\epsilon}t\leq\eta\lesssim1$ (compare with assumption A1 in \cite{maltsev_bulk_2024}).
\end{definition}

\begin{definition}\label{def:Eedge}
We say that $X\in\mc{E}_{edge}(t)$ if, for any $z_{0}\in\mbb{C}$ with $\delta_{0}:=|z_{0}|^{2}-1\simeq t$, we have
\begin{itemize}
\item[] \begin{equation}\tag{B1}\|X\|\leq e^{\log^{2}N};\label{B1}\end{equation}
\item[] \begin{equation}\tag{B2}s_{1}(z)\gtrsim\delta_{0}^{3/2};\label{B2}\end{equation}
\item[] \begin{equation}\tag{B3}|\delta_{0}\Tr{H_{z}(0)}-1|\lesssim|z-z_{0}|/\delta_{0};\label{B3}\end{equation}
\item[] \begin{equation}\tag{B4}\Tr{H^{2}_{z}(0)}\simeq \delta_{0}^{-4};\label{B4}\end{equation}
\item[] \begin{equation}\tag{B5}\Tr{H_{z}(0)\wt{H}_{z}(0)}\simeq\delta_{0}^{-3};\label{B5}\end{equation}
\item[] \begin{equation}\tag{B6}|\Tr{H_{z}(\eta)X_{z}H_{z}(\eta)}|\simeq \delta_{0}^{-2}\quad\text{for}\quad |\eta^{2}|<\delta_{0}^{2}|z-z_{0}|;\label{B6}\end{equation}
\end{itemize}
for $|z-z_{0}|\lesssim N^{-1/2}$.
\end{definition}

The fact that a non-Hermitian Wigner matrix $X$ belongs to $\mc{E}_{bulk}(t)$ for $t\geq N^{-1/2+\epsilon}$ with very high probability is proven in \cite[Section 8]{maltsev_bulk_2024} and follows from the two-resolvent local law in \cite[Theorem 3.3]{cipolloni_mesoscopic_2024} for complex matrices and \cite[Corollary B.4]{cipolloni_maximum_2024} for real matrices. In Section \ref{sec:gap} we show that $X\in\mc{E}_{edge}(t)$ for $t\geq N^{-2/5+\epsilon}$, also with very high probability.

In each lemma of this section we assume the following:
\begin{itemize}
\item $|z-z_{0}|\lesssim N^{-1/2}$;
\item $t\geq N^{-2/5+\epsilon}$;
\item $X\in\mc{E}_{bulk}(t)$ or $X\in\mc{E}_{edge}(t)$ if $z_{0}\in\mbb{D}_{\beta}$ or $z_{0}\in\sqrt{1+t}\mbb{T}_{\beta}$ respectively;
\end{itemize}
Note that the variance of the entries of $X+\sqrt{t}Y$ is $1+t$ so the edge is at $\sqrt{1+t}\mbb{T}_{\beta}$ rather than $\mbb{T}_{\beta}$. We will only give proofs for $z_{0}\in\sqrt{1+t}\mbb{T}_{\beta}$ since the bulk case is dealt with in \cite{maltsev_bulk_2024,osman_bulk_2024}.

We start with the equation 
\begin{align}
    t\Tr{H_{z}(\eta)}&=1,\qquad\eta^{2}>-s^{2}_{1}(z).\label{eq:critical}
\end{align}
\begin{lemma}\label{lem:eta}
The equation \eqref{eq:critical} satisfies the following:
\begin{enumerate}[i)]
\item for $z_{0}\in\mbb{D}_{\beta}$, there is unique solution $\eta_{z,t}>0$ such that
\begin{align}
    \eta_{z,t}&\simeq t;\label{eq:eta_zBulk}
\end{align}
\item for $z_{0}\in\sqrt{1+t}\mbb{T}_{\beta}$, there is a unique solution $\eta_{z,t}$ such that
\begin{align}
    |\theta_{z,t}|&\lesssim\frac{t^{2}}{N^{1/2}},\label{eq:eta_zEdge}
\end{align}
where $\theta_{z,t}:=\eta^{2}_{z,t}$.
\end{enumerate}
\end{lemma}
In the edge regime the solution $\eta_{z,t}$ can be real or imaginary depending on the sign of $t\Tr{H_{z}(0)}-1$, so we define $\theta_{z,t}=\eta^{2}_{z,t}$ in order to work with a real-valued quantity.
\begin{proof}
We rewrite \eqref{eq:critical} in the following form:
\begin{align*}
    1&=t\Tr{H_{z}(0)}-t\eta^{2}\Tr{H_{z}(\eta)H_{z}(0)}.
\end{align*}
Let $\omega>0$ and $\eta^{2}=\omega N^{-1/2}t^{2}$. By \eqref{B2} and \eqref{B4} we have
\begin{align*}
    \Tr{H_{z}(\eta)H_{z}(0)}&=\Tr{\left(1+\eta^{2}H_{z}(0)\right)^{-1}H^{2}_{z}(0)}\\
    &=\left[1+O\left(\frac{\eta^{2}}{t^{3}}\right)\right]\Tr{H^{2}_{z}(0)}\\
    &\simeq\frac{1}{t^{4}}.
\end{align*}
Thus we have
\begin{align*}
    t\Tr{H_{z}(0)}-t\eta^{2}\Tr{H_{z}(\eta)H_{z}(0)}&\leq1+\frac{C_{1}|z-z_{0}|}{t}-\frac{C_{2}\eta^{2}}{t^{3}}\\
    &\leq1+\frac{C_{1}}{N^{1/2}t}-\frac{C_{2}\omega}{N^{1/2}t}.
\end{align*}
We can choose $\omega$ large enough so that this is strictly less that 1. Repeating the same steps with $-\omega$ we conclude that the solution satisfies \eqref{eq:eta_zEdge}.
\end{proof}

We use the shorthand $G_{z,t}:=G_{z}(i\eta_{z,t}),\,H_{z,t}:=H_{z}(\eta_{z,t})$ and $\wt{H}_{z,t}:=\wt{H}_{z}(\eta_{z,t})$ for resolvents evaluated at $\eta_{z,t}$. Now we define $\phi_{z}:[0,\infty]\to\mbb{R}$ by
\begin{align}
    \phi_{z}(\eta)&=\frac{\eta^{2}}{t}-\Tr{\log(\eta^{2}+|X_{z}|^{2})}.\label{eq:phi_z}
\end{align}
We also use the shorthand $\phi_{z,t}:=\phi_{z}(\eta_{z,t})$. The relevant properties of $\phi_{z}$ are summarised in 
\begin{lemma}\label{lem:phi}
The function $\phi_{z}(\eta)$ defined in \eqref{eq:phi_z} satisfies the following:
\begin{enumerate}[i)]
\item for $z_{0}\in\mbb{D}_{\beta}$ we have
\begin{align}
    \phi_{z}(\eta)-\phi_{z,t}&\geq\frac{C(\eta-\eta_{z,t})^{2}}{t},\label{eq:phiBoundBulk}
\end{align}
uniformly in $\eta\geq0$, and
\begin{align}
    \phi_{z}(\eta)-\phi_{z,t}&=2\eta^{2}_{z,t}\Tr{H^{2}_{z,t}}(\eta-\eta_{z,t})^{2}+O\left(\frac{\log^{3}N}{\sqrt{Nt}}\right),\label{eq:phiEstimateBulk}
\end{align}
uniformly in $|\eta-\eta_{z,t}|<\sqrt{\frac{t}{N}}\log N$.
\item for $z_{0}\in\sqrt{1+t}\mbb{T}_{\beta}$ we have
\begin{align}
    \phi_{z}(\eta)-\phi_{z,t}&\geq Ct^{2}\cdot\frac{(\eta^{2}-\theta_{z,t})^{2}/t^{6}}{1+|\eta^{2}-\theta_{z,t}|/t^{3}},\label{eq:phiBoundEdge}
\end{align}
uniformly in $\eta\geq0$, and
\begin{align}
    \phi_{z}(\eta)-\phi_{z,t}&=\frac{1}{2}\Tr{H^{2}_{z,t}}(\eta^{2}-\theta_{z,t})^{2}+O\left(\frac{\log^{3}N}{\sqrt{Nt^{2}}}\right),\label{eq:phiEstimateEdge}
\end{align}
uniformly in $0\leq\eta\leq\frac{t\log N}{N^{1/4}}$.
\end{enumerate}
\end{lemma}
\begin{proof}
Using the inequality $\log(1+x)\leq\frac{x(6+x)}{6+4x}$, \eqref{B2} and \eqref{B4} we obtain \eqref{eq:phiBoundEdge}. When $|\eta^{2}-\theta_{z,t}|<\frac{t^{2}\log N}{N^{1/2}}$, \eqref{eq:phiEstimateEdge} follows by Taylor expansion.
\end{proof}

For the measure $\mu^{X}_{\beta,z}$, the asymptotics of the normalisation $K_{\beta}(z)$ have the same form in the bulk and the edge. The proof at the edge is exactly the same as in the bulk and is therefore omitted.
\begin{lemma}\label{lem:K}
Let $K_{\beta}(z)$ be the normalisation of the measure $\mu^{X}_{\beta,z}$ defined in \eqref{eq:mu}. Then for $z_{0}\in\mbb{D}_{\beta}$ or $z_{0}\in\sqrt{1+t}\mbb{T}_{\beta}$ we have
\begin{align}
    K_{\beta}(z)&=\left[1+O\left(\frac{\log^{3}N}{\sqrt{Nt}}\right)\right]\sqrt{\frac{4\pi}{\beta N\Tr{H^{2}_{z,t}}}}e^{\frac{\beta N}{2}\phi_{z,t}}.\label{eq:K}
\end{align}
\end{lemma}

We also have the following result on concentration of quadratic forms, whose proof we again omit since it is exactly the same as the proof of \cite[Lemma 7.2]{maltsev_bulk_2024} (see also the proof of Lemma \ref{lem:nuConc} below).
\begin{lemma}\label{lem:conc}
Let $z\in\mbb{C}$ and $\mbf{v}$ be distributed according to $\mu^{X}_{\beta,z}$ defined in \eqref{eq:mu}. Then for any $c_{1}>0$ there is a $c_{2}>0$ such that:
\begin{enumerate}[i)]
\item for $z_{0}\in\mbb{D}_{\beta}$ we have
\begin{align}
    \left|\eta\mbf{v}^{*}H_{z}(\eta)\mbf{v}-t\eta\Tr{H_{z,t}H_{z}(\eta)}\right|&\leq\frac{\log N}{\sqrt{Nt^{3}}},\\
    \left|\eta\mbf{v}^{*}\wt{H}_{z}(\eta)\mbf{v}-t\eta\Tr{H_{z,t}\wt{H}_{z}(\eta)}\right|&\leq\frac{\log N}{\sqrt{Nt^{2}}},\\
    \left|\mbf{v}^{*}X_{z}H_{z}(\eta)\mbf{v}-t\Tr{H_{z,t}X_{z}H_{z}(\eta)}\right|&\leq\frac{\log N}{\sqrt{Nt^{2}}},
\end{align}
uniformly in $\eta\geq c_{1}t$ with probability $1-e^{-c_{2}\log^{2}N}$;
\item for $z_{0}\in\sqrt{1+t}\mbb{T}_{\beta}$ we have
\begin{align}
    t^{3}\left|\mbf{v}^{*}H_{z}(\eta)\mbf{v}-t\Tr{H_{z,t}H_{z}(\eta)}\right|&\leq\frac{\log N}{\sqrt{Nt^{2}}},\\
    t^{2}\left|\mbf{v}^{*}\wt{H}_{z}(\eta)\mbf{v}-t\Tr{H_{z,t}\wt{H}_{z}(\eta)}\right|&\leq\frac{\log N}{\sqrt{Nt^{3}}},\\
    t\left|\mbf{v}^{*}X_{z}H_{z}(\eta)\mbf{v}-t\Tr{H_{z,t}X_{z}H_{z}(\eta)}\right|&\leq\frac{\log N}{\sqrt{Nt^{2}}},
\end{align}
uniformly in $|\eta^{2}|\leq c_{1}t^{3}$ with probability $1-e^{-c_{2}\log^{2}N}$.
\end{enumerate}
\end{lemma}

We also need to extend Lemmas \ref{lem:eta} and \ref{lem:phi} to projections onto subspaces of codimension 1. For a partial isometry $U\in\mbb{C}^{N\times n}$, we denote by $X^{(U)}$ the projection of $X$ onto the cokernel of $U$. We denote by a superscript $(U)$ any quantity obtained by replacing $X$ with $X^{(U)}$ in its definition, e.g. $H^{(U)}_{z}(\eta):=(\eta^{2}+|X^{(U)}_{z}|^{2})^{-1}$.
\begin{lemma}\label{lem:minor}
Let $\mbf{v}\in\mbb{S}^{N-1}_{\beta}$. Then: 
\begin{enumerate}[i)]
\item for $z_{0}\in\mbb{D}_{\beta}$ we have
\begin{align}
    |\eta^{(\mbf{v})}_{z,t}-\eta_{z,t}|&\lesssim\frac{1}{N},\label{eq:etaMinorBulk}\\
    e^{N\left(\phi_{z,t}-\phi^{(\mbf{v})}_{z,t}\right)}&=\left[1+O\left(\frac{1}{\sqrt{Nt^{2}}}\right)\right]|\det\bigl((1_{2}\otimes \mbf{v}^{*})G_{z,t}(1_{2}\otimes \mbf{v})\bigr)|;\label{eq:minorBulk}
\end{align}
\item for $z_{0}\in\sqrt{1+t}\mbb{T}_{\beta}$ we have
\begin{align}
    |\theta^{(\mbf{v})}_{z,t}-\theta_{z,t}|&\lesssim\frac{t}{N},\label{eq:etaMinorEdge}
\end{align}
uniformly in $\mbf{v}\in\mc{E}_{conc}$, and
\begin{align}
    e^{N\left(\phi_{z,t}-\phi^{(\mbf{v})}_{z,t}\right)}&=\left[1+O\left(\frac{1}{\sqrt{Nt^{2}}}\right)\right]|\det\bigl((1_{2}\otimes \mbf{v}^{*})G_{z,t}(1_{2}\otimes \mbf{v})\bigr)|.\label{eq:minorEdge}
\end{align}
\end{enumerate}
\begin{proof}
First we claim that
\begin{align}
    s^{(\mbf{v})}_{1}(z)\gtrsim\delta_{0}^{3/2},\label{eq:svMinor}
\end{align}
uniformly in $\mbf{v}\in\mc{E}_{conc}$, where $s^{(\mbf{v})}_{n}(z),\,n\in[N]$ are the singular values of $X^{(\mbf{v})}_{z}$. If this is true then
\begin{align}
    \Tr{(H^{(\mbf{v})}_{z}(\eta))^{l}}&=\Tr{H^{l}_{z}(\eta)}+O\left(\frac{1}{N(s_{1}^{2}(z)+\eta^{2})^{l}}\right),\label{eq:minorTrace}
\end{align}
and we can repeat the estimates in Lemmas \ref{lem:eta} and \ref{lem:phi} for $\eta^{(\mbf{v})}_{z}$ and $\phi^{(\mbf{v})}_{z}(\eta)$. The bound in \eqref{eq:etaMinorEdge} follows from \eqref{eq:minorTrace} and the estimate
\begin{align*}
    \theta^{(\mbf{v})}_{z,t}&=\frac{t\Tr{H^{(\mbf{v})}_{z}(0)}-1}{t\Tr{(H^{(\mbf{v})}_{z}(0))^{2}}}+O\left(\frac{t}{N}\right).
\end{align*}

To prove \eqref{eq:svMinor}, we note that by \cite[Lemma 3.1]{maltsev_bulk_2024} we have
\begin{align*}
    \left\|G^{(\mbf{v})}_{z}(i\eta)\right\|&\leq\|G_{z}(i\eta)\|+\left\|G_{z}(i\eta)(1_{2}\otimes\mbf{v})(1_{2}\otimes\mbf{v}^{*}G_{z}(i\eta)1_{2}\otimes\mbf{v})^{-1}(1_{2}\otimes\mbf{v}^{*})G_{z}(i\eta)\right\|.
\end{align*}    
Using $\|A\|^{2}=\|AA^{*}\|\leq\tr AA^{*}$ we can bound the second term by
\begin{align*}
    \frac{\sqrt{2\mbf{v}^{*}H_{z}(\eta)\mbf{v}\mbf{v}^{*}\wt{H}_{z}(\eta)\mbf{v}}}{|\mbf{v}^{*}X_{z}H_{z}(\eta)\mbf{v}|}.
\end{align*}
Setting $\eta=\eta_{z,t}$, on $\mc{E}_{conc}$ the above term is bounded by
\begin{align*}
    \frac{\sqrt{\Tr{H^{2}_{z,t}}}\sqrt{\Tr{H_{z,t}\wt{H}_{z,t}}}}{|\Tr{H_{z,t}X_{z}H_{z,t}}|}\lesssim\frac{1}{t^{3/2}},
\end{align*}
where we have used \eqref{B4}, \eqref{B5} and \eqref{B4}. Since we also have $s_{1}(z)\gtrsim t^{3/2}$, we obtain \eqref{eq:minorTrace}.

Using Cramer's rule we have
\begin{align*}
    e^{N(\phi_{z,t}-\phi^{(\mbf{v})}_{z,t})}&=e^{N(\phi_{z,t}-\phi_{z}(\eta^{(\mbf{v})}_{z,t}))}\left|\det1_{2}\otimes U^{*}G_{z}(\eta^{(\mbf{v})}_{z,t})1_{2}\otimes U\right|.
\end{align*}
By \eqref{eq:phiEstimateEdge} and \eqref{eq:etaMinorEdge} we have
\begin{align*}
    |\phi_{z,t}-\phi_{z}(\eta^{(\mbf{v})}_{z,t})|&\leq\frac{C}{N^{2}t^{2}}.
\end{align*}
By the resolvent identity, \eqref{eq:etaMinorEdge} and \eqref{B2} we have
\begin{align*}
    \left|\frac{\det1_{2}\otimes \mbf{v}^{*}G_{z}(\eta^{(\mbf{v})}_{z,t})1_{2}\otimes \mbf{v}}{\det1_{2}\otimes \mbf{v}^{*}G_{z,t}1_{2}\otimes \mbf{v}}\right|&=1+O\left(\frac{1}{\sqrt{Nt^{2}}}\right).
\end{align*}
Combining the previous two bounds we obtain \eqref{eq:minorEdge}.
\end{proof}
\end{lemma}

The global eigenvalue density of Gauss-divisible matrices is expressed in terms of 
\begin{align}
    \sigma_{z,t}&=\eta^{2}_{z,t}\Tr{H_{z,t}\wt{H}_{z,t}}+\frac{|\Tr{H_{z,t}X_{z}H_{z,t}}|^{2}}{\Tr{H^{2}_{z,t}}}.\label{eq:sigma}
\end{align}
\begin{lemma}\label{lem:sigma}
The density $\sigma_{z,t}$ satisfies
\begin{align}
    \sigma_{z,t}\simeq1.
\end{align}
\end{lemma}
\begin{proof}
At the edge this follows directly from \eqref{B4}, \eqref{B5} and \eqref{B6}.
\end{proof}

Finally, we note the following linear algebra identity.
\begin{lemma}\label{lem:b}
Let $U=(U_{1},U_{2})\in U(N)$ where $U^{*}_{1}U_{1}=1_{m}$ and $B=U_{2}^{*}X^{*}U_{1}$. Then
\begin{align}
    &1_{m}+B^{*}H^{(U_{1})}_{z}(\eta)B\nonumber\\&=\left(\eta^{2}(U^{*}_{1}\wt{H}_{z}(\eta)U_{1})+U^{*}_{1}X_{z}H_{z,t}U_{1}(U^{*}_{1}H_{z}(\eta)U_{1})^{-1}(U^{*}_{1}H_{z}(\eta)X_{z}^{*}U_{1})\right)^{-1}.
\end{align}
\end{lemma}
\begin{proof}
This follows by simple algebraic manipulations using the identities
\begin{align*}
    B^{*}H^{(U_{1})}_{z}B&=\begin{pmatrix}0&U^{*}_{1}X_{z}\end{pmatrix}(1_{2}\otimes U)G^{(U_{1})}_{z}(1_{2}\otimes U^{*})\begin{pmatrix}0\\X^{*}_{z}U_{1}\end{pmatrix},
\end{align*}
and
\begin{align*}
    &(1_{2}\otimes U)\begin{pmatrix}0&0\\0&G^{(U_{1})}_{z}\end{pmatrix}(1_{2}\otimes U)\\&=G_{z}-G_{z}(1_{2}\otimes U_{1})((1_{2}\otimes U^{*}_{1})G_{z}(1_{2}\otimes U_{1}))^{-1}(1_{2}\otimes U^{*}_{1})G_{z}.
\end{align*}
\end{proof}

After this preparatory work, we are in a position to prove Theorems \ref{thm1} and \ref{thm2} for Gauss-divisible matrices in the next two sections.

\section{Proof of Theorem \ref{thm1} for Gauss-divisible Matrices}\label{sec:thm1Gauss}
The proof of Theorem \ref{thm1} for Gauss-divisible matrices will be done by an adaptation of Fyodorov's method \cite{fyodorov_statistics_2018}. In \cite{fyodorov_statistics_2018}, the problem of computing the Laplace transform of the overlap is reduced to the evaluation of the expectation value of a ratio of characteristic polynomials. This expectation value can be computed exactly at finite $N$ using the supersymmetry method and the rotational invariance of the Gaussian distribution. The large $N$ limit can then be extracted directly from the resulting formula.

In our case the rotational invariance is broken by the deterministic shift $X$ and so some changes are required. Instead of computing the Laplace transform, we directly compute the expectation of an arbitrary bounded function of the overlap. We avoid the ratio of characteristic polynomials by a change of variables and the introduction of a measure on $S^{N-2}_{\beta}$ which is closely related to the measure $\mu^{X}_{\beta,z}$ defined in \eqref{eq:mu}. The analysis of this measure can be done following a similar approach to the analysis of $\mu^{X}_{\beta,z}$

The starting point in \cite{fyodorov_statistics_2018} is the observation that the overlaps are invariant under unitary conjugation, and so we can compute instead the overlap of the partially upper triangular matrix in the partial Schur decomposition. Let $M_{N}=X+\sqrt{t}Y_{N}$ have an eigenvalue $z=z_{n}$ and compute the partial Schur decomposition
\begin{align}
    M_{N}&=R(\mbf{v})\begin{pmatrix}z&\mbf{w}^{*}\\0&M_{N-1}\end{pmatrix}R(\mbf{v}).
\end{align}
Then the overlap $O_{nn}$ corresponding to $z$ is given by
\begin{align}
    O_{nn}&:=O_{nn}(z,\mbf{w},M_{N-1})=1+\|(M^{*}_{N-1}-\bar{z})^{-1}\mbf{w}\|^{2}.\label{eq:partialSchurOverlap}
\end{align}

We need to scale the overlap differently according to whether we are in the bulk or edge regime. Thus we define
\begin{align}
    S_{n}&=\frac{t^{2}\sigma_{z,t}}{N\eta^{2}_{z,t}}O_{nn},\label{eq:SBulkGauss}
\end{align}
when $|z|<1-\omega$, and
\begin{align}
    S_{n}&=\frac{t^{2}\Tr{H^{2}_{z,t}}^{1/2}\sigma_{z,t}}{N^{1/2}}O_{nn},\label{eq:SEdgeGauss}
\end{align}
when $\big||z|-1\big|<CN^{-1/2}$.
\begin{proposition}\label{prop:gaussDivisible}
Let $\epsilon>0$,\,$t=N^{-1/3+\epsilon}$ and $M_{N}=X+\sqrt{t}Y_{N}$ for $Y_{N}\sim Gin_{\beta}(N)$. For any $f:\mbb{F}_{\beta}\to\mbb{C}$ with compact support and bounded $g:[0,\infty]\to\mbb{C}$ there is a $\tau>0$ such that the following holds for sufficiently large $N$:
\begin{enumerate}[i)]
\item if $z_{0}\in\mbb{D}_{\beta}$ and $X\in\mc{E}_{bulk}(t)$ then
\begin{align}
    \mbb{E}_{N}\left[\frac{1}{N^{\beta/2}}\sum_{n=1}^{N_{\beta}}f_{\beta,z_{0}}(z_{n})g\left(S_{n}\right)\right]&=\int_{\mbb{F}_{\beta}\times[0,\infty]}f_{\beta,z_{0}}(z)g(s)\rho_{\beta,bulk}(z,s)\,\mathrm{d}m(z,s)+O(N^{-\tau}),\label{eq:gaussDivisibleBulk}
\end{align}
where $S_{n}$ is given by \eqref{eq:SBulkGauss};
\item if $z_{0}\in\sqrt{1+t}\mbb{T}_{\beta}$ and $X\in\mc{E}_{edge}(t)$ then
\begin{align}
    \mbb{E}_{N}\left[\frac{1}{N^{\beta/2}}\sum_{n=1}^{N_{\beta}}f_{z_{0}}(z_{n})g\left(S_{n}\right)\right]&=\int_{\mbb{F}_{\beta}\times[0,\infty]}f_{\beta,z_{0}}(z)g\left(s\right)\rho_{\beta,edge}(\delta_{z,t},s)\,\mathrm{d}m(z,s)+O(N^{-\tau}),\label{eq:gaussDivisibleEdge}
\end{align}
where $S_{n}$ is given by \eqref{eq:SEdgeGauss} and
\begin{align}
    \delta_{z,t}&:=\sqrt{N\Tr{H_{z,t}^{2}}}\theta_{z,t}.\label{eq:delta_zt}
\end{align}
\end{enumerate}
\end{proposition}
Since the proofs of \eqref{eq:gaussDivisibleBulk} and \eqref{eq:gaussDivisibleEdge} are very similar we give the details for the latter and sketch the requisite adjustments for the former at the end.

The proof begins with a formula for the expectation value on the left hand side. To prepare for its statement we make some definitions. For $\mbf{v}\in S^{N-1}_{\beta}$ and $s\in(0,\infty)$, we define a probability measure $\nu^{(\mbf{v})}_{\beta,z,s}$ on $S^{N-2}_{\beta}$ by
\begin{align}
    d\nu^{(\mbf{v})}_{\beta,z,s}(\mbf{u})&=\frac{1}{K_{\beta,z}(s)}\left(\frac{\beta Ns}{2\pi t(1+s)}\right)^{\beta (N-1)/2-1}e^{-\frac{\beta Ns}{2t(1+s)}\left\|X_{z}^{(\mbf{v})*}\mbf{u}-\frac{1}{\sqrt{s}}\mbf{b}\right\|^{2}}\,\mathrm{d}_{H}\mbf{u},\label{eq:nu}
\end{align}
where $K_{\beta,z}(s)$ is the normalisation. This is similar to the measure $\mu^{X}_{\beta,z}$ defined in \eqref{eq:mu}, except the Gaussian weight is shifted by the vector $\frac{1}{\sqrt{s}}\mbf{b}$. Denoting by $\mbb{E}_{\beta,z,s}$ the expectation with respect to $\nu^{(\mbf{v})}_{\beta,z,s}$, we define
\begin{align}
    P^{(\mbf{v})}(\sigma,s)&:=\mbb{E}_{\beta,z,s}\left[|1+\sqrt{s}\mbf{u}^{*}\wt{H}^{(\mbf{v})}_{z}(\sigma)X^{(\mbf{v})}_{z}\mbf{b}|^{2}+s\sigma^{2}(1+\mbf{b}^{*}H^{(\mbf{v})}_{z}(\sigma)\mbf{b})(\mbf{u}^{*}\wt{H}^{(\mbf{v})}_{z}(\sigma)\mbf{u})\right],\label{eq:P}
\end{align}
and the integral
\begin{align}
    I^{(\mbf{v})}_{\beta}(s)&=\frac{2^{\beta}N^{\beta^{2}}}{\beta!t^{\beta^{2}}}\int_{[0,\infty]^{\beta}}\Delta^{2}_{\beta}(\bs\eta^{2})\prod_{j=1}^{\beta}e^{-N\left[\phi^{(\mbf{v})}_{z}(\eta_{j})-\phi^{(\mbf{v})}_{z,t}\right]}P^{(\mbf{v})}(\eta_{j},s)\eta_{j}\,\mathrm{d}\eta_{j},\label{eq:I}
\end{align}
where $\Delta_{1}=1,\,\Delta_{2}(\bs\eta^{2})=\eta_{1}^{2}-\eta_{2}^{2}$. We then define
\begin{align}
    \psi^{(\mbf{v})}_{\beta}(s)&=e^{-\frac{\beta N}{2}\phi^{(\mbf{v})}_{z,t}}K_{\beta,z}(s)I^{(\mbf{v})}_{\beta}(s),\label{eq:psi}
\end{align}
and
\begin{align}
    \rho_{\beta,N}(z,s)&=\frac{1}{4\pi^{\beta-1}N^{\beta/2}}\left(\frac{\beta N}{2\pi t}\right)^{2}e^{-\frac{\beta N}{2}\phi_{z,t}}K_{\beta}(z)\mbb{E}_{\beta,z}\left[1_{\mc{E}_{conc}}e^{\frac{\beta N}{2}(\phi_{z,t}-\phi^{(\mbf{v})}_{z,t})}\psi^{(\mbf{v})}_{\beta}(s)\right].\label{eq:rhoN}
\end{align}
We recall that $\mc{E}_{conc}$ is the event that the estimates in Lemma \ref{lem:conc} hold. The function $\rho_{\beta,N}$ is an approximation of the joint density of $(z_{n},S_{n})$. In fact, without the factor $1_{\mc{E}_{conc}}$, it is the exact density, but we find it more convenient to insert this factor at this point.
\begin{lemma}\label{lem:formula}
We have
\begin{align}
    \mbb{E}_{N}\left[\frac{1}{N^{\beta/2}}\sum_{n=1}^{N_{\beta}}f_{\beta,z_{0}}(z_{n})g\left(S_{n}\right)\right]&=\int_{\mbb{F}_{\beta}\times[0,\infty]}f_{\beta,z_{0}}(z)g\left(\frac{1+s}{s_{0}}\right)\rho_{\beta,N}(z,s)\,\mathrm{d}m(z,s)\nonumber\\
    &+O(e^{-C\log^{2}N}),
\end{align}
where
\begin{align}
    s_{0}&=\begin{cases}
    \frac{N\eta_{z,t}^{2}}{t^{2}\sigma_{z,t}}&\quad z_{0}\in\mbb{D}_{\beta}\\
    \frac{N^{1/2}\sigma_{z,t}}{t^{2}\Tr{H^{2}_{z,t}}^{1/2}}&\quad z_{0}\in\sqrt{1+t}\mbb{T}_{\beta}
    \end{cases}.\label{eq:s0}
\end{align}
\end{lemma}
\begin{proof}
Using \eqref{eq:partialSchur} and the expression in \eqref{eq:partialSchurOverlap} for the overlap, we obtain
\begin{align*}
    \mbb{E}_{N}\left[\sum_{n=1}^{N_{\beta}}f_{z_{0}}(z_{n})g(S_{n})\right]&=\int_{\mbb{F}_{\beta}}f_{\beta,z_{0}}(z)F_{\beta,z_{0}}(z)\,\mathrm{d}^{\beta}z,
\end{align*}
where
\begin{align*}
    F_{\beta,z_{0}}(z)&:=\frac{\beta NK_{\beta}(z)}{4\pi^{\beta}N^{\beta/2}t}\mbb{E}_{\beta,z}\left[\mbb{E}_{\mbf{w}}\left[\mbb{E}_{N-1}\left[g\left(\frac{1+\|(M^{*}_{N-1}-\bar{z})^{-1}\mbf{w}\|^{2}}{s_{0}}\right)|\det(M_{N-1}-z)|^{\beta}\right]\right]\right].
\end{align*}
By the supersymmetry method and the estimates on $\phi_{z}$ in Lemma \ref{lem:phi}, we can obtain the bound
\begin{align*}
    \mbb{E}_{N-1}\left[|\det(M_{N-1}-z)|^{\beta}\right]&\lesssim\frac{e^{-\frac{\beta N}{2}\phi^{(\mbf{v})}_{z,t}}}{\sqrt{N\Tr{(H^{(\mbf{v})}_{z,t})^{2}}}}.
\end{align*}
We omit the details since they are very similar to the computations in Section \ref{sec:thm2Gauss}. Using Lemma \ref{lem:minor}, the asymptotics for $K_{\beta}(z)$ in Lemma \ref{lem:K} and the fact that $g$ is bounded, we conclude that the contribution from $\mc{E}^{c}_{conc}$ is $O(e^{-C\log^{2}N})$.

Now we interchange $\mbb{E}_{\mbf{w}}$ and $\mbb{E}_{N-1}$ and change variable $\mbf{w}\mapsto (M^{*}_{N-1}-\bar{z})\mbf{w}$, so that $\mbf{w}$ is now distributed according to
\begin{align*}
    \left(\frac{\beta N}{2\pi t}\right)^{\beta(N-1)/2}\exp\left\{-\frac{\beta N}{2t}\left\|(M^{*}_{N-1}-\bar{z})\mbf{w}-\mbf{b}\right\|^{2}\right\}|\det(M_{N-1}-\bar{z})|^{\beta}\,\mathrm{d}^{\beta}\mbf{w}.
\end{align*}

Interchanging once more $\mbb{E}_{N-1}$ and $\mbb{E}_{\mbf{w}}$, we consider first the expectation $\mbb{E}_{N-1}$, which we recall is the expectation with respect to $M_{N-1}=X^{(\mbf{v})}+\sqrt{\frac{Nt}{N-1}}Y_{N-1}$ for fixed $\mbf{v}$, where $Y_{N-1}\sim Gin_{\beta}(N-1)$. As well as the term $|\det(M_{N-1}-z)|^{2\beta}$, there are $Y_{N-1}$-dependent terms coming from the density of $\mbf{w}$ given by
\begin{align*}
    \exp\left\{-\frac{\beta N}{2}\mbf{w}^{*}Y_{N-1}Y_{N-1}^{*}\mbf{w}-\frac{\beta N}{\sqrt{t}}\Re\tr Y_{N-1}(X_{z}^{(\mbf{v})*}\mbf{w}-\mbf{b})\mbf{w}^{*}\right\}.
\end{align*}
We can evaluate the expectation over $Y_{N-1}$ by the supersymmetry method to obtain
\begin{align*}
    F_{\beta}(z,\lambda)&=\frac{1}{2\pi^{\beta-1}N^{\beta/2}}\left(\frac{\beta N}{2\pi t}\right)^{\beta (N-1)/2+1}K_{\beta}(z)\mbb{E}_{\beta,z}\Big[\\
    &\times\int_{\mbb{F}_{\beta}^{N-1}}\frac{g_{z_{0}}\left((1+\|\mbf{w}\|^{2})/s_{0}\right)}{(1+\|\mbf{w}\|^{2})^{\beta (N+3)/2}}e^{-\frac{\beta N}{2t(1+\|\mbf{w}\|^{2})}\|X^{(\mbf{v})*}_{z}\mbf{w}-\mbf{b}\|^{2}}\nonumber\\
    &\times\left(\frac{N}{\pi t}\right)^{\beta^{2}}\int_{\mbb{M}_{\beta}(\mbb{C})}e^{-\frac{N}{t}\tr|Q|^{2}}\det\begin{pmatrix}iQ\otimes(1+\mbf{w}\mbf{w}^{*})&1_{\beta}\otimes(X^{(\mbf{v})}_{z}+\mbf{w}\mbf{b}^{*})\\1_{\beta}\otimes(X^{(\mbf{v})*}_{z}+\mbf{b}\mbf{w}^{*})&iQ^{*}\otimes1_{N}\end{pmatrix}\,\mathrm{d}Q\,\mathrm{d}^{\beta}\mbf{w}\Big].
\end{align*}
Note that when $\beta=1$ the integral in the last line is over $Q=q\in\mbb{C}$. Changing variables to the singular value decomposition $Q=U\bs\eta V^{*}$ and integrating over $U$ and $V$ we find
\begin{align*}
    F_{\beta}(z,\lambda)&=\frac{1}{2\pi^{\beta-1}N^{\beta/2}}\left(\frac{\beta N}{2\pi t}\right)^{\beta (N-1)/2+1}K_{\beta}(z)\mbb{E}_{\beta,z}\Big[\\
    &\times\int_{\mbb{F}_{\beta}^{N-1}}\frac{g_{z_{0}}\left(1+\|\mbf{w}\|^{2}\right)}{(1+\|\mbf{w}\|^{2})^{\beta (N+3)/2}}e^{-\frac{\beta N}{2t(1+\|\mbf{w}\|^{2})}\|X_{z}^{(\mbf{v})*}\mbf{w}-\mbf{b}\|^{2}}\nonumber\\
    &\times\frac{2^{\beta}N^{\beta^{2}}}{\beta!t^{\beta^{2}}}\int_{[0,\infty]^{\beta}}\Delta^{2}_{\beta}(\bs\eta^{2})\prod_{j=1}^{\beta}e^{-\frac{N}{t}\eta_{j}^{2}}\det\left[\eta_{j}^{2}(1+\mbf{w}\mbf{w}^{*})+|X_{z}^{(\mbf{v})*}+\mbf{b}\mbf{w}^{*}|^{2}\right]\eta_{j}\,\mathrm{d}\eta_{j}\,\mathrm{d}^{\beta}\mbf{w}\Big],
\end{align*}
where $\Delta_{1}=1$ and $\Delta_{2}(\bs\eta^{2})=(\eta_{1}^{2}-\eta_{2}^{2})$. The matrix in the determinant is a rank 2 perturbation of $\eta^{2}+|X_{z}^{(\mbf{v})*}|^{2}$:
\begin{align*}
    &\frac{\det\left[\eta^{2}(1+\mbf{w}\mbf{w}^{*})+|X^{(\mbf{v})*}_{z}+\mbf{b}\mbf{w}^{*}|^{2}\right]}{\det\left(\eta^{2}+|X^{(\mbf{v})*}_{z}|^{2}\right)}\\
    &=|1+\mbf{w}^{*}\wt{H}^{(\mbf{v})}_{z}(\sigma)X^{(\mbf{v})}_{z}\mbf{b}|^{2}+\sigma^{2}(1+\mbf{b}^{*}H^{(\mbf{v})}_{z}(\sigma)\mbf{b})(\mbf{w}^{*}\wt{H}^{(\mbf{v})}_{z}(\sigma)\mbf{w}).
\end{align*}
Using spherical coordinates for the integral over $\mbf{w}\in\mbb{C}^{N-1}$ and rewriting the ensuing integral over the sphere in terms of the measure $\nu_{\beta,z,s}$ defined in \eqref{eq:nu} we obtain \eqref{eq:rhoN}.
\end{proof}

The rest of this section is devoted to the asymptotic analysis of $\rho_{\beta,N}$. We need to carry out an analysis of the measure $\nu^{(\mbf{v})}_{\beta,z,s}$ similar to the one carried out for $\mu^{X}_{\beta,z}$. Comparing \eqref{eq:mu} and \eqref{eq:nu}, we see that the Gaussian weight is non-centred, being shifted by the vector $\frac{1}{\sqrt{s}}\mbf{b}$. Thus we anticipate that in place of the equation \eqref{eq:critical} we should study a shifted equation. The appropriate equation turns out to be
\begin{align}
    t\Tr{H^{(\mbf{v})}_{z}(\eta)}+\frac{1+\mbf{b}^{*}H^{(\mbf{v})}_{z}(\eta)\mbf{b}}{1+s}&=1,\qquad\eta^{2}>-(s^{(\mbf{v})}_{1}(z))^{2},\label{eq:theta_s}
\end{align}
whose (unique) solution we denote by $\eta^{(\mbf{v})}_{z,t}(s)$. Since this equation is a deformation of \eqref{eq:critical}, we expect that $\eta^{(\mbf{v})}_{z,t}(s)$ will be close to the solution $\eta^{(\mbf{v})}_{z,t}$ of \eqref{eq:critical}. Indeed, if $s\to\infty$ then \eqref{eq:theta_s} reduces to \eqref{eq:critical} (with $X^{(\mbf{v})}$ in place of $X$) and so $\eta^{(\mbf{v})}_{z,t}(s)\to\eta^{(\mbf{v})}_{z,t}$.
\begin{lemma}\label{lem:theta_s}
The equation \eqref{eq:theta_s} has a unique solution $\eta^{(\mbf{v})}_{z,t}(s)$ such that:
\begin{enumerate}[i)]
\item $\theta^{(\mbf{v})}_{z,t}(s):=(\eta^{(\mbf{v})}_{z,t}(s))^{2}$ is decreasing in $s\in[0,\infty]$ and positive in $s\in[0,cs^{(\mbf{v})}_{0}]$ for some $c>0$, where
\begin{align}
    s^{(\mbf{v})}_{0}=\frac{N^{1/2}(1+\mbf{b}^{*}H^{(\mbf{v})}_{z,t}\mbf{b})}{t\sqrt{\Tr{(H^{(\mbf{v})}_{z,t})^{2}}}};\label{eq:s0v}
\end{align}
\item there is a sufficiently small $c>0$ such that for $(\theta^{(\mbf{v})}_{z,t}+|X^{(\mbf{v})}_{z}|^{2})s<ct$ we have 
\begin{align}
    \left(\frac{s-\mbf{b}^{*}H^{(\mbf{v})}_{z}(\eta^{(\mbf{v})}_{z,t}(s))\mbf{b}}{1+s}\right)\frac{\theta^{(\mbf{v})}_{z,t}(s)}{t}&\simeq1;\label{eq:theta_sSmall}
\end{align}
\item we have
\begin{align}
    \theta^{(\mbf{v})}_{z,t}\left(s\right)&=\theta^{(\mbf{v})}_{z,t}+\frac{1}{\sqrt{N\Tr{(H^{(\mbf{v})}_{z,t})^{2}}}}\cdot\frac{s^{(\mbf{v})}_{0}}{s}+O\left(\frac{t}{N}\log N\right),\label{eq:theta_sEdge}
\end{align}
uniformly in $s\geq s^{(\mbf{v})}_{0}/\log N$.
\end{enumerate}
\end{lemma}
\begin{proof}
For fixed $s>0$, the left hand side of \eqref{eq:theta_s} is decreasing in $\eta$ and from the definition of $\eta_{z,t}^{(\mbf{v})}$ we have
\begin{align*}
    t\Tr{H^{(\mbf{v})}_{z,t}}+\frac{1+\mbf{b}^{*}H^{(\mbf{v})}_{z,t}\mbf{b}}{1+s}&=1+\frac{1+\mbf{b}^{*}H^{(\mbf{v})}_{z,t}\mbf{b}}{1+s}>1,
\end{align*}
so there is a unique solution such that $\theta^{(\mbf{v})}_{z,t}(s)>\theta^{(\mbf{v})}_{z,t}$. For fixed $\eta>0$, the left hand side of \eqref{eq:theta_s} is decreasing in $s>0$ and so $\theta^{(\mbf{v})}_{z,t}(s)$ is decreasing in $s$. 

If $t\Tr{H^{(\mbf{v})}_{z}(0)}>1$ then $\theta^{(\mbf{v})}_{z,t}(s)$ is always positive, otherwise it is positive if
\begin{align*}
    1+s&<\frac{1+\mbf{b}^{*}H^{(\mbf{v})}_{z}(0)\mbf{b}}{1-t\Tr{H^{(\mbf{v})}_{z}(0)}}.
\end{align*}

To prove \eqref{eq:theta_sSmall}, we rewrite \eqref{eq:theta_s} as
\begin{align*}
    t\Tr{H^{(\mbf{v})}_{z}(\eta^{(\mbf{v})}_{z,t}(s))}&=\frac{s-\mbf{b}^{*}H^{(\mbf{v})}_{z}(\eta^{(\mbf{v})}_{z,t}(s))\mbf{b}}{1+s}.
\end{align*}
As $s\downarrow0$, $\theta^{(\mbf{v})}_{z,t}(s)\to\infty$ and \eqref{eq:theta_s} follows from the asymptotic $H_{z}(\eta)\simeq\eta^{-2}$ as $\eta\to\infty$.

Now let $s>s^{(\mbf{v})}_{0}/\log N$ and, for $\omega>0$, define the interval 
\begin{align*}
    \Omega_{\omega}&=\theta^{(\mbf{v})}_{z,t}+\left[0,\frac{\omega t^{2}}{N^{1/2}}\log N\right].
\end{align*}
We will show that $\theta^{(\mbf{v})}_{z,t}(s)\in\Omega_{\omega}$ for sufficiently large $\omega$. Let $\eta^{2}\in\Omega_{\omega}$; then since $\|H^{(\mbf{v})}_{z,t}\|\lesssim t^{-3}$, we have
\begin{align*}
    1+\mbf{b}^{*}H^{(\mbf{v})}_{z}(\eta)\mbf{b}&\lesssim1+\mbf{b}^{*}H^{(\mbf{v})}_{z,t}\mbf{b},
\end{align*}
and
\begin{align*}
    \Tr{H^{(\mbf{v})}_{z,t}H^{(\mbf{v})}_{z}(\eta)}&\gtrsim\Tr{(H^{(\mbf{v})}_{z,t})^{2}}.
\end{align*}
Thus there are constants $C_{1},C_{2}>0$ such that
\begin{align*}
    t\Tr{H^{(\mbf{v})}_{z}(\eta)}+\frac{1+\mbf{b}^{*}H^{(\mbf{v})}_{z}(\eta)\mbf{b}}{1+s}&\leq1-C_{1}t\Tr{(H^{(\mbf{v})}_{z,t})^{2}}(\eta^{2}-\theta^{(\mbf{v})}_{z,t})+\frac{C_{2}(1+\mbf{b}^{*}H^{(\mbf{v})}_{z,t}\mbf{b})}{s}.
\end{align*}
When $s>s^{(\mbf{v})}_{0}/\log N$, we can find an $\omega>0$ such that the right hand side is less than 1 for some $\eta\in\Omega_{\omega}$ and so $\theta^{(\mbf{v})}_{z,t}(s)\in\Omega_{\omega}$. Now we observe that \eqref{eq:theta_s} is equivalent to
\begin{align*}
    \theta^{(\mbf{v})}_{z,t}(s)&=\theta^{(\mbf{v})}_{z,t}+\frac{1+\mbf{b}^{*}H^{(\mbf{v})}_{z}(\eta^{(\mbf{v})}_{z,t}(s))\mbf{b}}{t\Tr{H^{(\mbf{v})}_{z,t}H^{(\mbf{v})}_{z}(\eta^{(\mbf{v})}_{z,t}(s))}(1+s)},
\end{align*}
from which \eqref{eq:theta_sEdge} follows.
\end{proof}

Henceforth we use the shorthand $H^{(\mbf{v})}_{z,t,s}:=H^{(\mbf{v})}_{z}(\eta^{(\mbf{v})}_{z,t}(s))$. Note that $\eta^{(\mbf{v})}_{z,t}(s)$ is the maximum at any fixed $s$ of the function $\phi^{(\mbf{v})}_{z,t}:[0,\infty]^{2}\to\mbb{R}$ given by
\begin{align}
    \phi^{(\mbf{v})}_{z,t}(\eta,s)&:=\left(1-\frac{1+\mbf{b}^{*}H^{(\mbf{v})}_{z}(\eta)\mbf{b}}{1+s}\right)\frac{\eta^{2}}{t}-\Tr{\log\left(\eta^{2}+|X^{(\mbf{v})}_{z}|^{2}\right)},\label{eq:phiTilde}
\end{align}
which should be compared with $\phi_{z}$ from \eqref{eq:phi_z}. With the shorthand $\phi^{(\mbf{v})}_{z,t}(s):=\phi^{(\mbf{v})}_{z,t}(\eta^{(\mbf{v})}_{z,t}(s),s)$, we define the function $\chi^{(\mbf{v})}_{z}:[0,\infty]\to\mbb{R}$
\begin{align}
    \chi^{(\mbf{v})}_{z}(s)&:=\phi^{(\mbf{v})}_{z,t}-\phi^{(\mbf{v})}_{z,t}(s).\label{eq:chi}
\end{align}
The extra term in \eqref{eq:phiTilde} as compared with \eqref{eq:phi_z} goes to zero as $s\to\infty$, so we immediately see that $\chi^{(\mbf{v})}_{z}(s)\to0$ as $s\to\infty$. The reason for introducing $\chi^{(\mbf{v})}_{z}$ is because the normalisation $K_{\beta,z}(s)$ of the measure $\nu^{(\mbf{v})}_{\beta,z,s}$ and thus $\psi^{(\mbf{v})}_{\beta}(s)$ can be expressed in terms of $\chi^{(\mbf{v})}_{z}$ (see Lemma \ref{lem:nuDuality} below).
\begin{lemma}\label{lem:chiTilde}
The function $\chi^{(\mbf{v})}_{z}(s)$ has the following properties:
\begin{enumerate}[i)]
\item for sufficiently small $c>0$ there is a constant $C>0$ such that
\begin{align}
    \chi^{(\mbf{v})}_{z}(s)&\geq-C+\log\frac{t}{(\theta^{(\mbf{v})}_{z,t}+|X^{(\mbf{v})}_{z}|^{2})s},\quad s\leq\frac{ct}{\theta^{(\mbf{v})}_{z,t}+\|X^{(\mbf{v})}_{z}\|^{2}}.\label{eq:chiTildeSmall}
\end{align}
\item there is a constant $C>0$ such that $\chi^{(\mbf{v})}_{z}(s)$ is decreasing in $\left[0,CN^{1/2}\right]$ 
\item we have
\begin{align}
    \chi^{(\mbf{v})}_{z}\left(s\right)&=\frac{(s^{(\mbf{v})}_{0})^{2}}{2Ns^{2}}\left(1-\frac{2\delta^{(\mbf{v})}_{z,t}s}{s^{(\mbf{v})}_{0}}\right)+O\left(\frac{\log^{3}N}{N^{3/2}t}\right),\label{eq:chiTildeEdge}
\end{align}
uniformly in $s>s^{(\mbf{v})}_{0}/\log N$.
\end{enumerate}
\end{lemma}
\begin{proof}
The explicit expression for $\chi^{(\mbf{v})}_{z}(s)$ is
\begin{align*}
    &\frac{\theta^{(\mbf{v})}_{z,t}}{t}-\left(\frac{s-\mbf{b}^{*}H^{(\mbf{v})}_{z}(\eta^{(\mbf{v})}_{z,t}(s))\mbf{b}}{1+s}\right)\frac{\theta^{(\mbf{v})}_{z,t}(s)}{t}+\Tr{\log\left(1+(\theta^{(\mbf{v})}_{z,t}(s)-\theta^{(\mbf{v})}_{z,t})H^{(\mbf{v})}_{z,t}\right)}
\end{align*}
When $s\leq ct(\theta^{(\mbf{v})}_{z,t}+\|X^{(\mbf{v})}_{z}\|^{2})^{-1}$ for small $c>0$, by \eqref{eq:theta_sSmall} we have $\theta^{(\mbf{v})}_{z,t}(s)\gtrsim t/s$. Hence the sum of the first two terms above is bounded below by $-C$, while for the third we have
\begin{align*}
    \Tr{\log\left(1+(\theta^{(\mbf{v})}_{z,t}(s)-\theta^{(\mbf{v})}_{z,t})H^{(\mbf{v})}_{z,t}\right)}&\gtrsim\Tr{\log\left(\theta^{(\mbf{v})}_{z,t}(s)H^{(\mbf{v})}_{z,t}\right)}\\
    &\gtrsim\Tr{\log\frac{\theta^{(\mbf{v})}_{z,t}(s)}{\theta^{(\mbf{v})}_{z,t}+\|X^{(\mbf{v})}_{z}\|^{2}}}\\
    &\gtrsim\log\frac{t}{(\theta^{(\mbf{v})}_{z,t}+\|X^{(\mbf{v})}_{z}\|^{2})s},
\end{align*}
from which \eqref{eq:chiTildeSmall} follows.

Differentiating \eqref{eq:theta_s} we find
\begin{align}
    \partial_{s}\theta^{(\mbf{v})}_{z,t}(s)&=-\left[t\Tr{(H^{(\mbf{v})}_{z,t,s})^{2}}+\frac{\mbf{b}^{*}(H^{(\mbf{v})}_{z,t,s})^{2}\mbf{b}}{1+s}\right]^{-1}\frac{1+\mbf{b}^{*}H^{(\mbf{v})}_{z,t,s}\mbf{b}}{(1+s)^{2}}.
\end{align}
Using this we find the derivative of $\chi^{(\mbf{v})}_{z}$:
\begin{align*}
    \partial_{s}\chi^{(\mbf{v})}_{z}(s)&=\Tr{(H^{(\mbf{v})}_{z,t,s})^{2}}\theta^{(\mbf{v})}_{z,t}(s)\partial_{s}\theta^{(\mbf{v})}_{z,t}(s)\nonumber\\
    &=-\left[t\Tr{(H^{(\mbf{v})}_{z,t,s})^{2}}+\frac{\mbf{b}^{*}(H^{(\mbf{v})}_{z,t,s})^{2}\mbf{b}}{1+s}\right]^{-1}\\
    &\times\frac{\Tr{(H^{(\mbf{v})}_{z,t,s})^{2}}(1+\mbf{b}^{*}H^{(\mbf{v})}_{z,t,s}\mbf{b})\theta^{(\mbf{v})}_{z,t}(s)}{(1+s)^{2}}.
\end{align*}
By Lemma \ref{lem:theta_s}, there is a $C>0$ such that $\theta^{(\mbf{v})}_{z,t}(s)$ is positive and decreasing, and hence the derivative of $\chi^{(\mbf{v})}_{z}(s)$ is negative, in $s\in[0,CN^{1/2}]$. When $s>s^{(\mbf{v})}_{0}/\log N$, we can approximate the derivative of $\chi^{(\mbf{v})}_{z}(s)$ using the estimates in Lemma \ref{lem:theta_s}, from which \eqref{eq:chiTildeEdge} follows.
\end{proof}

We now study quadratic forms in $\mbf{u}\sim\nu_{\beta,z,s}$, which we do via the moment generating function. For $F\in\mbb{M}^{H}_{N-1}(\mbb{F}_{\beta})$, define the shifted resolvent
\begin{align}
    H^{(\mbf{v}),F}_{z}(\eta)&=\left(\eta^{2}+|X^{(\mbf{v})}_{z}|^{2}-\frac{2t(1+s)}{\beta Ns}F\right)^{-1},
\end{align}
and set $H^{(\mbf{v}),F}_{z,t,s}:=H^{(\mbf{v}),F}_{z,t,s}(\eta^{(\mbf{v})}_{z,t}(s))$. We will consider quadratic forms in the shifted vector
\begin{align}
    \hat{\mbf{u}}&=\mbf{u}-\frac{1}{\sqrt{s}}\wt{H}^{(\mbf{v})}_{z,t,s}X^{(\mbf{v})}_{z}\mbf{b}.\label{eq:uHat}
\end{align}
\begin{lemma}\label{lem:nuDuality}
Let $F\in\mbb{M}^{H}_{N}(\mbb{F}_{\beta})$ and
\begin{align}
    m_{F}(s)&=\frac{t(1+s)}{s}\Tr{H^{(\mbf{v})}_{z,t,s}F},\\
    \hat{m}_{F}(s)&=e^{-\frac{t(1+s)}{s}\Tr{H^{(\mbf{v})}_{z,t,s}F}}\det^{-\beta/2}\left(1-\frac{2t(1+s)}{\beta Ns}\sqrt{H^{(\mbf{v})}_{z,t,s}}F\sqrt{H^{(\mbf{v})}_{z,t,s}}\right).
\end{align}
Then for any $\lambda\in\mbb{R}$ such that $\frac{2|\lambda|t(1+s)}{\beta N s}\|\sqrt{\wt{H}^{(\mbf{v})}_{z,t,s}}F\sqrt{\wt{H}^{(\mbf{v})}_{z,t,s}}\|<1$ we have
\begin{align}
    \mbb{E}_{\beta,z,s}\left[e^{\lambda (\hat{\mbf{u}}^{*}F\hat{\mbf{u}}-m_{F}(s))}\right]&=\frac{e^{-\frac{\beta N}{2}\chi^{(\mbf{v})}_{z}(s)}h_{\lambda F}(s)\hat{m}_{\lambda F}(s)}{e^{-\frac{\beta N}{2}\phi^{(\mbf{v})}_{z,t}}K_{\beta,z}(s)},\label{eq:nuDuality}
\end{align}
where
\begin{align}
    h_{F}(s)&=\int_{-\infty}^{\infty}e^{\frac{i\beta Np}{2t(1+s)}\left[s-\mbf{b}^{*}X^{(\mbf{v})*}_{z}\wt{H}^{(\mbf{v})}_{z,t,s}(1+ip\wt{H}^{(\mbf{v}),F}_{z,t,s})^{-1}\wt{H}^{(\mbf{v})}_{z,t,s}X^{(\mbf{v})}_{z}\mbf{b}\right]}\det^{-\beta/2}\left[1+ipH^{(\mbf{v}),F}_{z,t,s}\right]\,\mathrm{d}p.
\end{align}
\end{lemma}
The proof is a straightforward application of the duality formula for integrals on the sphere in \cite[Lemma 3.4]{maltsev_bulk_2024} and \cite[Lemma 3.4]{osman_bulk_2024}. Specialising to the case $F=0$ we can study the normalisation $K_{\beta,z}$.
\begin{lemma}\label{lem:Kr}
We have
\begin{align}
    e^{-\frac{\beta N}{2}\phi^{(\mbf{v})}_{z,t}}K_{\beta,z}(s)&\lesssim\frac{\theta^{(\mbf{v})}_{z,t}(s)+\|X^{(\mbf{v})}_{z}\|^{2}}{N^{1/2}}e^{\frac{-\beta N}{2}\chi^{(\mbf{v})}_{z}(s)},\label{eq:KUniform}
\end{align}
uniformly in $s\geq0$, and
\begin{align}
    e^{-\frac{\beta N}{2}\phi^{(\mbf{v})}_{z,t}}K_{\beta,z}(s)&=\left[1+O\left(\frac{\log^{3}N}{N^{1/2}t}\right)\right]\sqrt{\frac{4\pi}{\beta N\Tr{(H^{(\mbf{v})}_{z,t})^{2}}}}e^{\frac{-\beta N}{2}\chi^{(\mbf{v})}_{z}(s)},\label{eq:KLarge}
\end{align}
uniformly in $s\geq s^{(\mbf{v})}_{0}/\log N$.
\end{lemma}
\begin{proof}
Setting $F=0$ in \eqref{eq:nuDuality} we obtain
\begin{align*}
    e^{-\frac{\beta N}{2}\phi^{(\mbf{v})}_{z,t}}K_{\beta,z}(s)&=e^{-\frac{\beta N}{2}\chi^{(\mbf{v})}_{z}(s)}h_{0}(s).
\end{align*}
Using the definition of $\eta^{(\mbf{v})}_{z,t}(s)$ we can rewrite $h_{0}$ as follows:
\begin{align*}
    h_{0}(s)&=\int_{-\infty}^{\infty}e^{\frac{i\beta Np}{2t(1+s)}(\theta^{(\mbf{v})}_{z,t}(s)+ip)\mbf{b}^{*}H^{(\mbf{v})}_{z,t,s}H^{(\mbf{v})}_{z}(w_{p})\mbf{b}+\frac{i\beta Np}{2}\Tr{H^{(\mbf{v})}_{z,t,s}}}\det^{-\beta/2}\left(1+ipH^{(\mbf{v})}_{z,t,s}\right)\,\mathrm{d}p.
\end{align*}
Now note that 
\begin{align*}
    &\Re\left(\frac{i\beta Np}{2t(1+s)}(\theta^{(\mbf{v})}_{z,t}(s)+ip)\mbf{b}^{*}H^{(\mbf{v})}_{z,t,s}H^{(\mbf{v})}_{z}(w_{p})\mbf{b}\right)\\
    &=-\frac{\beta Np^{2}}{2t(1+s)}\mbf{b}^{*}\frac{(H^{(\mbf{v})}_{z,t,s})^{2}(1-\theta^{(\mbf{v})}_{z,t}(s)H^{(\mbf{v})}_{z,t,s})}{1+p^{2}(H^{(\mbf{v})}_{z,t,s})^{2}}\mbf{b}\\
    &\leq 0.
\end{align*}
Therefore the integral over $|p|>p_{0}$ is bounded by
\begin{align*}
    \int_{|p|>p_{0}}\det^{-\beta/4}\left[1+p^{2}(H^{(\mbf{v})}_{z,t,s})^{2}\right]\,\mathrm{d}p.
\end{align*}
For the uniform bound in \eqref{eq:KUniform}, we observe that
\begin{align*}
    \det^{-\beta/4}\left[1+p^{2}(H^{(\mbf{v})}_{z,t,s})^{2}\right]&\leq\left[1+\left(\frac{p}{\theta^{(\mbf{v})}_{z,t}(s)+|X^{(\mbf{v})}_{z}|^{2}}\right)^{2}\right]^{-\beta N/4}
\end{align*}
and change variable $p\mapsto (\theta^{(\mbf{v})}_{z,t}(s)+|X^{(\mbf{v})}_{z}|^{2})p$.

When $s>s^{(\mbf{v})}_{0}/\log N$, we have $\eta^{(\mbf{v})}_{z,t}(s)\leq\frac{Ct^{2}\log N}{N^{1/2}}$ and so $\Tr{(H^{(\mbf{v})}_{z,t,s})^{2}}\geq ct^{-4}$. Now we use the bounds
\begin{align*}
    \det^{-\beta/2}\left[1+p^{2}(H^{(\mbf{v})}_{z,t,s})^{2}\right]&\leq\exp\left\{-CNt^{2}\cdot\frac{p^{2}/t^{6}}{1+p^{2}/t^{6}}\right\}
\end{align*}
when $|p|<C\|X_{z}\|^{2}$ and
\begin{align*}
    \det^{-\beta/2}\left[1+p^{2}(H^{(\mbf{v})}_{z,t,s})^{2}\right]&\leq e^{-CN\log|p|}
\end{align*}
when $|p|>C\|X_{z}\|^{2}$ to restrict $p$ to the region $|p|<p_{0}:=\frac{t^{2}\log N}{N^{1/2}}$. The asymptotic in \eqref{eq:KLarge} then follows by Taylor expansion of the integrand.
\end{proof}

Besides the normalisation, we also need to consider expectation values of quadratic forms in $\mbf{u}$. Rather than computing these expectation values explicitly, we use the formula in Lemma \ref{lem:nuDuality} to prove concentration.
\begin{lemma}\label{lem:nuConc}
Let $\mbf{v}\in\mc{E}_{conc}$. Then there is a $c>0$ such that
\begin{align}
    \left|\hat{\mbf{u}}^{*}\wt{H}^{(\mbf{v})}_{z}(\eta)X^{(\mbf{v})}_{z}\mbf{b}\right|&\leq\sqrt{\frac{\mbf{b}^{*}H^{(\mbf{v})}_{z,t}\mbf{b}}{Nt^{2}}}\log N,\label{eq:nuConc1}\\
    \left|\hat{\mbf{u}}^{*}\wt{H}^{(\mbf{v})}_{z}(\eta)\wt{H}^{(\mbf{v})}_{z,t,s}X^{(\mbf{v})}_{z}\mbf{b}\right|&\leq\sqrt{\frac{\mbf{b}^{*}H^{(\mbf{v})}_{z,t}\mbf{b}}{Nt^{8}}}\log N,\label{eq:nuConc2}\\
    \left|\hat{\mbf{u}}^{*}\wt{H}^{(\mbf{v})}_{z}(\eta)\hat{\mbf{u}}-t\Tr{\wt{H}^{(\mbf{v})}_{z,t,s}\wt{H}^{(\mbf{v})}_{z}(\sigma)}\right|&\leq\frac{\log N}{N^{1/2}t^{4}},\label{eq:nuConc3}
\end{align}
with probability at least $1-e^{-c\log^{2}N}$, uniformly in $0\leq\eta\leq\frac{t}{N^{1/4}}\log N$ and $s>s^{(\mbf{v})}_{0}/\log N$. 
\end{lemma}
\begin{proof}
In all the estimates below we use the assumption $s>s^{(\mbf{v})}_{0}/\log N$ to ensure that $|\theta^{(\mbf{v})}_{z,t}(s)|<\frac{t^{2}\log N}{N^{1/2}}$, as follows from Lemma \ref{lem:theta_s}. Let $\lambda\in\mbb{R}$, $F\in\mbb{M}_{N}^{H}(\mbb{F}_{\beta})$ and define
\begin{align*}
    v^{2}_{F}(s)&=\frac{t^{2}(1+s)^{2}}{Ns^{2}}\Tr{\left(\wt{H}^{(\mbf{v})}_{z,t,s}F\right)^{2}}.
\end{align*}
Assume that
\begin{align*}
    \rho_{F}(\lambda)&:=\frac{2|\lambda|t(1+s)}{\beta Ns}\left\|\sqrt{\wt{H}^{(\mbf{v})}_{z,t,s}}F\sqrt{\wt{H}^{(\mbf{v})}_{z,t,s}}\right\|<1;
\end{align*}
Using Lemma \ref{lem:nuDuality} and the bound $-x-\log(1-x)\leq \frac{x^{2}}{1-x}$ for $x<1$ we find for the moment generating function
\begin{align*}
    e^{-\lambda m_{F}(s)}\mbb{E}_{\beta,z,s}\left[e^{\lambda\hat{\mbf{u}}^{*}F\hat{\mbf{u}}}\right]&\leq\frac{e^{-\frac{\beta N}{2}\chi^{(\mbf{v})}_{z}(s)}h_{\lambda F}(s)}{e^{-\frac{\beta N}{2}\phi^{(\mbf{v})}_{z}}K_{\beta,z}(s)}\cdot \exp\left\{\frac{\lambda^{2}v^{2}_{F}(s)}{1-\rho_{F}(\lambda)}\right\}.
\end{align*}
To bound $h_{\lambda F}$ we note that
\begin{align*}
    \Re\left(-ip\mbf{b}^{*}X^{(\mbf{v})*}_{z}\wt{H}^{(\mbf{v})}_{z,t,s}(1+ip\wt{H}^{(\mbf{v}),\lambda F}_{z,t,s})^{-1}\wt{H}^{(\mbf{v})}_{z,t,s}X^{(\mbf{v})}_{z}\mbf{b}\right)&\leq0,
\end{align*}
so that
\begin{align*}
    h_{\lambda F}(s)&\leq\int_{-\infty}^{\infty}\det^{-\beta/4}\left(1+p^{2}(H^{(\mbf{v}),F}_{z,t,s})^{2}\right)\,\mathrm{d}p.
\end{align*}
Since $\rho_{F}(\lambda)<1$, we have
\begin{align*}
    \Tr{(H^{(\mbf{v}),F}_{z,t,s})^{2}}&\geq\frac{1}{(1+\rho_{F}(\lambda))^{2}}\Tr{(H^{(\mbf{v})}_{z,t,s})^{2}},
\end{align*}
which leads to the bound
\begin{align*}
    h_{\lambda F}(s)&\lesssim\frac{1}{\sqrt{N\Tr{(H^{(\mbf{v})}_{z,t,s})^{2}}}}.
\end{align*}
From the asymptotics of $K_{\beta,z}$ in \eqref{eq:KLarge}, we find
\begin{align*}
    e^{-\lambda m_{F}(s)}\mbb{E}_{\beta,z,s}\left[e^{\lambda\hat{\mbf{u}}^{*}F\hat{\mbf{u}}}\right]&\lesssim \exp\left\{\frac{\lambda^{2}v^{2}_{F}(s)}{1-\rho_{F}(\lambda)}\right\}.
\end{align*}
It remains to choose $F\in\{\wt{H}^{(\mbf{v})}_{z}(\eta),\,\wt{H}^{(\mbf{v})}_{z}(\eta)X^{(\mbf{v})}_{z}\mbf{b}\}$ and estimate $m_{F}(s)$ and $v^{2}_{F}(s)$.

Set $F=\wt{H}^{(\mbf{v})}_{z}(\eta)X^{(\mbf{v})}_{z}\mbf{b}\mbf{b}^{*}X^{(\mbf{v})*}_{z}\wt{H}^{(\mbf{v})}_{z}(\eta)$; then
\begin{align*}
    \rho_{F}(\lambda)&\lesssim\frac{|\lambda|t}{N}\mbf{b}^{*}X^{(\mbf{v})*}_{z}\wt{H}^{(\mbf{v})}_{z}(\eta)\wt{H}^{(\mbf{v})}_{z,t,s}\wt{H}^{(\mbf{v})}_{z}(\eta)X^{(\mbf{v})}_{z}\mbf{b}\\
    &\lesssim\frac{|\lambda|t}{N}\mbf{b}^{*}H^{(\mbf{v})}_{z,t}\mbf{b}\cdot\|H^{(\mbf{v})}_{z,t}\|\\
    &\lesssim\frac{|\lambda|}{Nt^{2}}\mbf{b}^{*}H^{(\mbf{v})}_{z,t}\mbf{b},
\end{align*}
where we have used the fact that $\|H^{(\mbf{v})}_{z,t}\|\gtrsim t^{3}$ for $\mbf{v}\in\mc{E}_{conc}$ (see the proof of Lemma \ref{lem:minor}). By similar computations we find
\begin{align*}
    v^{2}_{F}(s)&=\frac{t^{2}(1+s)^{2}}{N^{2}s^{2}}\left(\mbf{b}^{*}X^{(\mbf{v})*}_{z}\wt{H}^{(\mbf{v})}_{z}(\eta)\wt{H}^{(\mbf{v})}_{z,t,s}\wt{H}^{(\mbf{v})}_{z}(\eta)X^{(\mbf{v})}_{z}\mbf{b}\right)^{2}\\
    &\lesssim\left(\frac{\mbf{b}^{*}H^{(\mbf{v})}_{z,t}\mbf{b}}{Nt^{2}}\right)^{2},
\end{align*}
and
\begin{align*}
    m_{F}(s)&=\frac{t(1+s)}{Ns}\mbf{b}^{*}X^{(\mbf{v})*}_{z}\wt{H}^{(\mbf{v})}_{z}(\eta)\wt{H}^{(\mbf{v})}_{z,t,s}\wt{H}^{(\mbf{v})}_{z}(\eta)X^{(\mbf{v})}_{z}\mbf{b}\\
    &\lesssim\frac{\mbf{b}^{*}H^{(\mbf{v})}_{z,t}\mbf{b}}{Nt^{2}}.
\end{align*}
Choosing $|\lambda|=cNt^{2}/\mbf{b}^{*}H^{(\mbf{v})}_{z,t}\mbf{b}$ in Markov's inequality for sufficiently small $c>0$ we obtain \eqref{eq:nuConc1}.

Now set $F=\wt{H}^{(\mbf{v})}_{z}(\eta)\wt{H}^{(\mbf{v})}_{z,t,s}X^{(\mbf{v})}_{z}\mbf{b}\mbf{b}^{*}X^{(\mbf{v})*}_{z}\wt{H}^{(\mbf{v})}_{z,t,s}\wt{H}^{(\mbf{v})}_{z}(\eta)$; in this case we find 
\begin{align*}
    \rho_{F}(\lambda)&\lesssim\frac{|\lambda|}{Nt^{8}}\mbf{b}^{*}H^{(\mbf{v})}_{z,t}\mbf{b},\\
    v^{2}_{F}(s)&\lesssim\left(\frac{\mbf{b}^{*}H^{(\mbf{v})}_{z,t}\mbf{b}}{Nt^{8}}\right)^{2},
\end{align*}
and
\begin{align*}
    m_{F}(s)&\lesssim\frac{\mbf{b}^{*}H^{(\mbf{v})}_{z,t}\mbf{b}}{Nt^{8}}.
\end{align*}
Choosing $|\lambda|=cNt^{8}/\mbf{b}^{*}H^{(\mbf{v})}_{z,t}\mbf{b}$ in Markov's inequality for sufficiently small $c>0$ we obtain \eqref{eq:nuConc2}.

Finally, set $F=\wt{H}^{(\mbf{v})}_{z}(\eta)$; then
\begin{align*}
    \rho_{F}(\lambda)&\lesssim\frac{|\lambda|}{Nt^{5}},\\
    v^{2}_{F}(s)&\lesssim\frac{t^{2}}{N}\Tr{(\wt{H}^{(\mbf{v})}_{z,t,s})^{4}}\lesssim\frac{1}{Nt^{8}},
\end{align*}
and
\begin{align*}
    m_{F}(s)&=\frac{t(1+s)}{s}\Tr{\wt{H}^{(\mbf{v})}_{z,t,s}\wt{H}^{(\mbf{v})}_{z}(\eta)}\\
    &=\left[1+O\left(\frac{\log N}{N^{1/2}t}\right)\right]t\Tr{(\wt{H}^{(\mbf{v})}_{z,t,s})^{2}}.
\end{align*}
Using Markov's inequality with $|\lambda|=\sqrt{Nt^{8}}\log N$ we obtain \eqref{eq:nuConc3}.
\end{proof}

With this concentration of quadratic forms we can evaluate the expectation over $\nu^{(\mbf{v})}_{\beta,z,s}$ in $P^{(\mbf{v})}$.
\begin{lemma}
We have
\begin{align}
    P^{(\mbf{v})}(\eta,s)&\leq 3\left[1+(1+\mbf{b}^{*}H^{(\mbf{v})}_{z}(\eta)\mbf{b})s\right],\label{eq:PBound}
\end{align}
uniformly in $s\geq0$, and
\begin{align}
    P^{(\mbf{v})}\left(\eta,s\right)&=\left[1+O\left(\sqrt{\frac{s}{Nt}}\log N\right)\right]^{2}\cdot(1+\mbf{b}^{*}H^{(\mbf{v})}_{z,t}\mbf{b})^{2}\left(1+\frac{N^{1/2}s\eta^{2}}{t^{2}s^{(\mbf{v})}_{0}}\right),\label{eq:PEstimate}
\end{align}
uniformly in $s>s^{(\mbf{v})}_{0}/\log N$.
\end{lemma}
\begin{proof}
Recall the definition of $P^{(\mbf{v})}(\eta,s)$:
\begin{align*}
    P^{(\mbf{v})}(\eta,s)&=\mbb{E}_{\beta,z,s}\left[\left|1+\sqrt{s}\mbf{u}^{*}\wt{H}^{(\mbf{v})}_{z}(\eta)X^{(\mbf{v})}_{z}\mbf{b}\right|^{2}+s\eta^{2}(1+\mbf{b}^{*}H^{(\mbf{v})}_{z}(\eta)\mbf{b})\mbf{u}^{*}\wt{H}^{(\mbf{v})}_{z}(\eta)\mbf{u}\right].
\end{align*}
To prove the upper bound in \eqref{eq:PBound}, we use Cauchy-Schwarz in the form
\begin{align*}
    \left|\mbf{u}^{*}\wt{H}^{(\mbf{v})}_{z}(\eta)X^{(\mbf{v})}_{z}\mbf{b}\right|&\leq\sqrt{\mbf{b}^{*}X^{(\mbf{v})*}_{z}(\wt{H}^{(\mbf{v})}_{z}(\eta))^{2}X^{(\mbf{v})}_{z}\mbf{b}}\\
    &=\sqrt{\mbf{b}^{*}H^{(\mbf{v})}_{z}(\eta)(1-\eta^{2}H^{(\mbf{v})}_{z}(\eta))\mbf{b}}\\
    &\leq\sqrt{\mbf{b}^{*}H^{(\mbf{v})}_{z}(\eta)\mbf{b}}.
\end{align*}
This and the trivial bound $\eta^{2}\mbf{u}^{*}H^{(\mbf{v})}_{z}(\eta)\mbf{u}\leq1$ imply \eqref{eq:PBound}.

For the estimate in \eqref{eq:PEstimate}, we use the concentration of quadratic forms in Lemma \ref{lem:nuConc}. Using the assumption $\eta^{2}<\frac{t^{2}\log N}{N^{1/2}}$ we have
\begin{align*}
    \mbf{u}^{*}\wt{H}^{(\mbf{v})}_{z}(\eta)\mbf{u}&=\left[1+O\left(\frac{\log N}{N^{1/2}t}\right)\right]\mbf{u}^{*}\wt{H}^{(\mbf{v})}_{z,t}\mbf{u},\\
    \mbf{b}^{*}H^{(\mbf{v})}_{z}(\eta)\mbf{b}&=\left[1+O\left(\frac{\log N}{N^{1/2}t}\right)\right]\mbf{b}^{*}H^{(\mbf{v})}_{z,t}\mbf{b},
\end{align*}
and
\begin{align*}
    \mbf{b}^{*}X^{(\mbf{v})*}\wt{H}^{(\mbf{v})}_{z,t,s}\wt{H}^{(\mbf{v})}_{z}(\eta)X^{(\mbf{v})}_{z}\mbf{b}&=\left[1+O\left(\frac{\log N}{N^{1/2}t}\right)\right]\mbf{b}^{*}X^{(\mbf{v})*}_{z}(\wt{H}^{(\mbf{v})}_{z,t})^{2}X^{(\mbf{v})}_{z}\mbf{b}\\
    &=\left[1+O\left(\frac{\log N}{N^{1/2}t}\right)\right]\mbf{b}^{*}H^{(\mbf{v})}_{z,t}\left(1-\theta^{(\mbf{v})}_{z,t}H^{(\mbf{v})}_{z,t}\right)\mbf{b}\\
    &=\left[1+O\left(\frac{\log N}{N^{1/2}t}\right)\right]\mbf{b}^{*}H^{(\mbf{v})}_{z,t}\mbf{b}.
\end{align*}
On the event that \eqref{eq:nuConc1} holds we have
\begin{align*}
    \left|1+\sqrt{s}\mbf{u}^{*}\wt{H}^{(\mbf{v})}_{z}(\eta)X^{(\mbf{v})}_{z}\mbf{b}\right|^{2}&=\left|1+\frac{\sqrt{s}\hat{\mbf{u}}^{*}\wt{H}^{(\mbf{v})}_{z}(\eta)X^{(\mbf{v})}_{z}\mbf{b}}{1+\mbf{b}^{*}X^{(\mbf{v})*}_{z}\wt{H}^{(\mbf{v})}_{z,t,s}\wt{H}^{(\mbf{v})}_{z}(\eta)X^{(\mbf{v})}_{z}\mbf{b}}\right|^{2}\\
    &\times\left(1+\mbf{b}^{*}X^{(\mbf{v})*}_{z}\wt{H}^{(\mbf{v})}_{z,t,s}\wt{H}^{(\mbf{v})}_{z}(\eta)X^{(\mbf{v})}_{z}\mbf{b}\right)^{2}\\
    &=\left[1+O\left(\sqrt{\frac{s}{N^{1/2}ts^{(\mbf{v})}_{0}}}\log N\right)\right]^{2}\left(1+\mbf{b}^{*}H^{(\mbf{v})}_{z,t}\mbf{b}\right)^{2},
\end{align*}
and on the event that \eqref{eq:nuConc2} and \eqref{eq:nuConc3} hold we have
\begin{align*}
    s\mbf{u}^{*}\wt{H}^{(\mbf{v})}_{z,t}\mbf{u}&=\left[1+O\left(\frac{\log N}{N^{1/2}t}\right)\right]st\Tr{(H^{(\mbf{v})}_{z,t})^{2}},
\end{align*}
from which \eqref{eq:PEstimate} follows.
\end{proof}

Combining the previous lemma and the estimates on $\phi_{z}$ in Lemma \ref{lem:phi} we can study $I^{(\mbf{v})}_{\beta}$.
\begin{lemma}
Let $I^{(\mbf{v})}_{\beta}$ be defined as in \eqref{eq:I}. Then we have
\begin{align}
    I^{(\mbf{v})}_{\beta}(s)&\lesssim \left(\frac{N}{t^{2}\Tr{(H^{(\mbf{v})}_{z,t})^{2}}}\right)^{\beta^{2}/2}\frac{1+s}{t},\label{eq:IUniform}
\end{align}
uniformly in $s\geq0$, and
\begin{align}
    I^{(\mbf{v})}_{\beta}(s)&=\left[1+O\left(\frac{\log^{3}N}{N^{1/2}t}\right)\right](1+\mbf{b}^{*}H^{(\mbf{v})}_{z,t}\mbf{b})^{2\beta}\left(\frac{N}{t^{2}\Tr{(H^{(\mbf{v})}_{z,t})^{2}}}\right)^{\beta^{2}/2}\nonumber\\
    &\times\det\left[\int_{\delta^{(\mbf{v})}_{z,t}}^{\infty}x^{j+k-2}e^{-\frac{1}{2}x^{2}}\left(1+\frac{s}{s^{(\mbf{v})}_{0}}\left(x-\delta^{(\mbf{v})}_{z,t}\right)\right)\,\mathrm{d}x\right]_{j,k=1}^{\beta},\label{eq:ILarge}
\end{align}
uniformly in $s>s^{(\mbf{v})}_{0}/\log N$, where
\begin{align}
    \delta^{(\mbf{v})}_{z,t}&:=\sqrt{N\Tr{(H^{(\mbf{v})}_{z,t})^{2}}}\theta^{(\mbf{v})}_{z,t}.
\end{align}
\end{lemma}
\begin{proof}
Expressing the Vandermonde factor as the square of a determinant and using the Andreif identity we have
\begin{align*}
    I^{(\mbf{v})}_{\beta}(s)&=\frac{2^{\beta}N^{\beta^{2}}}{t^{\beta^{2}}}\det\left[I^{(\mbf{v})}_{\beta,jk}(s,\mbf{u})\right]_{j,k=1}^{\beta},
\end{align*}
where
\begin{align}
    I^{(\mbf{v})}_{\beta,jk}(s)&=\int_{[0,\infty]}\eta^{2(j+k)-3}e^{-N\left[\phi^{(\mbf{v})}_{z}(\eta)-\phi^{(\mbf{v})}_{z,t}\right]}P^{(\mbf{v})}(\eta,s)\,\mathrm{d}\eta.
\end{align}
By \eqref{eq:phiBoundEdge} and Lemma \ref{lem:minor} we have
\begin{align*}
    \phi^{(\mbf{v})}_{z}(\eta)-\phi^{(\mbf{v})}_{z,t}&\geq\frac{C\log^{2}N}{N},\quad\eta\geq\frac{t\log^{1/2}N}{N^{1/4}},
\end{align*}
Using this and the bound for $P^{(\mbf{v})}$ in \eqref{eq:PBound} we can bound the contribution from the tail:
\begin{align*}
    \int_{\frac{t\log^{1/2} N}{N^{1/4}}}^{\infty}e^{-N\phi^{(\mbf{v})}_{z}(\eta)}P^{(\mbf{v})}(\eta,s)\eta^{j+k-1}\,\mathrm{d}\eta&\leq (1+s)e^{-C\log^{2}N}.
\end{align*}
When $\eta<\frac{t\log^{1/2}N}{N^{1/4}}$, we use \eqref{eq:phiEstimateEdge} to approximate $\phi^{(\mbf{v})}_{z}(\eta)$:
\begin{align*}
    \phi^{(\mbf{v})}_{z}(\eta)-\phi^{(\mbf{v})}_{z,t}&=\frac{1}{2}\Tr{(H^{(\mbf{v})}_{z,t})^{2}}(\eta^{2}-\theta^{(\mbf{v})}_{z,t})^{2}+O\left(\frac{\log^{3}N}{\sqrt{Nt^{2}}}\right).
\end{align*}
Changing variable to $x=\sqrt{N\Tr{(H^{(\mbf{v})}_{z,t})^{2}}}(\eta^{2}-\theta^{(\mbf{v})}_{z,t})$ and using the bound \eqref{eq:PBound} or the estimate \eqref{eq:PEstimate} for $P^{(\mbf{v})}$, we obtain \eqref{eq:ILarge}.
\end{proof}

To finish the proof we need to remove the $\mbf{v}$ dependence from various quantities.
\begin{lemma}\label{lem:vRemoval}
Let $\mbf{v}\in\mc{E}_{conc}$. Then
\begin{align}
    1+\mbf{b}^{*}H^{(\mbf{v})}_{z,t}\mbf{b}&=\left[1+O\left(\frac{1}{Nt^{3}}\right)\right]\frac{1}{t\sigma_{z,t}}\label{eq:b*Hb},\\
    \Tr{(H^{(\mbf{v})}_{z,t})^{n}}&=\left[1+O\left(\frac{1}{N^{1/2}t}\right)\right]\Tr{H^{n}_{z,t}},\label{eq:H^n}\\
    \delta^{(\mbf{v})}_{z,t}&=\delta_{z,t}+O\left(\frac{1}{N^{1/2}t}\right).\label{eq:delta^v}
\end{align}
\end{lemma}
\begin{proof}
Using Lemma \ref{lem:b} we have
\begin{align*}
    1+\mbf{b}^{*}H^{(\mbf{v})}_{z,t}\mbf{b}&=\left(\theta^{(\mbf{v})}_{z,t}\mbf{v}^{*}H_{z,t}\mbf{v}+\frac{|\mbf{v}^{*}X_{z}H_{z}\mbf{v}|^{2}}{\mbf{v}^{*}H_{z}\mbf{v}}\right)^{-1}.
\end{align*}
Using the estimates on the quadratic forms for $\mbf{v}\in\mc{E}_{conc}$ in Lemma \ref{lem:conc} we obtain \eqref{eq:b*Hb}. The other two estimates follow from the proof of Lemma \ref{lem:minor}.
\end{proof}

We can now conclude the proof of \eqref{eq:gaussDivisibleEdge}
\begin{proof}[Proof of \eqref{eq:gaussDivisibleEdge}]
The claim in \eqref{eq:gaussDivisibleEdge} follows from the bounds
\begin{align}
    \rho_{\beta,N}(z,s)&\lesssim e^{-cN|\log s|},\quad s\leq\frac{ct}{\theta^{(\mbf{v})}_{z,t}+\|X^{(\mbf{v})}_{z}\|^{2}},\\
    \rho_{\beta,N}(z,s)&\lesssim e^{-c\log^{2}N},\quad 0\leq s\leq s_{0}/\log N,
\end{align}
and the estimate
\begin{align}
    s_{0}\rho_{\beta,N}\left(z,s_{0}s\right)&=\left[1+O\left(\frac{\log N}{\sqrt{Nt^{3}}}\right)\right]\frac{\beta}{2\pi^{\beta}s^{2\beta+1}}e^{-\frac{\beta}{2}\left(\frac{1}{2s^{2}}-\frac{\delta_{z,t}}{s}\right)}\nonumber\\
    &\times\det\int_{\delta_{z,t}}^{\infty}x^{j+k-2}e^{-\frac{1}{2}x^{2}}\left[1+s(x-\delta_{z,t})\right]\,\mathrm{d}x,
\end{align}
uniformly in $s>1/\log N$. Indeed, with these estimates we can truncate the integral over $s$ to the region $s>1/\log N$, replace $\rho_{\beta,N}$ with the above approximation, and then extend the integral back to $s\in[0,\infty]$.

Recall the expression for $\rho_{\beta,N}$:
\begin{align*}
 \rho_{\beta,N}(z,s)&=\frac{1}{4\pi^{\beta-1}N^{\beta/2}}\left(\frac{\beta N}{2\pi t}\right)^{2}e^{-\frac{\beta N}{2}\phi_{z,t}}K_{\beta}(z)\mbb{E}_{\beta,z}\left[1_{\mc{E}_{conc}}e^{\frac{\beta N}{2}(\phi_{z,t}-\phi^{(\mbf{v})}_{z,t})}\psi^{(\mbf{v})}_{\beta}(s)\right].
\end{align*}
Using \eqref{eq:chiTildeSmall}, \eqref{eq:KUniform} and \eqref{eq:IUniform}, we find
\begin{align}
    \psi^{(\mbf{v})}_{\beta}(s)&\leq e^{-CN|\log s|},\quad s\leq\frac{ct}{\theta^{(\mbf{v})}_{z,t}+\|X^{(\mbf{v})}_{z}\|^{2}}.
\end{align}
Using the fact that $\chi^{(\mbf{v})}_{z}(s)$ is decreasing in $[0,cs^{(\mbf{v})}_{0}]$ and \eqref{eq:chiTildeSmall},\eqref{eq:KUniform} and \eqref{eq:IUniform} we find
\begin{align}
    \psi^{(\mbf{v})}_{\beta}(s)&\leq e^{-C\log^{2}N},\quad 0\leq s\leq s^{(\mbf{v})}_{0}/\log N.
\end{align}
When $s\geq s^{(\mbf{v})}_{0}/\log N$, we use \eqref{eq:chiTildeEdge}, \eqref{eq:KLarge}, \eqref{eq:ILarge} and the estimates in Lemma \ref{lem:vRemoval} to approximate $\psi^{(\mbf{v})}_{z}(s)$:
\begin{align}
    \psi^{(\mbf{v})}_{\beta}(s)&=\left[1+O\left(\frac{\log N}{\sqrt{Nt^{3}}}\right)\right]\left(\frac{4\pi}{\beta N\Tr{H^{2}_{z,t}}}\right)^{1/2}\frac{1}{(t\sigma_{z,t})^{2\beta}}\left(\frac{N}{t^{2}\Tr{H^{2}_{z,t}}}\right)^{\beta^{2}/2}\nonumber\\
    &\times \exp\left\{-\frac{\beta s^{2}_{0}}{4s^{2}}\left(1-\frac{2\delta_{z,t}s}{s_{0}}\right)\right\}\det\int_{\delta_{z,t}}^{\infty}x^{j+k-2}e^{-\frac{1}{2}x^{2}}\left[1+\frac{s}{s_{0}}(x-\delta_{z,t})\right]\,\mathrm{d}x.\label{eq:psiEstimate}
\end{align}
Since the right hand side of \eqref{eq:psiEstimate} is independent of $\mbf{v}$, we can take it outside the expectation over $\mu^{X}_{\beta,z}$, leaving us with
\begin{align*}
    \mbb{E}_{\beta,z}\left[1_{\mc{E}_z}e^{\frac{\beta N}{2}\left(\phi_{z,t}-\phi^{(\mbf{v})}_{z,t}\right)}\right]&=\left[1+O\left(\frac{\log N}{\sqrt{Nt^{3}}}\right)\right]\left(t^{2}\Tr{H^{2}_{z,t}}\sigma_{z,t}\right)^{\beta/2},
\end{align*}
which follows from the definition of $\mc{E}_{conc}$. The remaining term to estimate is $e^{-\frac{\beta N}{2}\phi_{z,t}}K_{\beta}(z)$, which is done in \eqref{eq:K}:
\begin{align*}
    e^{-\frac{\beta N}{2}\phi_{z,t}}K_{\beta}(z)&=\left[1+O\left(\frac{\log N}{\sqrt{Nt^{2}}}\right)\right]\sqrt{\frac{4\pi}{\beta N\Tr{H^{2}_{z,t}}}}.
\end{align*}
Collecting all the constant prefactors we find
\begin{align*}
    &\frac{1}{2\pi^{\beta-1}N^{\beta/2}}\left(\frac{\beta N}{2\pi t}\right)^{2}\cdot\frac{4\pi}{\beta N\Tr{H^{2}_{z,t}}}\cdot(t^{2}\Tr{H^{2}_{z,t}}\sigma_{z,t})^{\beta/2}\cdot\frac{1}{(t\sigma_{z,t})^{2\beta}}\left(\frac{N}{t^{2}\Tr{H^{2}_{z,t}}}\right)^{\beta^{2}/2}\\
    &=\frac{\beta}{2\pi^{\beta}}s_{0}^{2\beta},
\end{align*}
where we have used the fact that $(\beta-1)(\beta-2)=0$.
\end{proof}

We end this section with a brief discussion of the bulk case. In this case we define
\begin{align*}
    s^{(\mbf{v})}_{0}&:=\frac{N(1+\mbf{b}^{*}H^{(\mbf{v})}_{z,t}\mbf{b})(\eta^{(\mbf{v})}_{z,t})^{2}}{t}.
\end{align*}
By similar arguments as at the edge, we deduce that $\theta^{(\mbf{v})}_{z,t}$ satisfies \eqref{eq:theta_sSmall}, is positive and decreasing in $s>0$, and
\begin{align*}
    \theta^{(\mbf{v})}_{z,t}&=(\eta^{(\mbf{v})}_{z,t})^{2}+O\left(\frac{t}{N}\log^{2}N\right),
\end{align*}
when $s>s^{(\mbf{v})}_{0}/\log^{2}N$. Using these facts we find that 
$\chi^{(\mbf{v})}_{z}(s)$ is decreasing in $s>0$ and satisfies
\begin{align*}
    \chi^{(\mbf{v})}_{z}(s)&\geq -C+\log\frac{t}{(\theta^{(\mbf{v})}_{z,t}+|X^{(\mbf{v})}_{z}|^{2})s},\quad s\leq\frac{ct}{\theta^{(\mbf{v})}_{z,t}+|X^{(\mbf{v})}_{z}|^{2}},\\
    \chi^{(\mbf{v})}_{z}(s)&=\frac{s^{(\mbf{v})}_{0}}{s}+O\left(\frac{\log^{2} N}{N^{3/2}t}\right),\quad s>s^{(\mbf{v})}_{0}/\log^{2} N.
\end{align*}
This allows us to restrict to $s>s^{(\mbf{v})}_{0}/\log^{2} N$, where we can prove concentration of quadratic forms with respect to $\nu_{\beta,z,s}$ as before and ultimately estimate $P^{(\mbf{v})}$, $I^{(\mbf{v})}_{\beta}$ and $\rho_{\beta,N}$.

\section{Proof of Theorem \ref{thm2} for Gauss-divisible Matrices}\label{sec:thm2Gauss}
Let $t\geq0$ and $M_{N}=X+\sqrt{t}Y_{N}$, where $X$ is a non-Hermitian Wigner matrix and $Y_{N}\sim Gin_{\beta}(N)$. Let $f$ be a smooth function equal to $1$ when $|z|<1$ and 0 when $|z|>2$. Define
\begin{align}
    f_{z_{0}}(z)&=f\left(\frac{N^{1/2}(z-z_{0})}{r}\right),\label{eq:f}\\
    g_{\eta}(z)&=\eta\Im\tr G_{z}(i\eta)-2,\label{eq:g}
\end{align}
where $G_{z}$ is the resolvent of the Hermitisation of $M_{N}$. We want to prove the following stronger version of Theorem \ref{thm2} for Gauss-divisible matrices.
\begin{proposition}\label{prop:leastSVGauss}
Let $\epsilon>0$, $t\geq N^{-1/3+\epsilon}$ and $z_{0}\in\sqrt{1+t}\mbb{T}_{\beta}$. Then for any $\eta>0$ and fixed $r,D>0$ we have
\begin{align}
    \mbb{E}\left[\sum_{n=1}^{N_{\beta}}f_{z_{0}}(z_{n})g_{\eta}(z_{n})\right]&\lesssim N^{3/2}\eta^{2}\left|\log N^{3/4}\eta\right|^{2-\beta}+N^{-D},
\end{align}
where $\{z_{n}\}_{n=1}^{N_{\beta}}$ are the eigenvalues of $M_{N}$ in $\mbb{F}_{\beta}$.
\end{proposition}
The fact that this is stronger follows because on the event $s_{2}(z_{n})<\eta$ we have $g_{\eta}(z_{n})>1$ and so by a union bound and Markov's inequality we obtain
\begin{align}
    P\left(\min_{N^{1/2}|z_{n}-z_{0}|<r}s_{2}(z_{n})<\eta\right)&\leq\mbb{E}\left[\sum_{n=1}^{N_{\beta}}f_{z_{0}}(z_{n})g_{\eta}(z)\right].
\end{align}
\begin{proof}
Using the deterministic bound
\begin{align}
    \sum_{n=1}^{N_{\beta}}f_{z_{0}}(z_{n})g_{\eta}(z_{n})&\leq N^{2},
\end{align}
we can condition on the event that $X\in\mc{E}_{edge}(t)$, where $\mc{E}_{edge}$ is defined in Definition \ref{def:Eedge}, since this event has probability $1-N^{-D}$ for any $D>0$ as shown in Section \ref{sec:gap}. Thus we need to prove that on this event we have
\begin{align}
    \mbb{E}_{N}\left[\sum_{n=1}^{N_{\beta}}f_{z_{0}}(z_{n})g_{\eta}(z_{n})\right]&\lesssim N^{3/2}\eta^{2}\left|\log N^{3/4}\eta\right|^{2-\beta},
\end{align}
where we recall that $\mbb{E}_{N}$ is the expectation with respect to $Y_{N}\sim Gin_{\beta}(N)$. Using \eqref{eq:partialSchur} we have
\begin{align}
    \mbb{E}_{N}\left[\sum_{n=1}^{N_{\beta}}f_{z_{0}}(z_{n})g_{\eta}(z_{n})\right]&=\frac{\beta N}{2\pi^{\beta}t}\int_{\mbb{F}_{\beta}}f_{z_{0}}(z)K_{\beta}(z)F_{\beta}(z)\,\mathrm{d}^{\beta}z,
\end{align}
where
\begin{align}
    F_{\beta}(z)&:=\eta^{2}\mbb{E}_{\beta,z}\left[\mbb{E}_{\mbf{w}}\left[\tr(|M_{N-1}|^{2}+\mbf{w}\mbf{w}^{*}+\eta^{2})^{-1}|\det M_{N-1}|^{\beta}\right]\right].
\end{align}
At this point we consider real and complex matrices separately.

\paragraph{Complex matrices:} we use the bound
\begin{align*}
    \tr(|M_{N-1}|^{2}+\mbf{w}\mbf{w}^{*}+\eta^{2})^{-1}&\leq\tr|M_{N-1}|^{-2},
\end{align*}
and the identity
\begin{align*}
    \tr|M_{N-1}|^{-2}|\det M_{N-1}|^{2}&=\lim_{\eta'\to0}\frac{\partial}{\partial\eta'^{2}}\det\left(|M_{N-1}|^{2}+\eta'^{2}\right)
\end{align*}
to obtain (see the derivation of \cite[Eq. (6.6)]{osman_least_2024})
\begin{align*}
    F_{2}(z)&\leq\frac{2N^{3}\eta^{2}}{t^{3}}\int_{0}^{\infty}e^{-N\phi^{(\mbf{v})}_{z}(\sigma)}\sigma^{3}\,\mathrm{d}\sigma.
\end{align*}
Using \eqref{eq:phiBoundEdge} and \eqref{eq:phiEstimateEdge} we have
\begin{align}
    \int_{0}^{\infty}e^{-N\phi^{(\mbf{v})}_{z}(\sigma)}\sigma^{3}\,\mathrm{d}\sigma&=\left[1+O\left(\frac{\log^{3}N}{\sqrt{Nt^{2}}}\right)\right]\frac{e^{-N\phi^{(\mbf{v})}_{z}}}{2}\int_{\theta_{z,t}}^{\infty}e^{-\frac{1}{2}N\Tr{H^{2}_{z}}x^{2}}(x+\theta_{z,t})\,\mathrm{d}x\nonumber\\
    &\leq \frac{Ct^{4}}{N}e^{-N\phi^{(\mbf{v})}_{z}}.
\end{align}
By \eqref{eq:K} we have
\begin{align}
    K_{2}(z)&=\left[1+O\left(\frac{\log^{3}N}{\sqrt{Nt^{2}}}\right)\right]\sqrt{\frac{2\pi}{N\Tr{H^{2}_{z}}}}e^{N\phi_{z}},
\end{align}
and by \eqref{eq:minorEdge} we have
\begin{align}
    \mbb{E}_{2,z}\left[e^{N(\phi_{z}-\phi^{(\mbf{v})}_{z})}\right]&=\left[1+O\left(\frac{\log N}{\sqrt{Nt^{3}}}\right)\right]t^{2}\Tr{H^{2}_{z}}\sigma_{z,t}.
\end{align}
Putting everything together we find
\begin{align}
    \mbb{E}\left[\sum_{n=1}^{N}f_{z_{0}}(z_{n})g_{\eta}(z_{n})\right]&\lesssim N^{3/2}\eta^{2}\cdot N\|g\|_{1}\lesssim N^{3/2}\eta^{2}.
\end{align}

\paragraph{Real matrices:} we use the bound
\begin{align*}
    |\det M_{N-1}|&\leq\det^{1/2}\left(|M_{N-1}|^{2}+\eta^{2}\right),
\end{align*}
and the identity
\begin{align*}
    \tr(|M_{N-1}|^{2}+\eta^{2})^{-1}\det^{1/2}\left(|M_{N-1}|^{2}+\eta^{2}\right)&=-2\det(|M_{N-1}|^{2}+\eta^{2})\frac{\partial}{\partial \eta^{2}}\det^{-1/2}\left(|M_{N-1}|^{2}+\eta^{2}\right),
\end{align*}
to obtain (see \cite[Eq. (6.23)]{osman_least_2024})
\begin{align*}
    F_{1}(z)&\leq\frac{N^{2}\eta^{2}}{2\pi t^{2}}e^{-\frac{N}{2}\phi^{(\mbf{v})}_{z,t}}\int_{0}^{\infty}\frac{s}{(1+s)^{3}}e^{-\frac{N\eta^{2}}{t}s}\psi(s)ds,
\end{align*}
where
\begin{align*}
    \psi(s)&=e^{-\frac{N}{2}\phi^{(\mbf{v})}_{z,t}}K_{1,z}(s)I(s),
\end{align*}
and
\begin{align*}
    I(s)&=\frac{2N}{t}e^{-\frac{N}{t}\eta^{2}}\int_{0}^{\infty}e^{-N\left(\phi^{(\mbf{v})}_{z}(\sigma)-\phi^{(\mbf{v})}_{z,t}\right)}\left[(1+s\sigma^{2}\mbb{E}_{s}(\mbf{u}^{T}H^{(\mbf{v})}_{z}(\sigma)\mbf{u}))I_{0}\left(\frac{2N\eta\sigma}{t}\right)\right.\\
    &\left.+\eta s(1+s)\sigma\mbb{E}_{s}(\mbf{u}^{T}H^{(\mbf{v})}_{z}(\sigma)\mbf{u})I_{1}\left(\frac{2N\eta\sigma}{t}\right)\right]\sigma \,\mathrm{d}\sigma.
\end{align*}
Here $I_{m}(x)$ is the modified Bessel function of the first kind and the expectation $\mbb{E}_{\nu}$ is with respect to
\begin{align}
    d\nu_{s}(\mbf{u})&=\frac{1}{K_{1,z}(s)}\left(\frac{Ns}{2\pi t(1+s)}\right)^{(N-1)/2-1}e^{-\frac{Ns}{2t(1+s)}\|X^{(\mbf{v})}_{z}\mbf{u}\|^{2}}\,\mathrm{d}_{H}\mbf{u}.
\end{align}

Using \eqref{eq:phiBoundEdge}, we deduce that the integral is dominated by the region $\sigma\lesssim\frac{t\log N}{N^{1/4}}$. Now observe that $\nu_{s}$ is the same as $\nu_{\beta,z,s}$  from \eqref{eq:nu} but with $\mbf{b}=0$. We can repeat the same analysis as in the previous section to deduce that 
\begin{align*}
    e^{-\frac{N}{2}\phi^{(\mbf{v})}_{z,t}}K_{1,z}(s)&\leq e^{-cN|\log s|},\quad s\leq\frac{ct}{\theta_{z,t}+\|X^{(\mbf{v})}_{z}\|^{2}},\\
    e^{-\frac{N}{2}\phi^{(\mbf{v})}_{z,t}}K_{1,z}(s)&\leq e^{-c\log^{2}N},\quad s\lesssim N^{1/2}t/\log N,
\end{align*}
and, for $s\gtrsim N^{1/2}t/\log N$,
\begin{align*}
    e^{-\frac{N}{2}\phi^{(\mbf{v})}_{z,t}}K_{1,z}(s)&=\left[1+O\left(\frac{\log N}{\sqrt{Nt^{2}}}\right)\right]\sqrt{\frac{4\pi}{N\Tr{(H^{(\mbf{v})}_{z,t})^{2}}}}\exp\left\{-\frac{1}{4}\left(\frac{s^{(\mbf{v})}_{0}}{s}\right)^{2}\left(1-\frac{2\delta^{(\mbf{v})}_{z,t}s}{s^{(\mbf{v})}_{0}}\right)\right\},
\end{align*}
where we now have
\begin{align*}
    s^{(\mbf{v})}_{0}&=\frac{N^{1/2}}{t\Tr{(H^{(\mbf{v})}_{z,t})^{2}}^{1/2}}=O(N^{1/2}t).
\end{align*}
Moreover, when $s\gtrsim N^{1/2}t/\log N$, we have concentration of quadratic forms and thus
\begin{align*}
    \mbb{E}_{s}\mbf{u}^{T}H^{(\mbf{v})}_{z}(\sigma)\mbf{u}&\lesssim t\Tr{H^{(\mbf{v})}_{z}}\lesssim\frac{1}{t^{3}}.
\end{align*}
Combining these estimates we find
\begin{align*}
    F_{1}(z)&\lesssim N^{3/2}\eta^{2}e^{-\frac{N}{2}\phi^{(\mbf{v})}_{z,t}}\int_{\frac{N^{1/2}t}{\log N}}^{\infty}\frac{1}{s^{2}}\left(1+s+N^{3/2}\eta^{2}s^{2}\right)e^{-\frac{N\eta^{2}}{t}s-\frac{cNt^{2}}{s}}ds\\
    &\lesssim N^{3/2}\eta^{2}|\log N^{3/4}\eta|\cdot e^{-\frac{N}{2}\phi^{(\mbf{v})}_{z,t}}.
\end{align*}
Since
\begin{align*}
    K_{1}(z)&\lesssim\frac{t^{2}}{N^{1/2}}e^{\frac{N}{2}\phi_{z,t}},
\end{align*}
and
\begin{align*}
    \mbb{E}_{1,z}\left[e^{\frac{N}{2}\left(\phi_{z,t}-\phi^{(\mbf{v})}_{z,t}\right)}\right]&\lesssim\frac{\sigma_{z,t}^{1/2}}{t},
\end{align*}
we find
\begin{align*}
    \mbb{E}\left[\sum_{n=1}^{N_{1}}f_{z_{0}}(z_{n})g_{\eta}(z_{n})\right]&\lesssim N^{3/2}\eta^{2}|\log N^{3/4}\eta|\cdot N^{1/2}\|g_{z_{0}}\|_{1}\\
    &\lesssim N^{3/2}\eta^{2}|\log N^{3/4}\eta|.
\end{align*}
\end{proof}

\section{Overlap of Singular Vectors}\label{sec:svOverlap}
In this section we are concerned with obtaining an upper bound on the inner product between left and right singular vectors of $X-z$ when $X$ is a non-Hermitian Wigner matrix. Let us first recall some definitions. The Hermitisation $W_{z}$ is defined by
\begin{align*}
    W_{z}&=\begin{pmatrix}0&X-z\\X^{*}-\bar{z}&0\end{pmatrix},
\end{align*}
with eigenvalues $\lambda_{n}$ and eigenvectors $\mbf{w}_{n}$ for $n=-N,...,-1,1,...,N$. The resolvent is denoted by $G_{z}(w):=(W_{z}-w)^{-1}$ and the deterministic approximation is 
\begin{align}
    M_{z}(w)&=\begin{pmatrix}m_{z}(w)&-zu_{z}(w)\\-\bar{z}u_{z}(w)&m_{z}(w)\end{pmatrix},
\end{align}
where $m_{z}$ is the unique solution of
\begin{align}
    -\frac{1}{m_{z}(w)}&=m_{z}(w)+w-\frac{|z|^{2}}{m_{z}(w)+w},\quad \Im w\Im m_{z}(w)>0\label{eq:cubic}
\end{align}
and
\begin{align}
    u_{z}(w)&=\frac{m_{z}(w)}{m_{z}(w)+w}.\label{eq:u}
\end{align}
Since most bounds will be expressed in terms of the imaginary part of $m_{z}(w)$, we define $m_{z}(w)=:\sigma_{z}(w)+i\rho_{z}(w)$. For $|z|\leq1$ we have
\begin{align}
    \rho_{z}(w)&\simeq (1-|z|^{2})^{1/2}+|w|^{1/3}.\label{eq:rhoAsymp1}
\end{align}
For $|z|>1$ there is a gap $[-\Delta/2,\Delta/2]$ in the support of $\rho_{z}(w)$ with $\Delta\simeq(|z|-1)^{3/2}$ and we have
\begin{align}
    \rho_{z}(w)&\simeq\begin{cases}
    (|\kappa|+\eta)^{1/2}(\Delta+|\kappa|+\eta)^{-1/6}&\quad \kappa\in[0,c]\\
    \frac{\eta}{(\Delta+|\kappa|+\eta)^{1/6}(|\kappa|+\eta)^{1/2}}&\quad \kappa\in[-\Delta/2,0]
    \end{cases},\label{eq:rhoAsymp2}
\end{align}
where $\kappa:=E-\Delta/2$. These asymptotics hold for $|w|<c$ for some small $c>0$ and can be found in \cite[Eq. (3.6)]{cipolloni_universality_2024}. The quantiles $\gamma_{n},\,|n|\in[N]$ are defined by
\begin{align}
    \frac{1}{\pi}\int_{0}^{\gamma_{n}}\rho_{z}(\lambda)\,\mathrm{d}\lambda&=\frac{n}{2N},\quad n\in[N],
\end{align}
and $\gamma_{-n}=-\gamma_{n}$.

We have the following averaged and isotropic single resolvent laws.
\begin{proposition}[Theorem 3.1 in \cite{cipolloni_precise_2024}]\label{prop:singleLL}
There are (small) constants $\tau,\tau'>0$ such that for $\big||z|-1\big|<\tau$ and $|\Re w|<\tau'$ we have
\begin{align}
    \left|\Tr{G_{z}(w)-M_{z}(w)}\right|&\prec\frac{1}{N|\Im w|},
\end{align}
and, for any $\mbf{x},\mbf{y}\in S^{2N-1}$,
\begin{align}
    \left|\mbf{x}^{*}(G_{z}(w)-M_{z}(w))\mbf{y}\right|&\prec\sqrt{\frac{\rho_{z}(w)}{N|\Im w|}}+\frac{1}{N|\Im w|}.
\end{align}
\end{proposition}
A standard corollary of the averaged local law is rigidity \cite[Corollary 3.2]{cipolloni_precise_2024}:
\begin{align}
    |\lambda_{n}-\gamma_{n}|&\prec\max\left\{\frac{1}{N^{3/4}n^{1/4}},\frac{\Delta^{1/9}}{N^{2/3}n^{1/3}}\right\},\quad|n|\leq cN,\label{eq:rigidity}
\end{align}
for some $c>0$.

Let 
\begin{align}
    M_{z}(w_{1},B_{1},...,B_{k-1},w_{k})
\end{align}
denote the deterministic approximation to resolvent chain
\begin{align*}
    G_{z}(w_{1}) B_{1}\cdots G_{z}(w_{k-1})B_{k-1}G_{z}(w_{k}).
\end{align*}
A recursive definition of $M_{z}$ can be found in \cite[Definition 4.1]{cipolloni_optimal_2024}. If $\Im G_{z}(w_{j})$ appears in a resolvent chain we denote its corresponding argument in $M_{z}$ by $\wh{w}_{j}$, e.g. $M_{z}(\wh{w}_{1},F,\wh{w}_{2})$ is the deterministic approximation to $\Im G_{z}(w_{1})F\Im G_{z}(w_{2})$.

Before stating the relevant bounds, we define the domain
\begin{align}
    \Omega_{K,\xi,\epsilon}&:=\left\{(z,w_{1},w_{2})\in\mbb{C}^{3}:\big||z|-1\big|<KN^{-1/2},\,N\eta_{j}\rho_{j}>N^{\xi},\,|w_{j}|<\epsilon\right\}.
\end{align}
Here and below we use the shorthand $m_{j}:=m_{z}(w_{j})$ and $u_{j}:=u_{z}(w_{j})$. Since the combination $\eta_{j}\rho_{j}$ will play a prominent role, we define
\begin{align}
    l_{j}&:=\eta_{j}\rho_{j},\label{eq:l}
\end{align}
The upper bounds will be stated in terms of the quantities
\begin{align}
    \phi^{av}_{2}(w_{1},w_{2})&:=\frac{|m_{1}m_{2}|}{\phi(w_{1},w_{2})},\label{eq:phi2av}\\
    \phi^{iso}_{1}(w_{1},w_{2})&:=\frac{|m_{1}|\vee |m_{2}|}{\phi(w_{1},w_{2})},\label{eq:phi1iso}\\
    \phi^{iso}_{2}(w_{1},w_{2})&:=\frac{|m_{1}m_{2}|}{|\eta_{1}|\phi(w_{1},w_{2})},\label{eq:phi2iso}
\end{align}
where
\begin{align}
    \phi(w_{1},w_{2})&:=1-|z|^{2}|\Re(u_{1}u_{2})|-|\Re(m_{1}m_{2})|.\label{eq:phi}
\end{align}
We define the $2N\times2N$ block matrices
\begin{align}
    F&:=\begin{pmatrix}0&1_{N}\\0&0\end{pmatrix},\quad E_{\pm}=\begin{pmatrix}1_{N}&0\\0&-1_{N}\end{pmatrix},
\end{align}
and, for $i,j\in[N]$,
\begin{align}
    \Delta_{ij}&=\mbf{e}_{i}\mbf{e}_{N+j}^{*}+\mbf{e}_{N+j}\mbf{e}_{i}^{*},\label{eq:Delta}
\end{align}
where $\mbf{e}_{i}\in\mbb{C}^{2N}$ is the $i$-th standard coordinate vector. The upper bounds on the deterministic approximation are contained in the following lemma, whose proof is deferred to Section \ref{sec:deterministicBounds}.
\begin{lemma}\label{lem:MBounds}
Let $z\in\mbb{C}$ and $w_{1},w_{2}\in\mbb{C}\setminus\mbb{R}$. Then
\begin{align}
    \left|\Tr{M_{z}(w_{1},F,w_{2})F^{*}}\right|&\lesssim\phi^{av}_{2}(w_{1},w_{2}),\label{eq:M12FF*}\\
    \|M_{z}(w_{1},F,w_{2})\|&\lesssim\phi^{iso}_{1}(w_{1},w_{2}),\label{eq:M12F}\\
    \|M_{z}(w_{1},F,\wh{w}_{2},F^{*},\bar{w}_{1})\|&\lesssim\phi^{iso}_{2}(w_{1},w_{2}).\label{eq:M121FF*}
\end{align}
Moreover, we have
\begin{align}
    \phi(w_{1},w_{2})&\gtrsim\frac{\eta_{1}}{\rho_{1}}+\frac{\eta_{2}}{\rho_{2}}+(|m_{1}|-|m_{2}|)^{2}.\label{eq:phiBound}
\end{align}
\end{lemma}
We will often use the simplified bound
\begin{align}
    \frac{1}{\phi(w_{1},w_{2})}&\lesssim\frac{\rho_{1}}{\eta_{1}}\wedge\frac{\rho_{2}}{\eta_{2}},\label{eq:phiBoundSimplified}
\end{align}
which follows directly from \eqref{eq:phiBound}.

The bound on the overlap of left and right singular vectors follows from an extension of the local law in \cite[Theorem 3.5]{cipolloni_universality_2024} to a small neighbourhood of the real axis.
\begin{proposition}\label{prop:ll}
Let $\xi>0$, $z\in\mbb{C}$ and $w_{j}=E_{j}+i\eta_{j}$ be such that $\big||z|-1\big|\lesssim N^{-1/2}$ and $l_{j}:=|\eta_{j}|\rho_{j}>N^{\xi-1}$, where $\rho_{j}:=\Im m_{z}(E_{j}+i\eta_{j})$. Then there is a $c>0$ such that for $|w_{j}|<c$ we have the averaged laws
\begin{align}
    \left|\Tr{G_{z}(w_{1})FG_{z}(w_{2})-M_{z}(w_{1},F,w_{2})}\right|&\prec \frac{\phi^{av}_{1}(w_{1},w_{2})}{(N(l_{1}\wedge l_{2}))^{1/4}},\label{eq:avLL1}\\
    \left|\Tr{(\Im G_{z}(w_{1})F\Im G_{z}(w_{2})-M_{z}(\wh{w}_{1},F,\wh{w}_{2}))F^{*}}\right|&\prec \frac{\phi^{av}_{2}(w_{1},w_{2})}{(N(l_{1}\wedge l_{2}))^{1/4}},\label{eq:avLL2}
\end{align}
and the isotropic laws
\begin{align}
    \left|\mbf{e}_{\mu}^{*}(G_{z}(w_{1})FG_{z}(w_{2})-M_{z}(w_{1},F,w_{2}))\mbf{e}_{\nu}\right|&\prec \frac{\phi^{iso}_{1}(w_{1},w_{2})}{(N(l_{1}\wedge l_{2}))^{1/4}},\label{eq:isoLL1}\\
    \mbf{e}_{\mu}^{*}G_{z}(w_{1})F\Im G_{z}(w_{2})F^{*}G^{*}_{z}(w_{1})\mbf{e}_{\mu}&\prec \phi^{iso}_{2}(w_{1},w_{2}).\label{eq:isoLL2}
\end{align}
\end{proposition}
The error in Proposition \ref{prop:ll} is not optimal but we are mainly interested in upper bounds. Using this we can prove Theorem \ref{thm3}.
\begin{proof}[Proof of Theorem \ref{thm3}]
Let $\xi>0$ and $\eta_{n}=N^{-1+\xi}|\gamma_{n}|^{-1/3}$ and $\eta_{m}=N^{-1+\xi}|\gamma_{m}|^{-1/3}$. Since $\rho_{n}:=\Im m_{z}(\gamma_{n}+i\eta_{n})\simeq(|\gamma_{n}|+\eta)^{1/3}$, we have $N^{\xi}\lesssim N\eta_{n}\rho_{n}\lesssim N^{4\xi/3}$ for all $n\geq 1$. From the spectral decomposition of $G_{z}(w)$ we have
\begin{align*}
    \frac{\eta_{n}\eta_{m}|\mbf{w}_{n}^{*}F\mbf{w}_{m}|^{2}}{((\lambda_{n}-\gamma_{n})^{2}+\eta^{2}_{n})((\lambda_{m}-\gamma_{m})^{2}+\eta_{m}^{2})}&\leq N\Tr{\Im G_{z}(w_{n})F\Im G_{z}(w_{m})F^{*}},
\end{align*}
where $w_{n}=E_{n}+i\eta_{n}$. Using the rigidity of singular values $|\lambda_{n}-\gamma_{n}|\prec\frac{1}{N^{3/4}|n|^{1/4}}$ and Proposition \ref{prop:ll} we find
\begin{align*}
    |\mbf{w}_{n}^{*}F\mbf{w}_{m}|^{2}&\prec\frac{N\eta_{n}\eta_{m}\rho_{n}\rho_{m}}{\phi(w_{n},w_{m})}\\
    &\lesssim\frac{N^{4\xi/3}}{N\phi(w_{n},w_{m})}.
\end{align*}
Recall the bound for $\phi$ in \eqref{eq:phiBound}. When $|n|\wedge|m|\geq c|n|\vee|m|$, we use the simplified version in \eqref{eq:phiBoundSimplified} to obtain
\begin{align*}
    |\mbf{w}_{n}^{*}F\mbf{w}_{m}|^{2}&\prec\frac{N^{4\xi/3}}{N}\left(\frac{\rho_{n}}{\eta_{n}}\wedge\frac{\rho_{m}}{\eta_{m}}\right)\\
    &\lesssim N^{\xi/3}(\rho_{n}\wedge\rho_{m})^{2}.
\end{align*}
Since $|m_{z}(w)|\simeq\rho_{z}(w)$ for $\big||z|-1\big|\lesssim N^{-1/2}$ (see Lemma \ref{lem:interchange} below), when $|n|\wedge|m|\leq c|n|\vee|m|$ for sufficiently small $c>0$ we have $\big||m_{z}(w_{n})|-|m_{z}(w_{m})|\big|\simeq\rho_{n}\vee\rho_{m}$ and hence
\begin{align*}
    |\mbf{w}_{n}^{*}F\mbf{w}_{m}|^{2}&\prec\frac{N^{4\xi/3}}{N(\rho_{n}\vee\rho_{m})^{2}}.
\end{align*}
Now we use the asymptotics $\rho_{n}\simeq(|\gamma_{n}|+\eta_{n})^{1/3}$, $|\gamma_{n}|\simeq\bigl(\frac{|n|}{N}\bigr)^{3/4}$ and the fact that $\xi$ is arbitrary to conclude.
\end{proof}

To prove the local law we follow the dynamical method of \cite{cipolloni_universality_2024}, which consists of two parts. First we prove the analogous local law for matrices with a Gaussian component by considering the evolution $X(t)=e^{-t/2}X+\sqrt{1-e^{-t}}Y$ and allowing the parameters $z$ and $w_{j}$ to evolve simultaneously according to the ``characteristic flow". Then we argue that the Gaussian component can be removed by a Green's function comparison theorem (GFT). Since we want to capture the decay of traces of resolvents in the real parts of their arguments, there are several occasions in which we cannot afford to use Cauchy-Schwarz and the Ward identity, since this necessarily entails a factor of $1/|\Im w|$. To circumvent this, we need to first prove entrywise local laws and combine these with fluctuation averaging (see Lemma \ref{lem:fluctuationAveraging}). We write the details for the most involved case of real matrices.

\subsection{Characteristic Flow}
Let $B_{ij}(t),\,i,j=1,...,N$ be iid standard Brownian motions and consider the matrix stochastic differential equation (SDE)
\begin{align}
    \mathrm{d}X(t)&=-\frac{1}{2}X(t)\mathrm{d}t+\frac{1}{\sqrt{N}}\mathrm{d}B(t).
\end{align}
Let 
\begin{align}
    \Lambda_{t}&:=\begin{pmatrix}w_{t}&z_{t}\\\bar{z}_{t}&w_{t}\end{pmatrix}
\end{align}
solve
\begin{align}
    \frac{\mathrm{d}\Lambda_{t}}{\mathrm{d}t}&=-S[M_{t}]-\frac{1}{2}\Lambda_{t},\label{eq:dLambda}
\end{align}
with initial condition $w_{0},z_{0}\in\mbb{C}$, where
\begin{align}
    M_{t}&:=\begin{pmatrix}m_{t}&-z_{t}u_{t}\\-\bar{z}_{t}u_{t}&m_{t}\end{pmatrix}:=M_{z_{t}}(w_{t}).
\end{align}
This is equivalent to the system
\begin{align}
    \frac{\mathrm{d}w_{t}}{\mathrm{d}t}&=-\frac{1}{2}m_{t}-w_{t},\label{eq:dw}\\
    \frac{\mathrm{d}z_{t}}{\mathrm{d}t}&=-\frac{1}{2}z_{t}.\label{eq:dz}
\end{align}
By \cite[Lemma 6.5]{cipolloni_mesoscopic_2024}, we have
\begin{align}
    \mathrm{d}M_{t}&=\frac{1}{2}M_{t}\mathrm{d}t.
\end{align}
In particular, $\mathrm{d}m_{t}=\frac{1}{t}m_{t}\mathrm{d}t$. Taking the derivative of \eqref{eq:dw} we find
\begin{align*}
    \frac{\mathrm{d}^{2}w_{t}}{\mathrm{d}t^{2}}&=-\frac{\mathrm{d}m_{t}}{\mathrm{d}t}-\frac{1}{2}\frac{\mathrm{d}w_{t}}{\mathrm{d}t}\\
    &=-\frac{1}{2}m_{t}+\frac{1}{2}(m_{t}+w_{t})\\
    &=\frac{1}{4}w_{t}.
\end{align*}
Thus
\begin{align}
    w_{t}&=w_{0}e^{-t/2}-m_{0}(e^{t/2}-e^{-t/2}),\\
    z_{t}&=z_{0}e^{-t/2}.
\end{align}
The explicit solution allows us to reverse the dynamics:
\begin{align}
    w_{0}&=\left(w_{t}+m_{t}(1-e^{-t})\right)e^{t/2},\\
    z_{0}&=z_{t}e^{t/2}.
\end{align}
We denote the map $\mbb{C}^{3}\ni(z,w_{1},w_{2})\mapsto(z_{t},w_{1,t},w_{2,t})$ by $S_{t}$ and its inverse by $S^{-1}_{t}$. For any $t>0$, we define the time-reversed domain
\begin{align}
    \Omega_{K,\xi,\epsilon}(t)&:=S^{-1}_{t}(\Omega_{K,\xi,\epsilon}),
\end{align}
and the time-evolved quantities
\begin{align}
    \phi^{av/iso}_{j,t}(w_{1},w_{2})&:=\phi^{av/iso}_{j}(w_{1,t},w_{2,t}),\\
    \phi_{t}(w_{1},w_{2})&:=\phi(w_{1,t},w_{2,t}).
\end{align}

The important properties of this flow are summarised below.
\begin{lemma}
\begin{enumerate}[i)]
The flow in \eqref{eq:dLambda} satsifies:
\item $|z_{t}|$, $|\eta_{t}|$, $\frac{|\eta_{t}|}{\rho_{t}}$, $|\eta_{t}|\rho_{t}$ and $1/\phi_{t}$ are decreasing in $t\geq0$;
\item we have
\begin{align}
    \int_{s}^{t}\frac{\rho_{r}}{\eta_{r}}\,\mathrm{d}r&=\log\left(\frac{\eta_{s}}{\eta_{t}}\right)-\frac{t-s}{2};\label{eq:int1}
\end{align}
\item for any $\alpha>1$, there is a $C_{\alpha}>0$ such that
\begin{align}
    \int_{s}^{t}\frac{1}{|\eta_{r}|^{\alpha}}\,\mathrm{d}r&\leq\frac{C_{\alpha}}{|\eta_{t}|^{\alpha-1}\rho_{t}}.\label{eq:int2}
\end{align}
\end{enumerate}
\end{lemma}
Note that the right hand side of \eqref{eq:int1} is at most $O(\log N)$. We will routinely ignore factors of $\log N$ in various bounds since they are irrelevant to the concept of stochastic domination.
\begin{proof}
All of the statements are contained in \cite[Lemma 5.1]{cipolloni_universality_2024} except for the statement that $1/\phi_{t}$ is decreasing. This follows immediately from the facts that $m_{t}=e^{t/2}m_{0}$ and $u_{t}=e^{t}u_{0}$.
\end{proof}
The following lemma allows us to interchange $m_{j}$ with $\rho_{j}$, which is mostly a notational convenience.
\begin{lemma}\label{lem:interchange}
Let $t\geq0$, $(z_{-t},w_{1,-t},w_{2,-t})\in\Omega_{K,\xi,\epsilon}(t)$ and $m_{z_{-t}}(w_{j,-t})=:\sigma_{j,-t}+i\rho_{j,-t}$. Then for sufficiently small $\epsilon>0$ we have
\begin{align}
    \frac{\sigma_{j,-t}}{\rho_{j,-t}}&\lesssim1.
\end{align}
\end{lemma}
\begin{proof}
Let $(z,w_{1},w_{2})\in\Omega_{K,\xi,\epsilon}$. The bounds $|\Re m_{z}(w)|\lesssim|\Re w|^{1/3}$ and $|\Im m_{z}(w)|\gtrsim |w|^{1/3}$ follow from the defining cubic equation \eqref{eq:cubic} for $m_{z}(w)$ when $\big||z|-1\big|^{3/2}\leq|w|\leq\epsilon$. The latter condition is guaranteed by the definition of $\Omega_{K,\xi,\epsilon}$. Thus
\begin{align*}
    \left|\frac{\Re m_{z}(E_{j}+i\eta_{j})}{\Im m_{z}(E_{j}+i\eta_{j})}\right|&\lesssim 1.
\end{align*}
Now let $(z_{-t},w_{1,-t},w_{2,-t})\in\Omega_{K,\xi,\epsilon}(t)$. Since $m_{z_{-t}}(w_{j,-t})=e^{-t/2}m_{z}(w_{j})$ (note that this is the reversed dynamics), we have
\begin{align*}
    \left|\frac{\sigma_{j,-t}}{\rho_{j,-t}}\right|&=\left|\frac{\sigma_{j}}{\rho_{j}}\right|\lesssim 1.
\end{align*}
\end{proof}

Let $(z_{t},w_{1,t},w_{2,t})$ solve \eqref{eq:dw} with initial condition $(z,w_{1},w_{2})$ and set $G_{j,t}:=G_{z_{t}}(w_{j,t})$. We define
\begin{align}
    S^{av}_{t}(w_{1},A_{1},...,w_{k},A_{k})&:=\Tr{G_{1,t}A_{1}\cdots G_{k,t}A_{k}},\label{eq:Sav}\\
    S^{iso}_{t}(w_{1},A_{1},...,A_{k-1},w_{k})&:=\mbf{x}^{*}G_{1,t}A_{1}\cdots G_{k-1,t}A_{k-1}G_{k}\mbf{y},\label{eq:Siso}
\end{align}
and denote by $R^{av/iso}_{t}$ the difference between $S^{av/iso}_{t}$ and the corresponding deterministic approximation. For each $R^{av/iso}_{j,t}$, we define an error as a function of the initial condition $(z,w_{1},w_{2})$ by
\begin{align}
    \psi^{av}_{t}(w_{1},F)&:=\frac{\rho_{1,t}}{N\eta_{1,t}},\label{eq:psi1av}\\
    \psi^{av}_{t}(w_{1},F,w_{2},E_{\pm})&:=\frac{1}{\sqrt{N(l_{1,t}\wedge l_{2,t})}}\left(\frac{\phi^{iso}_{1,t}(w_{1},w_{2})}{\sqrt{N(l_{1,t}\vee l_{2,t})}}+\sqrt{\frac{\phi^{av}_{2,t}(w_{1},w_{2})}{N\eta_{1,t}\eta_{2,t}}}\right),\label{eq:psi2av}\\
    \psi^{av}_{t}(w_{1},F,w_{2},F^{*})&:=\frac{\phi^{av}_{2,t}(w_{1},w_{2})}{\sqrt{N(l_{1,t}\wedge l_{2,t})}},\label{eq:psi3av}\\
    \psi^{iso}_{t}(w_{1},F,w_{2})&:=\frac{\phi^{iso}_{1,t}(w_{1},w_{2})}{\sqrt{N(l_{1,t}\wedge l_{2,t})}}+\sqrt{\frac{\phi^{av}_{2,t}(w_{1},w_{2})}{N\eta_{1,t}\eta_{2,t}}},\label{eq:psi1iso}\\
    \psi^{iso}_{t}(w_{1},F,\wh{w}_{2},F^{*},\bar{w}_{1})&:=\phi^{iso}_{2,t}(w_{1},w_{2}),\label{eq:psi2iso}
\end{align}
and stopping times
\begin{align}
    \tau^{av}(w_{1},B_{1},...,w_{k},B_{k})&:=\inf\left\{t\geq0:\frac{|R^{av}_{t}(w_{1},B_{1},....,w_{k},B_{k})|}{\psi^{av}_{t}(w_{1},B_{1},...,w_{k},B_{k})}>N^{2\xi}\right\},\label{eq:tauav1}
\end{align}
and
\begin{align}
    \tau^{iso}_{1}(w_{1},w_{2})&:=\inf\left\{t\geq0:\frac{|R^{iso}_{t}(w_{1},F,w_{2})|}{\psi^{iso}_{t}(w_{1},F,w_{2})}>N^{2\xi}\right\},\label{eq:tauiso1}\\
    \tau^{iso}_{2}(w_{1},w_{2})&:=\inf\left\{t\geq0:\frac{|S^{iso}_{t}(w_{1},F,\wh{w}_{2},F^{*},\bar{w}_{2})|}{\psi^{iso}_{t}(w_{1},F,\wh{w}_{2},F^{*},\bar{w}_{1})}>N^{2\xi}\right\}.\label{eq:tauiso2}
\end{align}
Note that in the definition of $\tau^{iso}_{2}$ we have $S^{iso}_{t}$ instead of $R^{iso}_{t}$, i.e. we are only considering an upper bound on $\mbf{x}^{*}G_{1}F\Im G_{2}F^{*}G^{*}_{1}\mbf{x}$ rather than a bound on the error of the deterministic approximation, hence the fact that the corresponding ``error" parameter defined in \eqref{eq:psi2iso} is simply the upper bound from \eqref{eq:phi2iso}. The overall stopping time is the minimum
\begin{align}
    \tau^{av}(w_{1},w_{2})&:=\min\left\{\tau^{av}(w_{1},F),\tau^{av}(w_{1},F,w_{2},E_{\pm}),\tau^{av}(\wh{w}_{1},F,\wh{w}_{2},F^{*})\right\},\label{eq:tauav}\\
    \tau(w_{1},w_{2})&:=\min\left\{\tau^{av}(w_{1},w_{2}),\tau^{iso}_{1}(w_{1},w_{2}),\tau^{iso}_{2}(w_{1},w_{2}),\tau^{iso}_{2}(w_{2},w_{1})\right\}.\label{eq:tau}
\end{align}

The main goal of this section is the local law for matrices with a Gaussian component of size $T$, which is equivalent to the following statement.
\begin{proposition}\label{prop:gaussLL}
For any sufficiently small and fixed $\epsilon>0$ and $T>0$, we have
\begin{align}
    \tau(w_{1},w_{2})>T
\end{align}
with very high probability, uniformly in $(z,w_{1},w_{2})\in S^{-1}_{T}(\Omega_{K,\xi,\epsilon})$.
\end{proposition}
We prove this through  a series of lemmas, one for each quantity $R^{av/iso}_{j,t}$. The proof of each lemma follows the same pattern as the proofs in \cite[Section 5]{cipolloni_universality_2024}: derive an SDE and bound the time integral of each term. To compute derivatives of the deterministic approximation we can use the ``meta model" argument as outlined in the proof of \cite[Lemma 4.8]{cipolloni_eigenstate_2023}. 

In the remainder of this section we fix a $T>0$ and consider initial conditions in $S^{-1}_{T}(\Omega_{K,\xi,\epsilon})$. In particular, we have $Nl_{j,t}\geq N^{\xi}$ for all $0\leq t\leq T$. In the statements and proofs of Lemmas \ref{lem:av1Gaussian}--\ref{lem:iso2Gaussian}, all inequalities hold with very high probability.
\begin{lemma}\label{lem:av1Gaussian}
Assume that $|R^{av}_{0}(w_{1},F)|\leq N^{\xi/2}\psi^{av}_{0}(w_{1},F)$. Then
\begin{align}
    |R^{av}_{t}(w_{1},F)|&\lesssim N^{\xi/2}\psi^{av}_{0}(w_{1},F).
\end{align}
\end{lemma}
\begin{proof}
By It\^{o}'s lemma and the chain rule we find
\begin{align}
    \mathrm{d}R^{av}_{t}(w_{1},F)&=\frac{1}{2}R^{av}_{t}(w_{1},F)\mathrm{d}t+\Tr{G_{1,t}-M_{t}(w_{1})}\Tr{G_{1,t}FG_{1,t}}\mathrm{d}t\nonumber\\
    &+\frac{1}{2N}\sum_{\nu=\pm}\nu\Tr{G_{1,t}E_{\nu}G^{T}_{1,t}E_{\nu}G_{1,t}F}\mathrm{d}t+\frac{1}{\sqrt{N}}\sum_{i,j=1}^{N}\partial_{ij}(R^{av}_{t}(w_{1},F))\mathrm{d}B_{ij},\label{eq:d(G-M)F}
\end{align}

Consider the stochastic term in \eqref{eq:d(G-M)F}, which has quadratic variation
\begin{align*}
    \frac{1}{N}\sum_{i,j=1}^{N}\left|\Tr{\Delta_{ij}G_{1,t}FG_{1,t}}\right|^{2}\mathrm{d}t&=\frac{1}{N^{2}}\sum_{\nu=\pm}\nu\Tr{G_{1,t}FG_{1,t}E_{\nu}G^{*}_{1,t}F^{*}G^{*}_{1,t}E_{\nu}}\mathrm{d}t\\
    &+\frac{1}{N^{2}}\sum_{\nu=\pm}\nu\Tr{(G_{1,t}FG_{1,t})^{T}E_{\nu}G^{*}_{1,t}F^{*}G_{1,t}^{*}E_{\nu}}\mathrm{d}t.
\end{align*}
Each term on the right hand side can be bounded by
\begin{align*}
    \frac{1}{N^{2}\eta_{1,t}^{2}}\Tr{\Im G_{1,t}F\Im G_{1,t}F^{*}}&\lesssim\frac{\phi^{av}_{2,t}(w_{1},w_{1})}{N^{2}\eta_{1,t}^{2}}\lesssim\frac{\rho_{1,t}^{3}}{N^{2}\eta^{3}_{1,t}},
\end{align*}
using Cauchy-Schwarz. Thus the integral up to $t\leq\tau\wedge T$ is bounded by (c.f. \cite[Eq. (5.70)]{cipolloni_universality_2024})
\begin{align*}
    \int_{0}^{t}\frac{\rho^{3}_{1,s}}{N^{2}\eta_{1,s}^{3}}\,\mathrm{d}s&\lesssim\frac{\rho^{2}_{1,t}}{N^{2}\eta^{2}_{1,t}}\int_{0}^{t}\frac{\rho_{1,s}}{\eta_{1,s}}\,\mathrm{d}s\\
    &\lesssim(\psi^{av}_{t}(w_{1},F))^{2}.
\end{align*}
By the Burkholder-Davis-Gundy (BDG) inequality we obtain
\begin{align*}
    \left|\int_{0}^{t}\frac{1}{\sqrt{N}}\sum_{i,j=1}^{N}\partial_{ij}(R^{(av)}_{s}(w_{1},F))\,\mathrm{d}B_{ij}(s)\right|&\lesssim N^{\xi/2}\psi^{av}_{t}(w_{1},F).
\end{align*}

Now consider the second term in \eqref{eq:d(G-M)F}. Using the single resolvent local law in Proposition \ref{prop:singleLL}, the definition of $\tau$ to bound $\Tr{G_{1,s}FG_{1,s}}$ and the bound for $|\Tr{M_{t}(w_{1},F,w_{2})}|$ in \eqref{eq:M12F}, we bound the integral up to $t\leq\tau\wedge T$ by
\begin{align*}
    \int_{0}^{t}|\Tr{G_{1,s}-M_{s}(w_{1})}\Tr{G_{1,s}FG_{1,s}}|\,\mathrm{d}s&\lesssim\int_{0}^{t}\frac{N^{\xi/2}\rho^{2}_{1,s}}{N\eta^{2}_{1,s}}\,\mathrm{d}s\\
    &\lesssim \frac{N^{\xi/2}\rho_{1,t}}{N\eta_{1,t}}\int_{0}^{t}\frac{\rho_{1,s}}{\eta_{1,s}}\,\mathrm{d}s\\
    &\lesssim N^{\xi/2}\psi^{av}_{t}(w_{1},F).
\end{align*}

Finally, consider the third term,
\begin{align*}
    \frac{1}{N}\sum_{\nu=\pm}\nu\Tr{G_{1,s}^{T}E_{\nu}G_{1,s}FG_{1,s}E_{\nu}}\,\mathrm{d}t.
\end{align*}
Using Cauchy-Schwarz, the single resolvent local law and the definition of $\tau$, we bound the integral by
\begin{align*}
    \frac{1}{N}\int_{0}^{t}\sqrt{\frac{\Tr{\Im G_{1,s}}}{\eta_{1,s}}}\sqrt{\frac{\Tr{\Im G_{1,s}F\Im G_{1,s}F^{*}}}{\eta_{1,s}^{2}}}\,\mathrm{d}s&\leq\frac{1}{N}\int_{0}^{t}\frac{\rho_{1,s}^{2}}{\eta_{1,s}^{2}}\,\mathrm{d}s\\
    &\lesssim\psi^{av}_{1}(w_{1,t}).
\end{align*}
Combining these bounds and the input $|R^{av}_{0}(w_{1},F)|\leq N^{\xi/2}\psi^{av}_{0}(w_{1},F)\leq N^{\xi/2}\psi^{av}_{t}(w_{1},F)$ we conclude the proof.
\end{proof}

\begin{lemma}\label{lem:av2Gaussian}
Let $\mu=\pm$ and assume that $|R^{av}_{0}(w_{1},F,w_{2},E_{\mu})|\leq N^{\xi/2}\psi^{av}_{0}(w_{1},F,w_{2},E_{\mu})$. Then
\begin{align}
    |R^{av}_{t}(w_{1},F,w_{2},E_{\mu})|&\lesssim N^{\xi/2}\psi^{av}_{0}(w_{1},F,w_{2},E_{\mu}).
\end{align}
\end{lemma}
\begin{proof}
We have the SDE
\begin{align}
    \mathrm{d}R^{av}_{t}(w_{1},F,w_{2},E_{\mu})&=R^{av}_{t}(w_{1},F,w_{2},E_{\mu})\mathrm{d}t+\sum_{m=1}^{3}(A_{m}+B_{m})dt\nonumber\\&+\frac{1}{\sqrt{N}}\sum_{i,j=1}^{N}\partial_{ij}(R^{av}_{t}(w_{1},F,w_{2},E_{\mu}))\mathrm{d}B_{ij},\label{eq:R2av}
\end{align}
where
\begin{align*}
    A_{1}&=\Tr{G_{1,t}-M_{t}(w_{1})}\Tr{G^{2}_{1,t}FG_{2,t}E_{\mu}},\\
    A_{2}&=\Tr{G_{2,t}-M_{t}(w_{2})}\Tr{G_{1,t}FG^{2}_{2,t}E_{\mu}},\\
    A_{3}&=\Tr{G_{1,t}FG_{2,t}E_{\mu}}\Tr{G_{2,t}E_{\mu}G_{1,t}E_{\mu}}-\Tr{M_{t}(w_{1},F,w_{2})E_{\mu}}\Tr{M_{t}(w_{2},E_{\mu},w_{1})E_{\mu}},
\end{align*}
and
\begin{align*}
    B_{1}&=\frac{1}{2N}\sum_{\nu=\pm}\nu\Tr{G_{1,t}E_{\nu}G_{1,t}^{T}E_{\nu}G_{1,t}FG_{2,t}E_{\mu}},\\
    B_{2}&=\frac{1}{2N}\sum_{\nu=\pm}\nu\Tr{G_{1,t}FG_{2,t}E_{\nu}G^{T}_{2,t}E_{\nu}G_{2,t}E_{\mu}},\\
    B_{3}&=\frac{2}{2N}\sum_{\nu=\pm}\nu\Tr{G_{1,t}FG_{2,t}E_{\nu}G^{T}_{1,t}E_{\mu}G^{T}_{2,t}E_{\nu}}.
\end{align*}

Consider $A_{1}$; by Cauchy-Schwarz we have
\begin{align*}
    \left|\Tr{G^{2}_{1,t}FG_{2,t}E_{\mu}}\right|&\leq\sqrt{\frac{\Tr{\Im G_{1,t}}}{\eta_{1,t}}}\cdot\sqrt{\frac{\Tr{\Im G_{1,t}F\Im G_{2,t}F^{*}}}{\eta_{1,t}\eta_{2,t}}}.
\end{align*}
Using the definition of $\tau$ and the averaged single resolvent local law we find
\begin{align*}
    \int_{0}^{t}|A_{1,s}|\,\mathrm{d}s&\lesssim \sqrt{\frac{\phi^{av}_{2,t}(w_{1},w_{2})}{N\eta_{1,t}\eta_{2,t}}}\int_{0}^{t}\sqrt{\frac{N^{\xi}\rho_{1,s}}{N\eta^{3}_{1,s}}}\,\mathrm{d}s\\
    &\lesssim\sqrt{\frac{N^{\xi}}{Nl_{1,t}}}\cdot\sqrt{\frac{\phi^{av}_{2,t}(w_{1},w_{2})}{N\eta_{1,t}\eta_{2,t}}}\\
    &\lesssim N^{\xi/2}\psi^{av}_{2,t}(w_{1},F,w_{2},E_{\mu}),
\end{align*}
and similarly for $A_{2}$ with $l_{2,t}$ in place of $l_{1,t}$.

We rewrite $A_{3}$ as follows:
\begin{align*}
    A_{3}&=R^{av}_{t}(w_{1},F,w_{2},E_{\mu})\Tr{M_{t}(w_{2},E_{\mu},w_{1})}\\
    &+\left(\Tr{M_{t}(w_{1},F,w_{2})E_{\mu}}+R^{av}_{t}(w_{1},F,w_{2},E_{\mu})\right)\Tr{(G_{2,t}E_{\mu}G_{1,t}-M_{t}(w_{2},E_{\mu},w_{1}))E_{\mu}}.
\end{align*}
By contour integration or the resolvent identity (depending on the size of $|w_{1}-w_{2}|$) and the averaged single resolvent local law we have
\begin{align*}
    \left|\Tr{(G_{2,t}E_{\mu}G_{1}-M_{t}(w_{2},E_{\mu},w_{1}))E_{\mu}}\right|&\lesssim\frac{N^{\xi/2}}{N\eta_{1,t}\eta_{2,t}}.
\end{align*}
Thus the integral of the term on the second line of the new expression for $A_{3}$ is bounded by
\begin{align*}
    \frac{N^{\xi/2}}{N\sqrt{l_{1,t}l_{2,t}}}\left(\phi^{iso}_{1,t}(w_{1},w_{2})+N^{\xi}\psi^{av}_{2,t}(w_{1},F,w_{2},E_{\mu})\right)\lesssim N^{\xi/2}\psi^{av}_{2,t}(w_{1},F,w_{2},E_{\mu}).
\end{align*}

For $B_{j},\,j=1,2,3$ we use Cauchy-Schwarz. All three terms are analogous so we demonstrate the computations for $B_{1}$:
\begin{align*}
    \int_{0}^{t}|B_{1,s}|\,\mathrm{d}s&\lesssim\frac{1}{N}\int_{0}^{t}\Tr{|E_{\mu}G_{1,s}E_{\nu}G^{T}_{1,s}E_{\nu}|^{2}}^{1/2}\Tr{|G_{1,s}FG_{2,s}|^{2}}^{1/2}\,\mathrm{d}s\\
    &\lesssim\sqrt{\frac{\phi^{av}_{2,t}(w_{1},w_{2})}{N\eta_{1,t}\eta_{2,t}}}\int_{0}^{t}\frac{\rho^{1/2}_{1,s}}{N^{1/2}\eta_{1,s}^{3/2}}\,\mathrm{d}s\\
    &\lesssim\frac{1}{\sqrt{Nl_{1,t}}}\cdot\sqrt{\frac{\phi^{av}_{2,t}(w_{1},w_{2})}{N\eta_{1,t}\eta_{2,t}}}\\
    &\lesssim\psi^{av}_{2,t}(w_{1},F,w_{2},E_{\mu}).
\end{align*}

The quadratic variation of the stochastic term can be bounded by
\begin{align*}
    &\frac{1}{N^{2}}\sum_{\nu=\pm}\nu\Tr{G_{1,t}FG_{2,t}E_{\mu}G_{1,t}E_{\nu}G_{1,t}^{*}E_{\mu}G_{2,t}^{*}F^{*}G_{1,t}^{*}E_{\nu}}\\
    &+\frac{1}{N^{2}}\sum_{\nu=\pm}\nu\Tr{(G_{1,t}FG_{2,t}E_{\mu}G_{1,t})^{T}E_{\nu}G_{1,t}^{*}E_{\mu}G_{2,t}^{*}F^{*}G_{1,t}^{*}E_{\nu}}.
\end{align*}
The term on the second line can be bounded by the first using Cauchy-Schwarz. For the latter, we have
\begin{align*}
    \frac{1}{N^{2}}\left|\Tr{G_{1,t}FG_{2,t}E_{\mu}G_{1,t}E_{\nu}G_{1,t}^{*}E_{\mu}G_{2,t}^{*}F^{*}G_{1,t}^{*}E_{\nu}}\right|&\leq\frac{1}{N^{2}\eta_{1,t}^{3}\eta_{2,t}}\Tr{\Im G_{1,t}F\Im G_{2,t}F^{*}}\\
    &\lesssim\frac{\phi^{av}_{2,t}(w_{1},w_{2})}{N^{2}\eta_{1,t}^{3}\eta_{2,t}}.
\end{align*}
Using the BDG inequality we obtain
\begin{align*}
    \left|\int_{0}^{t}\sum_{i,j=1}^{N}\partial_{ij}(R^{av}_{t}(w_{1},F,w_{2},E_{\mu}))\,\mathrm{d}B_{ij}\right|&\leq\frac{N^{\xi/2}}{\sqrt{N(l_{1,t}\wedge l_{2,t})}}\cdot\sqrt{\frac{\phi^{av}_{2,t}(w_{1},w_{2})}{N\eta_{1,t}\eta_{2,t}}}\\
    &\lesssim N^{\xi/2}\psi^{av}_{t}(w_{1},F,w_{2},E_{\mu}).
\end{align*}

Altogether we have shown that
\begin{align*}
    |R^{av}_{t}(w_{1},F,w_{2},E_{\mu})|&\leq \alpha(t)+\int_{0}^{t}\beta(s)|R^{av}_{s}(w_{1},F,w_{2},E_{\mu})|\,\mathrm{d}s,
\end{align*}
where
\begin{align*}
    \alpha(t)&=|R^{av}_{0}(w_{1},F,w_{2},E_{\mu})|+N^{\xi/2}\psi^{av}_{t}(w_{1},F,w_{2},E_{\mu}),\\
    \beta(t)&=\left|\Tr{M_{s}(w_{2},E_{\mu},w_{1},E_{\mu})}\right|.
\end{align*}
By Gr\"{o}nwall's inequality we obtain
\begin{align*}
    |R^{av}_{t}(w_{1},F,w_{2},E_{\mu})|&\leq\alpha(t)+\int_{0}^{t}\alpha(s)\beta(s)e^{\int_{s}^{t}\beta(r)\,\mathrm{d}r}\,\mathrm{d}s.
\end{align*}
By Cauchy-Schwarz we find
\begin{align*}
    \int_{s}^{t}\beta(r)\,\mathrm{d}r&\leq\int_{s}^{t}\sqrt{\frac{\rho_{1,r}\rho_{2,r}}{\eta_{1,r}\eta_{2,r}}}\,\mathrm{d}r\\
    &\leq\sqrt{\int_{s}^{t}\frac{\rho_{1,r}}{\eta_{1,r}}\,\mathrm{d}r}\sqrt{\int_{s}^{t}\frac{\rho_{2,r}}{\eta_{2,r}}\,\mathrm{d}r}\\
    &\leq\frac{1}{2}\log\left(\frac{\eta_{1,s}}{\eta_{1,t}}\right)+\frac{1}{2}\log\left(\frac{\eta_{2,s}}{\eta_{2,t}}\right).
\end{align*}
Moreover, we have
\begin{align*}
    \int_{0}^{t}\sqrt{\frac{\eta_{1,s}\eta_{2,s}}{\eta_{1,t}\eta_{2,t}}}\cdot\sqrt{\frac{\rho_{1,s}\rho_{2,s}}{\eta_{1,s}\eta_{2,s}}}\cdot\frac{1}{\sqrt{l_{1,s}l_{2,s}}}\,\mathrm{d}s&\lesssim\frac{1}{\sqrt{l_{1,s}l_{2,s}}},\\
    \int_{0}^{t}\sqrt{\frac{\eta_{1,s}\eta_{2,s}}{\eta_{1,t}\eta_{2,t}}}\cdot\sqrt{\frac{\rho_{1,s}\rho_{2,s}}{\eta_{1,s}\eta_{2,s}}}\cdot\frac{1}{\sqrt{\eta_{1,s}\eta_{2,s}}}\,\mathrm{d}s&\lesssim\frac{1}{\sqrt{\eta_{1,t}\eta_{2,t}}}.
\end{align*}
Using these bounds and the fact that 
\begin{align*}
    |R^{av}_{0}(w_{1},F,w_{2},E_{\nu})|&\leq N^{\xi/2}\psi^{av}_{0}(w_{1},F,w_{2},E_{\mu})\leq N^{\xi/2}\psi^{av}_{t}(w_{1},F,w_{2},E_{\mu}),
\end{align*}
by assumption (the second inequality follows because $\psi^{av}_{t}$ is increasing), we conclude the proof.
\end{proof}

\begin{lemma}\label{lem:av3Gaussian}
Assume that $|R^{av}_{0}(\wh{w}_{1},F,\wh{w}_{2},F^{*})|\leq N^{\xi/2}\psi^{av}_{0}(w_{1},F,w_{2},F^{*})$. Then
\begin{align}
    |R^{av}_{t}(\wh{w}_{1},F,\wh{w}_{2},F^{*})|&\lesssim N^{\xi/2}\psi^{av}_{t}(w_{1},F,w_{2},F^{*}).
\end{align}
\end{lemma}
\begin{proof}
The SDE for $R^{av}_{t}(w_{1},F,w_{2},F^{*})$ can be obtained from \eqref{eq:R2av} after replacing $E_{\mu}$ with $F^{*}$. Recall that for $R^{av}_{t}(\wh{w}_{1},F,\wh{w}_{2},F^{*})$, we have $\Im G_{1}$ and $\Im G_{2}$ in place of $G_{1}$ and $G_{2}$. Thus we can replace all linear terms in $G_{j,t}$ in \eqref{eq:R2av} with $\Im G_{j,t}$. Consider $A_{1}$; using the averaged single resolvent local law we find
\begin{align*}
    \left|\Tr{G_{1,t}-M_{1,t}}\Tr{G^{2}_{1,t}F\Im G_{2,t}F^{*}}\right|&\leq\frac{N^{\xi/2}}{N\eta_{1,t}}\cdot\frac{\Tr{\Im G_{1,t}F\Im G_{2,t}F^{*}}}{\eta_{1,t}}\\
    &\leq\frac{N^{\xi/2}\phi^{av}_{2,t}(w_{1},w_{2})}{N\eta^{2}_{1,t}}.
\end{align*}
Integrating up to $t$ using \eqref{eq:int2} we have
\begin{align*}
    \int_{0}^{t}\frac{N^{\xi/2}\phi^{av}_{2,s}(w_{1},w_{2})}{N\eta^{2}_{1,s}}\,\mathrm{d}s&\lesssim\frac{N^{\xi/2}\phi^{av}_{2,t}(w_{1},w_{2})}{Nl_{1,t}}\lesssim N^{\xi/2}\psi^{av}_{t}(w_{1},F,w_{2},F^{*}).
\end{align*}
For $A_{2}$ we have the same bounds with $l_{2,t}$ in place of $l_{1,t}$.

Now consider $A_{3}$, which we split as follows:
\begin{align}
    A_{3}&=\sum_{\nu=\pm}\nu\Tr{(G_{1,t}FG_{2,t}-M_{t}(w_{1},F,w_{2}))E_{\nu}}\Tr{M_{t}(w_{2},F^{*},w_{1})E_{\nu}}\nonumber\\
    &+\sum_{\nu=\pm}\nu\Tr{(G_{2,t}F^{*}G_{1,t}-M_{t}(w_{2},F^{*},w_{1}))E_{\nu}}\Tr{M_{t}(w_{1},F,w_{2})E_{\nu}}\nonumber\\
    &+\sum_{\nu=\pm}\nu\Tr{(G_{1,t}FG_{2,t}-M_{t}(w_{1},F,w_{2}))E_{\nu}}\Tr{(G_{2,t}F^{*}G_{1,t}-M_{t}(w_{2},F^{*},w_{1}))E_{\nu}}.\label{eq:A3}
\end{align}
The terms on the first two lines are of the same form so it is enough to bound one of them. Using \eqref{eq:M12F} and the definition of $\tau$ we have
\begin{align*}
    &\left|\int_{0}^{t}\Tr{(G_{1,s}FG_{2,s}-M_{s}(w_{1},F,w_{2}))E_{\nu}}\Tr{M_{s}(w_{2},F^{*},w_{1})E_{\nu}}\,\mathrm{d}s\right|\\
    &\leq\int_{0}^{t}N^{\xi}\psi^{av}_{s}(w_{1},F,w_{2},E_{\nu})\phi^{iso}_{1,s}(w_{1},w_{2})\,\mathrm{d}s\\
    &\leq\int_{0}^{t}\frac{N^{\xi}\phi^{iso}_{1,s}(w_{1},w_{2})}{\sqrt{N(l_{1,s}\wedge l_{2,s})}}\left(\frac{\phi^{iso}_{1,s}(w_{1},w_{2})}{\sqrt{N(l_{1,s}\vee l_{2,s})}}+\sqrt{\frac{\phi^{av}_{2,s}(w_{1},w_{2})}{N\eta_{1,s}\eta_{2,s}}}\right)\,\mathrm{d}s\\
    &\leq\frac{N^{\xi}\phi^{av}_{2,t}(w_{1},w_{2})}{N(\rho_{1,t}\wedge\rho_{2,t})}\int_{0}^{t}\left(\frac{\phi^{iso}_{1,s}(w_{1},w_{2})}{\sqrt{l_{1,s}l_{2,s}}}+\sqrt{\frac{\phi^{av}_{2,s}(w_{1},w_{2})}{(l_{1,s}\wedge l_{2,s})\eta_{1,s}\eta_{2,s}}}\right)\,\mathrm{d}s.
\end{align*}
Using the bound
\begin{align*}
    \phi^{iso}_{1,s}(w_{1},w_{2})&=\frac{\phi^{av}_{2,s}(w_{1},w_{2})}{\rho_{1,s}\wedge\rho_{2,s}}\leq(\rho_{1,s}\vee\rho_{2,s})\left(\frac{\rho_{1,s}}{\eta_{1,s}}\wedge\frac{\rho_{2,s}}{\eta_{2,s}}\right),
\end{align*}
we find
\begin{align*}
    \frac{1}{\rho_{1,t}\wedge\rho_{2,t}}\int_{0}^{t}\frac{\phi^{iso}_{1,s}(w_{1},w_{2})}{\sqrt{l_{1,s}l_{2,s}}}\,\mathrm{d}s&\lesssim\frac{\rho_{1,t}\vee\rho_{2,t}}{\rho_{1,t}\wedge\rho_{2,t}}\int_{0}^{t}\sqrt{\frac{\rho_{1,s}}{\eta_{1,s}^{3}\eta_{2,s}\rho_{2,s}}}\wedge\sqrt{\frac{\rho_{2,s}}{\eta_{1,s}\eta_{2,s}^{3}\rho_{1,s}}}\,\mathrm{d}s\\
    &\lesssim\frac{\rho_{1,t}\vee\rho_{2,t}}{\rho_{1,t}\wedge\rho_{2,t}}\left(\frac{1}{\eta_{1,t}\rho_{2,t}}\wedge\frac{1}{\eta_{2,t}\rho_{1,t}}\right)\\
    &\lesssim\frac{1}{l_{1,t}\wedge l_{2,t}},
\end{align*}
and
\begin{align*}
    \frac{1}{\rho_{1,t}\wedge\rho_{2,t}}\int_{0}^{t}\sqrt{\frac{\phi^{av}_{2,s}(w_{1},w_{2})}{(l_{1,s}\wedge l_{2,s})\eta_{1,s}\eta_{2,s}}}\,\mathrm{d}s&\lesssim\frac{1}{\sqrt{l_{1,t}\wedge l_{2,t}}(\rho_{1,t}\wedge\rho_{2,t})}\int_{0}^{t}\sqrt{\frac{\rho_{1,s}^{2}\rho_{2,s}}{\eta_{1,s}^{2}\eta_{2,s}}}\wedge\sqrt{\frac{\rho_{1,s}\rho_{2,s}^{2}}{\eta_{1,s}\eta_{2,s}^{2}}}\,\mathrm{d}s\\
    &\lesssim\frac{1}{\sqrt{l_{1,t}\wedge l_{2,t}}(\rho_{1,t}\wedge\rho_{2,t})}\left(\sqrt{\frac{\rho_{1,t}}{\eta_{1,t}}}\wedge\sqrt{\frac{\rho_{2,t}}{\eta_{2,t}}}\right)\\
    &\lesssim\frac{1}{l_{1,t}\wedge l_{2,t}}.
\end{align*}
Combining these two bounds we have
\begin{align*}
    \left|\int_{0}^{t}\Tr{(G_{1,s}FG_{2,s}-M_{s}(w_{1},F,w_{2}))E_{\nu}}\Tr{M_{t}(w_{2},F^{*},w_{1})E_{\nu}}\,\mathrm{d}s\right|&\lesssim\frac{N^{\xi}\phi^{av}_{2,t}(w_{1},w_{2})}{N(l_{1,t}\wedge l_{2,t})}\\
    &\lesssim N^{\xi/2}\psi^{av}_{t}(w_{1},F,w_{2},F^{*}).
\end{align*}
The integral of the term on the third line of \eqref{eq:A3} is bounded by
\begin{align*}
    N^{2\xi}\int_{0}^{t}\left(\psi^{av}_{s}(w_{1},F,w_{2},F^{*})\right)^{2}\,\mathrm{d}s.
\end{align*}
The only new term to estimate here is
\begin{align*}
    N^{2\xi}\int_{0}^{t}\frac{\phi^{av}_{2,s}(w_{1},w_{2})}{N^{2}\eta_{1,s}\eta_{2,s}(l_{1,s}\wedge l_{2,s})}\,\mathrm{d}s&\lesssim\frac{N^{2\xi}\phi^{av}_{2,s}(w_{1},w_{2})}{N^{2}\sqrt{l_{1,t}l_{2,t}}(l_{1,t}\wedge l_{2,t})}\\
    &\lesssim N^{\xi/2}\psi^{av}_{t}(w_{1},F,w_{2},F^{*}).
\end{align*}
Thus we have shown that
\begin{align*}
    \int_{0}^{t}|A_{j,s}|\,\mathrm{d}s&\lesssim N^{\xi/2}\psi^{av}_{t}(w_{1},F,w_{2},F^{*}),\quad j=1,2,3.
\end{align*}

Now consider $B_{1}$ and $B_{2}$; it is enough to bound $B_{1}$ since they are of the same form. Moreover, as mentioned earlier, since we are considering the evolution of $d\Tr{\Im G_{1,t}F\Im G_{2,t}F^{*}}$, it is enough to bound $B_{1}$ with $\Im G_{2,t}$ in place of $G_{2,t}$, i.e.
\begin{align*}
    \frac{1}{N}\sum_{\nu=\pm}\nu\Tr{G_{1,t}E_{\nu}G^{T}_{1,t}E_{\nu}G_{1,t}F\Im G_{2,t}F^{*}}.
\end{align*}
Using Cauchy-Schwarz we have
\begin{align*}
    \frac{1}{N}\left|\Tr{G_{1,t}E_{\nu}G^{T}_{1,t}E_{\nu}G_{1,t}F\Im G_{2,t}F^{*}}\right|&\leq\frac{1}{N}\Tr{\frac{\Im G_{1,t}}{\eta_{1,t}}}^{1/2}\cdot\Tr{\left(\frac{\Im G_{1,t}}{\eta_{1,t}}F\Im G_{2,t}F^{*}\right)^{2}}^{1/2}\\
    &\lesssim\frac{\rho_{1,t}^{1/2}}{N^{1/2}\eta_{1,t}^{3/2}}\Tr{\Im G_{1,t}F\Im G_{2,t}F^{*}}\\
    &\lesssim\frac{\rho_{1,t}^{1/2}\phi^{av}_{2,t}(w_{1},w_{2})}{N^{1/2}\eta_{1,t}^{3/2}}.
\end{align*}
By \eqref{eq:int2} we find
\begin{align*}
    \left|\int_{0}^{t}\frac{1}{N}\Tr{G_{1,s}E_{\nu}G^{T}_{1,s}E_{\nu}G_{1,s}F\Im G_{2,s}F^{*}}\,\mathrm{d}s\right|&\lesssim\frac{\rho_{1,t}^{1/2}\phi^{av}_{2,t}(w_{1},w_{2})}{N^{1/2}}\int_{0}^{t}\frac{1}{\eta_{1,s}^{3/2}}\,\mathrm{d}s\\
    &\lesssim\frac{\phi^{av}_{2,t}(w_{1},w_{2})}{\sqrt{Nl_{1,t}}}.
\end{align*}

Consider $B_{3}$; by Cauchy-Schwarz we have
\begin{align*}
    \frac{1}{N}\left|\Tr{G_{1,t}FG_{2,t}E_{\nu}G^{T}_{1,t}FG^{T}_{2,t}E_{\nu}}\right|&\leq\frac{1}{N\eta_{1,t}\eta_{2,t}}\Tr{\Im G_{1,t}F\Im G_{2,t}F^{*}}\\
    &\lesssim\frac{\phi^{av}_{2,t}(w_{1},w_{2})}{N\eta_{1,t}\eta_{2,t}}.
\end{align*}
Integrating up to $t$ we find
\begin{align*}
    \int_{0}^{t}\frac{\phi^{av}_{2,s}(w_{1},w_{2})}{N\eta_{1,s}\eta_{2,s}}\,\mathrm{d}s&\lesssim\frac{\phi^{av}_{2,t}(w_{1},w_{2})}{N\sqrt{l_{1,t}l_{2,t}}}.
\end{align*}

Finally, we consider the quadratic variation of the stochastic term. Again it is enough to replace $G_{j,t}$ with $\Im G_{j,t}$ wherever the former appears linearly. A representative term is
\begin{align*}
    \frac{1}{N^{2}}\left|\Tr{G_{1,t}F\Im G_{2,t}F^{*}G_{1,t}E_{\nu}G_{1,t}^{*}F^{*}\Im G_{2,t}FG_{1,t}E_{\nu}}\right|&\leq\frac{1}{N\eta_{1,t}^{2}}\Tr{\Im G_{1,t}F\Im G_{2,t}F^{*}}^{2}\\
    &\leq\frac{(\phi^{av}_{2,t}(w_{1},w_{2}))^{2}}{N\eta_{1,t}^{2}}.
\end{align*}
Integrating up to $t$ using \eqref{eq:int2} we find
\begin{align*}
    \left(\int_{0}^{t}\frac{(\phi^{av}_{2,s}(w_{1},w_{2}))^{2}}{N\eta_{1,s}^{2}}\,\mathrm{d}s\right)^{1/2}&\lesssim\frac{\phi^{av}_{2,t}(w_{1},w_{2})}{\sqrt{Nl_{1,t}}}.
\end{align*}
Combining all these bounds and the input $|R^{av}_{0}(\wh{w}_{1},F,\wh{w}_{2},F^{*})|<N^{\xi/2}$ we conclude the proof.
\end{proof}

Now we come to the isotropic quantities.
\begin{lemma}\label{lem:iso1Gaussian}
Assume that $|R^{iso}_{0}(w_{1},F,w_{2})|<N^{\xi/2}\psi^{iso}_{0}(w_{1},F,w_{2})$. Then
\begin{align}
    |R^{iso}_{t}(w_{1},F,w_{2})|\lesssim N^{\xi/2}\psi^{iso}_{0}(w_{1},F,w_{2}).
\end{align}
\end{lemma}
\begin{proof}
We have the SDE
\begin{align*}
    \mathrm{d}R^{iso}_{t}(w_{1},F,w_{2})&=R^{iso}_{t}(w_{1},F,w_{2})\mathrm{d}t+\frac{1}{2}\sum_{m=1}^{3}(A_{m}+B_{m})\mathrm{d}t\\
    &+\frac{1}{\sqrt{N}}\sum_{i,j=1}^{N}\partial_{ij}(R^{iso}_{t}(w_{1},F,w_{2}))\mathrm{d}B_{ij},
\end{align*}
where
\begin{align*}
    A_{1,t}&=\Tr{G_{1,t}-M_{1,t}}\mbf{x}^{*}G_{1,t}^{2}FG_{2,t}\mbf{y},\\
    A_{2,t}&=\Tr{G_{2,t}-M_{2,t}}\mbf{x}^{*}G_{1,t}FG^{2}_{2,t}\mbf{y},\\
    A_{3,t}&=\sum_{\nu=\pm}\nu\left[\Tr{G_{1,t}FG_{2,t}E_{\nu}}\mbf{x}^{*}G_{1,t}E_{\nu}G_{2,t}\mbf{y}-\Tr{M_{t}(w_{1},F,w_{2})E_{\nu}}\mbf{x}^{*}M_{t}(w_{1},E_{\nu},w_{2})\mbf{y}\right],
\end{align*}
and
\begin{align*}
    B_{1,t}&=\frac{1}{N}\sum_{\nu=\pm}\nu\mbf{x}^{*}G_{1,t}E_{\nu}G_{1,t}^{T}E_{\nu}G_{1,t}FG_{2,t}\mbf{y},\\
    B_{2,t}&=\frac{1}{N}\sum_{\nu=\pm}\nu\mbf{x}^{*}G_{1,t}E_{\nu}(G_{1,t}FG_{2,t})^{T}E_{\nu}G_{2,t}\mbf{y},\\
    B_{3,t}&=\frac{1}{N}\sum_{\nu=\pm}\nu\mbf{x}^{*}G_{1,t}FG_{2,t}E_{\nu}G_{2,t}^{T}E_{\nu}G_{2,t}\mbf{y}.
\end{align*}

Consider $A_{1}$; by Cauchy-Schwarz we have
\begin{align*}
    \left|\mbf{x}^{*}G_{1,t}^{2}FG_{2,t}\mbf{y}\right|&\leq\sqrt{\frac{\mbf{x}^{*}\Im G_{1,t}\mbf{x}}{\eta_{1,t}}}\cdot\sqrt{\frac{\mbf{y}^{*}G_{2,t}^{*}F^{*}\Im G_{1,t}FG_{2,t}\mbf{y}}{\eta_{1,t}}}\\
    &\lesssim\sqrt{\frac{N^{\xi}\rho_{1,t}\phi^{iso}_{2,t}(w_{2},w_{1})}{\eta_{1,t}^{2}}}\\
    &=\sqrt{\frac{N^{\xi}\rho_{1,t}}{\eta_{1,t}}}\cdot\sqrt{\frac{\phi^{av}_{2,t}(w_{1},w_{2})}{\eta_{1,t}\eta_{2,t}}}.
\end{align*}
Using the averaged single resolvent local law we obtain
\begin{align*}
    \int_{0}^{t}A_{1,s}\,\mathrm{d}s&\lesssim N^{\xi}\sqrt{\frac{\phi^{av}_{2,t}(w_{1},w_{2})}{N\eta_{1,t}\eta_{2,t}}}\int_{0}^{t}\sqrt{\frac{\rho_{1,s}}{N\eta_{1,s}^{3}}}\,\mathrm{d}s\\
    &\lesssim \frac{N^{\xi}}{\sqrt{Nl_{1,s}}}\cdot\sqrt{\frac{\phi^{av}_{2,t}(w_{1},w_{2})}{N\eta_{1,t}\eta_{2,t}}}\\
    &\lesssim N^{\xi/2}\psi^{iso}_{t}(w_{1},F,w_{2}),
\end{align*}
and likewise for $A_{2,s}$.

We rewrite $A_{3}$ as follows:
\begin{align*}
    A_{3}&=\sum_{\nu=\pm}\nu\Tr{(G_{1,t}FG_{2,t}-M_{t}(w_{1},F,w_{2}))E_{\nu}}\mbf{x}^{*}M_{t}(w_{1},E_{\nu},w_{2})\mbf{y}\\
    &+\sum_{\nu=\pm}\nu\Tr{M_{t}(w_{1},F,w_{2})E_{\nu}}\mbf{x}^{*}(G_{1,t}E_{\nu}G_{2,t}-M_{t}(w_{1},E_{\nu},w_{2}))\mbf{y}.
\end{align*}
We have the bounds
\begin{align*}
    \left|\Tr{(G_{1,t}FG_{2,t}-M_{t}(w_{1},F,w_{2}))E_{\nu}}\right|&\lesssim N^{\xi}\psi^{av}_{t}(w_{1},F,w_{2}),\\
    \left|\mbf{x}^{*}(G_{1,t}E_{\nu}G_{2,t}-M_{t}(w_{1},E_{\nu},w_{2}))\mbf{y}\right|&\prec\frac{\rho_{1,t}^{1/2}}{N^{1/2}\eta_{1,t}^{3/2}}+\frac{\rho_{2,t}^{1/2}}{N^{1/2}\eta_{2,t}^{3/2}},\\
    \left|\mbf{x}^{*}M_{t}(w_{1},E_{\nu},w_{2})\mbf{y}\right|&\lesssim\sqrt{\frac{\rho_{1,t}\rho_{2,t}}{\eta_{1,t}\eta_{2,t}}},
\end{align*}
where the second bound follows from the isotropic single resolvent local law and contour integration. Moreover,
\begin{align*}
    \int_{0}^{t}\frac{\rho_{j,s}^{1/2}}{N^{1/2}\eta_{j,s}^{3/2}}\,\mathrm{d}s&\lesssim\frac{1}{\sqrt{N\eta_{j,t}\rho_{j,t}}},\\
    \int_{0}^{t}\sqrt{\frac{\rho_{1,s}\rho_{2,s}}{\eta_{1,s}\eta_{2,s}}}\,\mathrm{d}s&\lesssim1.
\end{align*}
Thus we obtain
\begin{align*}
    \int_{0}^{t}|A_{3,s}|\,\mathrm{d}s&\prec N^{\xi/2}\psi^{iso}_{t}(w_{1},F,w_{2}).
\end{align*}

Now consider $B_{1}$; by Cauchy-Schwarz we have
\begin{align*}
    \left|\mbf{x}^{*}G_{1,t}E_{\nu}G_{1,t}^{T}E_{\nu}G_{1,t}FG_{2,t}\mbf{y}\right|&\leq\left(\mbf{x}^{*}G_{1,t}E_{\nu}G_{1,t}^{T}E_{\nu}\bar{G}_{1}E_{\nu}G_{1,t}^{*}\mbf{x}\right)^{1/2}\left(\mbf{y}^{*}G_{2,t}^{*}F^{*}G_{1,t}^{*}G_{1,t}FG_{2,t}\mbf{y}\right)^{1/2}\\
    &\leq\left(\frac{\mbf{x}^{*}\Im G_{1,t}\mbf{x}}{\eta_{1,t}^{3}}\right)^{1/2}\left(\frac{\mbf{y}^{*}G_{2,t}^{*}F^{*}\Im G_{1,t}FG_{2,t}\mbf{y}}{\eta_{1,t}}\right)^{1/2}\\
    &\prec\sqrt{\frac{\rho_{1}\phi^{iso}_{2,t}(w_{2},w_{1})}{\eta_{1,t}^{4}}}.
\end{align*}
After integration we obtain
\begin{align*}
    \int_{0}^{t}|B_{1,s}|\,\mathrm{d}s&\prec\frac{\sqrt{\phi^{iso}_{2,t}(w_{2},w_{1})}}{N\eta_{1,t}\rho_{1,t}^{1/2}}=\frac{1}{\sqrt{Nl_{1,t}}}\cdot\sqrt{\frac{\phi^{av}_{2,t}(w_{1},w_{2})}{N\eta_{1,t}\eta_{2,t}}}.
\end{align*}
By symmetry we also obtain
\begin{align*}
    \int_{0}^{t}|B_{3,s}|\,\mathrm{d}s&\prec\frac{1}{\sqrt{Nl_{2,t}}}\cdot\sqrt{\frac{\phi^{av}_{2}(w_{1,t},w_{2,t})}{N\eta_{1,t}\eta_{2,t}}}.
\end{align*}
For $B_{2}$ we again use Cauchy-Schwarz to find
\begin{align*}
    \left|\mbf{y}^{T}G_{2,t}^{T}E_{\nu}G_{1,t}FG_{2,t}E_{\nu}G_{1,t}^{T}\bar{\mbf{x}}\right|&\leq\left(\frac{\mbf{y}^{T}G_{2,t}^{T}E_{\nu}G_{1,t}F\Im G_{2,t}F^{*}G_{1,t}^{*}E_{\nu}\bar{G}_{2}\bar{\mbf{y}}}{\eta_{2,t}}\right)^{1/2}\left(\mbf{x}^{T}\bar{G}_{1}E_{\nu}G_{1,t}^{T}\mbf{x}\right)^{1/2}.
\end{align*}
For the first factor on the right hand side we use Cauchy-Schwarz again:
\begin{align*}
    \mbf{y}^{T}G_{2,t}^{T}E_{\nu}G_{1,t}F\Im G_{2,t}F^{*}G_{1,t}^{*}E_{\nu}\bar{G}_{2,t}\bar{\mbf{y}}&\leq\left(\mbf{y}^{T}G_{2,t}^{T}E_{\nu}\bar{G}_{2,t}\bar{\mbf{y}}\right)^{1/2}\\
    &\times\left(\frac{\mbf{y}^{T}G_{2,t}^{T}E_{\nu}G_{1,t}F\Im G_{2,t}F^{*}\Im G_{1,t}F\Im G_{2,t}F^{*}G_{1,t}^{*}E_{\nu}\bar{G}_{2,t}\bar{\mbf{y}}}{\eta_{1,t}}\right)^{1/2}.
\end{align*}
Now we use a spectral decomposition for the second term on the right hand side:
\begin{align*}
    &\mbf{y}^{T}G_{2,t}^{T}E_{\nu}G_{1,t}F\Im G_{2,t}F^{*}\Im G_{1,t}F\Im G_{2,t}F^{*}G_{1,t}^{*}E_{\nu}\bar{G}_{2,t}\bar{\mbf{y}}\\
    &=\sum_{n,m}\frac{\eta^{2}(\mbf{y}^{T}G_{2,t}^{T}E_{\nu}G_{1,t}F\mbf{w}_{2,n})(\mbf{w}_{2,n}^{*}F^{*}\Im G_{1,t}F\mbf{w}_{2,m})(\mbf{w}_{2,m}^{*}F^{*}G_{1,t}^{*}E_{\nu}\bar{G}_{2,t}\bar{\mbf{y}})}{|\lambda_{2,n}-w_{2}|^{2}|\lambda_{2,m}-w_{2}|^{2}}\\
    &\leq\left(\sum_{n,m}\frac{\eta^{2}|\mbf{y}^{T}G_{2,t}^{T}E_{\nu}G_{1,t}F\mbf{w}_{2,n}|^{2}|\mbf{w}_{2,m}^{*}F^{*}G_{1,t}^{*}E_{\nu}\bar{G}_{2,t}\bar{\mbf{y}}|^{2}}{|\lambda_{2,n}-w_{2}|^{2}|\lambda_{2,m}-w_{2}|^{2}}\right)^{1/2}\\
    &\times\left(\sum_{n,m}\frac{\eta^{2}|\mbf{w}_{2,n}^{*}F^{*}\Im G_{1,t}F\mbf{w}_{2,m}|^{2}}{|\lambda_{2,n}-w_{2}|^{2}|\lambda_{2,m}-w_{2}|^{2}}\right)^{1/2}\\
    &\leq N\Tr{\Im G_{1,t}F\Im G_{2,t}F^{*}}\mbf{y}^{T}G_{2,t}^{T}E_{\nu}G_{1,t}F\Im G_{2,t}F^{*}G_{1,t}^{*}E_{\nu}\bar{G}_{2,t}\bar{\mbf{y}}.
\end{align*}
Combined with the preceeding display we have shown that
\begin{align*}
    &\mbf{y}^{T}G_{2,t}^{T}E_{\nu}G_{1,t}F\Im G_{2,t}F^{*}G_{1,t}^{*}E_{\nu}\bar{G}_{2,t}\bar{\mbf{y}}\\
    &\leq\left(\mbf{y}^{T}G_{2,t}^{T}E_{\nu}\bar{G}_{2,t}\bar{\mbf{y}}\right)^{1/2}\left(\frac{N\Tr{\Im G_{1,t}F\Im G_{2,t}F^{*}}\mbf{y}^{T}G_{2,t}^{T}E_{\nu}G_{1,t}F\Im G_{2,t}F^{*}G_{1,t}^{*}E_{\nu}\bar{G}_{2,t}\bar{\mbf{y}}}{\eta_{1,t}}\right)^{1/2},
\end{align*}
which after rearrangement becomes
\begin{align*}
    \mbf{y}^{T}G_{2,t}^{T}E_{\nu}G_{1,t}F\Im G_{2,t}F^{*}G_{1,t}^{*}E_{\nu}\bar{G}_{2,t}\bar{\mbf{y}}&\leq\frac{N\Tr{\Im G_{1,t}F\Im G_{2,t}F^{*}}\mbf{y}^{T}G_{2,t}^{T}E_{\nu}\bar{G}_{2,t}\bar{\mbf{y}}}{\eta_{1,t}}\\
    &\lesssim\frac{N\rho_{2,t}\phi^{av}_{2}(w_{1},w_{2})}{\eta_{1,t}\eta_{2,t}}.
\end{align*}
After integration we obtain
\begin{align*}
    \int_{0}^{t}|B_{2,s}|\,\mathrm{d}s&\lesssim\int_{0}^{t}\sqrt{\frac{\rho_{1,s}\rho_{2,s}\phi^{av}_{2,s}(w_{1},w_{2})}{N\eta_{1,s}^{2}\eta_{2,s}^{2}}}\,\mathrm{d}s\\
    &\lesssim\sqrt{\frac{\phi^{av}_{2,t}(w_{1},w_{2})}{N\eta_{1,t}\eta_{2,t}}}\int_{0}^{t}\sqrt{\frac{\rho_{1,s}\rho_{2,s}}{\eta_{1,s}\eta_{2,s}}}\,\mathrm{d}s\\
    &\lesssim\sqrt{\frac{\phi^{av}_{2,t}(w_{1},w_{2})}{N\eta_{1,t}\eta_{2,t}}}.
\end{align*}

It remains to estimate the quadratic variation of the stochastic term, which is bounded (up to a constant) by
\begin{align*}
    \frac{1}{N}\sum_{i,j=1}^{N}\left(\left|\mbf{x}^{*}G_{1,t}\Delta_{ij}G_{1,t}FG_{2,t}\mbf{y}\right|^{2}+\left|\mbf{x}^{*}G_{1,t}FG_{2,t}\Delta_{ij}G_{2,t}\mbf{y}\right|^{2}\right).
\end{align*}
Consider the first term on the right hand side:
\begin{align*}
    \sum_{i,j=1}^{N}\left|\mbf{x}^{*}G_{1,t}\Delta_{ij}G_{1,t}FG_{2,t}\mbf{y}\right|^{2}&=\sum_{\nu=\pm}\nu\mbf{x}^{*}G_{1,t}E_{\nu}G_{1,t}^{*}\mbf{x}\mbf{y}^{*}G_{2,t}^{*}F^{*}G_{1,t}^{*}E_{\nu}G_{1,t}FG_{2,t}\mbf{y}\\
    &+\sum_{\nu=\pm}\nu\mbf{x}^{T}\bar{G}_{1}E_{\nu}G_{1,t}FG_{2,t}\mbf{y}\mbf{y}^{*}G_{2,t}^{*}F^{*}G_{1,t}^{*}E_{\nu}G_{1,t}^{T}\bar{\mbf{x}}.
\end{align*}
The second term can be bounded by the first with Cauchy-Schwarz. Using $E_{\nu}\leq1$ and the definition of $\tau$ we have
\begin{align*}
    \mbf{y}^{*}G_{2,t}^{*}F^{*}G_{1,t}^{*}E_{\nu}G_{1,t}FG_{2,t}\mbf{y}&\lesssim\frac{N^{\xi}\phi^{iso}_{2,t}(w_{2},w_{1})}{\eta_{1,t}}.
\end{align*}
Therefore we obtain
\begin{align*}
    \int_{0}^{t}\frac{1}{N}\sum_{i,j=1}^{N}\left|\mbf{x}^{*}G_{1,s}\Delta_{ij}G_{1,s}FG_{2,s}\mbf{y}\right|^{2}\,\mathrm{d}s&\lesssim N^{\xi}\int_{0}^{t}\left(\frac{\rho_{1,s}\phi^{iso}_{2,s}(w_{2},w_{1})}{N\eta_{1,s}^{2}}+\frac{\rho_{2,s}\phi^{iso}_{2,s}(w_{1},w_{2})}{N\eta_{2,s}^{2}}\right)\,\mathrm{d}s\\
    &\lesssim\frac{N^{\xi}\phi^{av}_{2,t}(w_{1},w_{2})}{N\eta_{1,s}\eta_{2,s}}\int_{0}^{t}\left(\frac{\rho_{1,s}}{\eta_{1,s}}+\frac{\rho_{2,s}}{\eta_{2,s}}\right)\,\mathrm{d}s\\
    &\lesssim\frac{N^{\xi}\phi^{av}_{2,t}(w_{1},w_{2})}{N\eta_{1,t}\eta_{2,t}}.
\end{align*}
This completes the estimates of each term in the SDE and hence the proof of the lemma.
\end{proof}

For the next lemma we only prove an upper bound on $S^{iso}_{t}(w_{1},F,\wh{w}_{2},F^{*},\bar{w}_{1})$ and not on the difference $R^{iso}_{t}$. We do this because it is sufficient for our purposes and slightly simpler.
\begin{lemma}\label{lem:iso2Gaussian}
Assume that $S^{iso}_{0}(w_{1},F,\wh{w}_{2},F^{*},\bar{w}_{1})\leq N^{\xi/2}\phi^{iso}_{2,0}(w_{1},w_{2})$. Then for any $t\geq0$ we have
\begin{align}
    S^{iso}_{t}(w_{1},F,\wh{w}_{2},F^{*},\bar{w}_{1})\leq N^{\xi/2}\psi^{iso}_{2,t}(w_{1},w_{2}).
\end{align}
\end{lemma}
\begin{proof}
We have the SDE
\begin{align*}
    \mathrm{d}S^{iso}_{t}(w_{1},F,\wh{w}_{2},F^{*},\bar{w}_{1})&=\frac{3}{2}S^{iso}_{t}(w_{1},F,\wh{w}_{2},F^{*},\bar{w}_{1})\mathrm{d}t\\&+\frac{1}{2}\sum_{m=1}^{6}(A_{m}+B_{m})\mathrm{d}t+\frac{1}{\sqrt{N}}\sum_{i,j=1}^{N}\partial_{ij}(S^{iso}_{t}(w_{1},F,\wh{w}_{2},F^{*},\bar{w}_{1}))\mathrm{d}B_{ij},
\end{align*}
where
\begin{align*}
    A_{1}&=\Tr{G_{1,t}-M_{1}}\mbf{x}^{*}G_{1,t}^{2}F\Im G_{2,t}F^{*}G^{*}_{1}\mbf{x},\\
    A_{2}&=\Im\left(\Tr{G_{2,t}-M_{2}}\mbf{x}^{*}G_{1,t}FG^{2}_{2}F^{*}G^{*}_{1}\mbf{x}\right),\\
    A_{3}&=\Tr{G^{*}_{1}-M^{*}_{1}}\mbf{x}^{*}G_{1,t}F\Im G_{2,t}F^{*}(G^{*}_{1})^{2}\mbf{x},\\
    A_{4}&=\Im\sum_{\nu=\pm}\nu\Tr{G_{1,t}FG_{2,t}E_{\nu}}\mbf{x}^{*}G_{1,t}E_{\nu}G_{2,t}F^{*}G^{*}_{1}\mbf{x},\\
    A_{5}&=\Im\sum_{\nu=\pm}\nu\Tr{G_{2,t}F^{*}G^{*}_{1}E_{\nu}}\mbf{x}^{*}G_{1,t}FG_{2,t}E_{\nu}G^{*}_{1}\mbf{x},\\
    A_{6}&=\sum_{\nu=\pm}\nu\Tr{G_{1,t}F\Im G_{2,t}F^{*}G^{*}_{1}E_{\nu}}\mbf{x}^{*}G_{1,t}E_{\nu}G^{*}_{1}\mbf{x},
\end{align*}
and
\begin{align*}
    B_{1}&=\frac{1}{N}\sum_{\nu=\pm}\nu\mbf{x}^{*}G_{1,t}E_{\nu}G_{1,t}^{T}E_{\nu}G_{1,t}F\Im G_{2,t}F^{*}G^{*}_{1}\mbf{x},\\
    B_{2}&=\frac{1}{N}\Im\sum_{\nu=\pm}\nu\mbf{x}^{*}G_{1,t}E_{\nu}(G_{1,t}FG_{2,t})^{T}E_{\nu}G_{2,t}F^{*}G^{*}_{1}\mbf{x},\\
    B_{3}&=\frac{1}{N}\sum_{\nu=\pm}\nu\mbf{x}^{*}G_{1,t}E_{\nu}(G_{1,t}F\Im G_{2,t}F^{*}G^{*}_{1})^{T}E_{\nu}G^{*}_{1}\mbf{x},\\
    B_{4}&=\frac{1}{N}\Im\sum_{\nu=\pm}\nu\mbf{x}^{*}G_{1,t}FG_{2,t}E_{\nu}G^{T}_{2}E_{\nu}G_{2,t}F^{*}G^{*}_{1}\mbf{x},\\
    B_{5}&=\frac{1}{N}\Im\sum_{\nu=\pm}\nu\mbf{x}^{*}G_{1,t}FG_{2,t}E_{\nu}(G_{2,t}F^{*}G^{*}_{1})^{T}E_{\nu}G^{*}_{1}\mbf{x},\\
    B_{6}&=\frac{1}{N}\sum_{\nu=\pm}\nu\mbf{x}^{*}G_{1,t}F\Im G_{2,t}F^{*}G^{*}_{1}E_{\nu}\bar{G}_{1}E_{\nu}G^{*}_{1}\mbf{x}.
\end{align*}

Consider $A_{1}$; by Cauchy-Schwarz, the isotropic single resolvent local law and a spectral decomposition we have
\begin{align*}
    \left|\mbf{x}^{*}G_{1,t}^{2}F\Im G_{2,t}F^{*}G_{1,t}^{*}\mbf{x}\right|&\leq\left(\frac{\mbf{x}^{*}\Im G_{1,t}\mbf{x}}{\eta_{1,t}}\right)^{1/2}\left(\mbf{x}^{*}G_{1,t}F\Im G_{2,t}F^{*}\frac{\Im G_{1,t}}{\eta_{1,t}}F\Im G_{2,t}F^{*}G_{1,t}^{*}\mbf{x}\right)^{1/2}\\
    &\lesssim\sqrt{\frac{N^{1+\xi}\rho_{1,t}\phi^{iso}_{2,t}(w_{1},w_{2})\phi^{av}_{2,t}(w_{1},w_{2})}{\eta_{1,t}^{2}}}\\
    &=\frac{N^{(1+\xi)/2}\rho^{1/2}_{1,t}\phi^{iso}_{2,t}(w_{1},w_{2})}{\eta^{1/2}_{1,t}}.
\end{align*}
By the averaged single resolvent local law we find
\begin{align*}
    \int_{0}^{t}|A_{1,s}|\,\mathrm{d}s&\lesssim\int_{0}^{t}\frac{N^{\xi/2}}{N\eta_{1,s}}\cdot\frac{N^{(1+\xi)/2}\rho^{1/2}_{1,s}\phi^{iso}_{2,s}(w_{1},w_{2})}{\eta^{1/2}_{1,s}}\,\mathrm{d}s\\
    &\lesssim\frac{N^{\xi}\phi^{iso}_{2,t}(w_{1},w_{2})}{\sqrt{Nl_{1,t}}}\\
    &\lesssim N^{\xi/2}\phi^{iso}_{2,t}(w_{1},w_{2}),
\end{align*}
and similarly for $A_{3}$. $A_{2}$ can be bounded using the fact that $|G^{2}_{2}|=\frac{\Im G_{2,t}}{\eta_{2}}$:
\begin{align*}
    \int_{0}^{t}|A_{2,s}|\,\mathrm{d}s&\lesssim\int_{0}^{t}\frac{N^{\xi/2}}{N\eta_{2,s}}\cdot\frac{N^{\xi}\phi^{iso}_{2,s}(w_{1},w_{2})}{\eta_{2,s}}\,\mathrm{d}s\\
    &\lesssim\frac{N^{\xi/2}\phi^{iso}_{2,s}(w_{1},w_{2})}{Nl_{2,t}}.
\end{align*}

We write $A_{4}+A_{5}$ as follows:
\begin{align*}
    A_{4}+A_{5}&=\sum_{\nu=\pm}\nu\left[\Tr{G_{1,t}F\Im G_{2,t}E_{\nu}}\mbf{x}^{*}G_{1,t}E_{\nu}G^{*}_{2}F^{*}G_{1,t}^{*}\mbf{x}+\Tr{G_{1,t}FG_{2,t}E_{\nu}}\mbf{x}^{*}G_{1,t}E_{\nu}\Im G_{2,t} F^{*}G_{1,t}^{*}\mbf{x}\right]\\
    &+\sum_{\nu=\pm}\nu\left[\Tr{\Im G_{2,t}F^{*}G_{1,t}^{*}E_{\nu}}\mbf{x}^{*}G_{1,t}FG_{2,t}E_{\nu}G_{1,t}^{*}\mbf{x}+\Tr{G^{*}_{2}F^{*}G_{1,t}^{*}E_{\nu}}\mbf{x}^{*}G_{1,t}F\Im G_{2,t}E_{\nu}G_{1,t}^{*}\mbf{x}\right].
\end{align*}
The second line is analogous to the first so we concentrate on the latter. For the first term in the first line we have
\begin{align*}
    \left|\Tr{G_{1,t}F\Im G_{2,t}E_{\nu}}\mbf{x}^{*}G_{1,t}E_{\nu}G_{2,t}^{*}F^{*}G_{1,t}^{*}\mbf{x}\right|&\leq\sqrt{\frac{\Tr{\Im G_{1,t}F\Im G_{2,t}F^{*}}\Tr{\Im G_{2,t}}}{\eta_{1,t}}}\nonumber\\
    &\times\sqrt{\frac{\mbf{x}^{*}\Im G_{1,t}\mbf{x}\cdot\mbf{x}^{*}G_{1,t}F\Im G_{2,t}F^{*}G_{1,t}^{*}\mbf{x}}{\eta_{1,t}\eta_{2,t}}}\nonumber\\
    &\lesssim\sqrt{\frac{N^{\xi}\rho_{1,t}\rho_{2,t}}{\eta_{1,t}\eta_{2,t}}}\cdot\phi^{iso}_{2,t}(w_{1},w_{2}).
\end{align*}
For the second term, we note that $\Im G_{2,t}$ commutes with $E_{\nu}$ and so
\begin{align*}
    \left|\mbf{x}^{*}G_{1,t}E_{\nu}\Im G_{2,t}F^{*}G_{1,t}^{*}\mbf{x}\right|&=\left|\mbf{x}^{*}G_{1,t}\Im G_{2,t}F^{*}G_{1,t}^{*}\mbf{x}\right|\\
    &\leq\sqrt{\frac{\mbf{x}^{*}\Im G_{1,t}\Im G_{2,t}\mbf{x}\cdot\mbf{x}^{*}G_{1,t}F\Im G_{2,t}F^{*}G_{1,t}^{*}\mbf{x}}{\eta_{1,t}}}\\
    &\lesssim\sqrt{\frac{N^{\xi}\phi^{iso}_{2,t}(w_{1},w_{2})}{\eta_{1,t}}\left(\frac{\rho_{1,t}}{\eta_{2,t}}\wedge\frac{\rho_{2,t}}{\eta_{1,t}}\right)}\\
    &\lesssim\sqrt{\frac{N^{\xi}\rho_{1,t}\rho_{2,t}}{\eta_{1,t}^{2}}\left(\frac{\rho_{1,t}}{\eta_{2,t}}\wedge\frac{\rho_{2,t}}{\eta_{1,t}}\right)\left(\frac{\rho_{1,t}}{\eta_{1,t}}\wedge\frac{\rho_{2,t}}{\eta_{2,t}}\right)}\\
    &\lesssim\sqrt{\frac{N^{\xi}\rho_{1,t}\rho_{2,t}}{\eta_{1,t}\eta_{2,t}}}\cdot\frac{\rho_{1,t}\wedge\rho_{2,t}}{\eta_{1,t}}.
\end{align*}
Moreover, from the definition of $\tau$ we have
\begin{align*}
    \left|\Tr{G_{1,t}FG_{2,t}E_{\nu}}\right|&\leq\phi^{iso}_{1,t}(w_{1},w_{2})+N^{\xi}\psi^{av}_{2,t}(w_{1},F,w_{2},E_{\nu})\\
    &\lesssim\phi^{iso}_{1,t}(w_{1},w_{2})+\frac{N^{\xi}}{\sqrt{N(l_{1,t}\wedge l_{2,t})}}\cdot\sqrt{\frac{\phi^{av}_{2,t}(w_{1},w_{2})}{N\eta_{1,t}\eta_{2,t}}}\\
    &=\frac{\eta_{1}\phi^{iso}_{2,t}(w_{1},w_{2})}{\rho_{1,t}\wedge\rho_{2,t}}+\frac{N^{\xi}}{\sqrt{N(l_{1,t}\wedge l_{2,t})}}\cdot\sqrt{\frac{\phi^{av}_{2,t}(w_{1},w_{2})}{N\eta_{1,t}\eta_{2,t}}}.
\end{align*}
Since we have
\begin{align*}
    \int_{0}^{t}\sqrt{\frac{\rho_{1,s}\rho_{2,s}}{\eta_{1,s}\eta_{2,s}}}\,\mathrm{d}s&\lesssim1,\\
    \int_{0}^{t}\sqrt{\frac{\rho_{1,s}}{N\eta_{1,s}\eta^{2}_{2,s}}}\,\mathrm{d}s&\lesssim\frac{1}{\sqrt{Nl_{2,t}}},
\end{align*}
we obtain
\begin{align*}
    \int_{0}^{t}|A_{4,s}+A_{5,s}|\,\mathrm{d}s&\lesssim N^{\xi/2}\phi^{iso}_{2,t}(w_{1},w_{2}).
\end{align*}

The bound for $A_{6}$ is straightforward:
\begin{align*}
    \Tr{G_{1,t}F\Im G_{2,t}F^{*}G_{1,t}^{*}E_{\nu}}\mbf{x}^{*}G_{1,t}E_{\nu}G_{1,t}^{*}\mbf{x}&\leq\frac{1}{\eta_{1,t}^{2}}\Tr{\Im G_{1,t}F\Im G_{2,t}F^{*}}\mbf{x}^{*}\Im G_{1,t}\mbf{x}\\
    &\lesssim\frac{\rho_{1,t}\phi^{iso}_{2,t}(w_{1},w_{2})}{\eta_{1,t}}.
\end{align*}
After integration, 
\begin{align*}
    \int_{0}^{t}|A_{6,s}|\,\mathrm{d}s&\lesssim\phi^{iso}_{2,t}(w_{1},w_{2}).
\end{align*}

$B_{1}$ and $B_{6}$ are of the same form so we focus on the former. By Cauchy-Schwarz we have
\begin{align*}
    \left|\mbf{x}^{*}G_{1,t}E_{\nu}G_{1,t}^{T}E_{\nu}G_{1,t}F\Im G_{2,t}F^{*}G_{1,t}^{*}\mbf{x}\right|&\leq\left(\mbf{x}^{*}G_{1,t}E_{\nu}G_{1,t}^{T}\bar{G}_{1}E_{\nu}G_{1,t}^{*}\mbf{x}\right)^{1/2}\\
    &\times\left(\mbf{x}^{*}G_{1,t}F\Im G_{2,t}F^{*}G^{*}_{1}G_{1,t}F\Im G_{2,t}F^{*}G_{1,t}^{*}\mbf{x}\right)^{1/2}\\
    &\lesssim\sqrt{\frac{\rho_{1,t}}{\eta_{1,t}^{4}}}\left(\mbf{x}^{*}G_{1,t}F\Im G_{2,t}F^{*}\Im G_{1,t}F\Im G_{2,t}F^{*}G_{1,t}^{*}\mbf{x}\right)^{1/2}\\
    &\lesssim\sqrt{\frac{N\rho_{1,t}\phi^{av}_{2,t}(w_{1},w_{2})\phi^{iso}_{2,t}(w_{1},w_{2})}{\eta_{1,t}^{4}}},
\end{align*}
where the last line follows by a spectral decomposition (as with $A_{1}$). Integrating up to $t$ we find
\begin{align*}
    \int_{0}^{t}|B_{1,s}|\,\mathrm{d}s&\lesssim\phi^{iso}_{2,t}(w_{1},w_{2})\int_{0}^{t}\sqrt{\frac{\rho_{1,s}}{N\eta^{3}_{1,s}}}\,\mathrm{d}s\lesssim\frac{\phi^{iso}_{2,t}(w_{1},w_{2})}{\sqrt{Nl_{1,t}}}.
\end{align*}

Consider $B_{2}$; by Cauchy-Schwarz we have
\begin{align*}
    \left|\mbf{x}^{*}G_{1,t}E_{\nu}G_{2,t}^{T}F^{*}G_{1,t}^{T}E_{\nu}G_{2,t}F^{*}G_{1,t}^{*}\mbf{x}\right|&\leq\left(\frac{1}{\eta_{2}}\mbf{x}^{*}G_{1,t}F\Im G_{2,t}F^{*}G_{1,t}^{*}\mbf{x}\right)^{1/2}\\
    &\times\left(\frac{1}{\eta_{1,t}}\mbf{x}^{*}G_{1,t}E_{\nu}G_{2,t}^{T}F^{*}\Im\bar{G}_{1}F\bar{G}_{2}E_{\nu}G_{1,t}^{*}\mbf{x}\right)^{1/2}.
\end{align*}
We focus on the second factor. Another application of Cauchy-Schwarz gives
\begin{align*}
    \mbf{x}^{*}G_{1,t}E_{\nu}G_{2,t}^{T}F^{*}\Im\bar{G}_{1}F\bar{G}_{2}E_{\nu}G_{1,t}^{*}\mbf{x}&\leq\left(\frac{1}{\eta_{1}}\mbf{x}^{*}\Im G_{1,t}\mbf{x}\right)^{1/2}\\
    &\times\left(\frac{1}{\eta_{2}}\mbf{x}^{*}G_{1,t}E_{\nu}G_{2,t}^{T}F^{*}\Im\bar{G}_{1}F\Im\bar{G}_{2}F^{*}\Im\bar{G}_{1}F\bar{G}_{2}E_{\nu}G_{1,t}^{*}\mbf{x}\right)^{1/2}\\
    &\leq\left(\frac{1}{\eta_{1}}\mbf{x}^{*}\Im G_{1,t}\mbf{x}\right)^{1/2}\\
    &\times\left(\frac{N}{\eta_{2}}\mbf{x}^{*}G_{1,t}E_{\nu}G_{2,t}^{T}F^{*}\Im\bar{G}_{1}F\bar{G}_{2}E_{\nu}G^{*}_{1}\mbf{x}\Tr{\Im G_{1,t}F\Im G_{2,t}F^{*}}\right)^{1/2},
\end{align*}
where the last line follows by a spectral decomosition. Notice that the left hand side appears again in the last line, so we obtain
\begin{align*}
    \mbf{x}^{*}G_{1,t}E_{\nu}G_{2,t}^{T}F^{*}\Im\bar{G}_{1}F\bar{G}_{2}E_{\nu}G_{1,t}^{*}\mbf{x}&\leq\frac{N\mbf{x}^{*}\Im G_{1,t}\mbf{x}\Tr{\Im G_{1,t}F\Im G_{2,t}F^{*}}}{\eta_{1,t}\eta_{2,t}}.
\end{align*}
Thus we have found
\begin{align*}
    \left|\mbf{x}^{*}G_{1,t}E_{\nu}G_{2,t}^{T}F^{*}G_{1,t}^{T}E_{\nu}G_{2,t}F^{*}G_{1,t}^{*}\mbf{x}\right|&\leq\left(\frac{N\mbf{x}^{*}\Im G_{1,t}\mbf{x}\mbf{x}^{*}G_{1,t}F\Im G_{2,t}F^{*}G_{1,t}^{*}\mbf{x}\Tr{\Im G_{1,t}F\Im G_{2,t}F^{*}}}{\eta_{1,t}^{2}\eta_{2,t}^{2}}\right)^{1/2}\\
    &\lesssim\sqrt{\frac{N\rho_{1,t}}{\eta_{1,t}\eta_{2,t}^{2}}}\cdot\phi^{iso}_{2,t}(w_{1},w_{2}),
\end{align*}
and after integration
\begin{align*}
    \int_{0}^{t}|B_{2,s}|\,\mathrm{d}s&\lesssim\phi^{iso}_{2,t}(w_{1},w_{2})\int_{0}^{t}\sqrt{\frac{\rho_{1,s}}{N\eta_{1,s}\eta_{2,s}^{2}}}\,\mathrm{d}s\lesssim\frac{\phi^{iso}_{2,t}(w_{1},w_{2})}{\sqrt{Nl_{2,t}}}.
\end{align*}
$B_{5}$ is of the same form and bounded in an analogous way.

Now consider $B_{3}$; by similar manipulations as applied to $B_{2}$, we find
\begin{align*}
    \left|\mbf{x}^{*}G_{1,t}E_{\nu}\bar{G}_{1}F\Im G_{2,t}^{T}F^{*}G^{T}_{1}E_{\nu}G_{1,t}^{*}\mbf{x}\right|&\leq\frac{N\mbf{x}^{*}\Im G_{1,t}\mbf{x}\Tr{\Im G_{1,t}F\Im G_{2,t}F^{*}}}{\eta_{1,t}^{2}}\\
    &\lesssim\frac{N\rho_{1,t}\phi^{iso}_{2,t}(w_{1},w_{2})}{\eta_{1,t}},
\end{align*}
and upon integration
\begin{align*}
    \int_{0}^{t}|B_{3,s}|\,\mathrm{d}s&\lesssim\phi^{iso}_{2,t}(w_{1},w_{2}).
\end{align*}

For $B_{4}$ we have
\begin{align*}
    \left|\mbf{x}^{*}G_{1,t}FG_{2,t}E_{\nu}G_{2,t}^{T}E_{\nu}G_{2,t}F^{*}G_{1,t}^{*}\mbf{x}\right|&\leq\frac{1}{\eta_{2,t}^{2}}\mbf{x}^{*}G_{1,t}F\Im G_{2,t}F^{*}G_{1,t}^{*}\mbf{x}\\
    &\lesssim\frac{N^{\xi}\phi^{iso}_{2,t}(w_{1},w_{2})}{\eta_{2,t}^{2}},
\end{align*}
and upon integration
\begin{align*}
    \int_{0}^{t}|B_{4,s}|\,\mathrm{d}s&\lesssim\frac{N^{\xi}\phi^{iso}_{2,t}(w_{1},w_{2})}{Nl_{2}}.
\end{align*}
Using $|x+y|^{2}\leq\sqrt{2}(|x|^{2}+|y|^{2})$, we can bound the quadratic variation of the stochastic term by
\begin{align*}
    &\frac{1}{N}\sum_{\nu=\pm}\nu\mbf{x}^{*}G_{1,t}E_{\nu}G_{1,t}^{*}\mbf{x}\mbf{x}^{*}G_{1,t}^{*}F\Im G_{2,t}F^{*}G_{1,t}^{*}E_{\nu}G_{1,t}F\Im G_{2,t}F^{*}G_{1,t}\mbf{x}\\
    &+\frac{1}{N}\sum_{\nu=\pm}\nu\mbf{x}^{*}G_{1,t}E_{\nu}(G^{*}_{1}F\Im G_{2,t}FG^{*}_{1})^{T}\bar{\mbf{x}}\mbf{x}^{T}(G_{1,t}F\Im G_{2,t}F^{*}G_{1,t})^{T}E_{\nu}G_{1,t}\mbf{x}\\
    &+\frac{1}{N}\sum_{\nu=\pm}\nu\mbf{x}^{*}G_{1,t}FG_{2,t}E_{\nu}G_{2,t}^{*}F^{*}G_{1,t}^{*}\mbf{x}\mbf{x}^{*}G_{1,t}^{*}FG_{2,t}^{*}E_{\nu}G_{2,t}F^{*}G_{1,t}\mbf{x}\\
    &+\frac{1}{N}\sum_{\nu=\pm}\nu\mbf{x}^{*}G_{1,t}FG_{2,t}E_{\nu}(G_{1,t}^{*}FG_{2,t}^{*})^{T}\bar{\mbf{x}}\mbf{x}^{T}(G_{2,t}F^{*}G_{1,t})^{T}E_{\nu}G_{2,t}^{*}F^{*}G_{1,t}^{*}\mbf{x}\\
    &+\frac{1}{N}\sum_{\nu=\pm}\nu\mbf{x}^{*}G_{1,t}F\Im G_{2,t}F^{*}G_{1,t}E_{\nu}G_{1,t}^{*}F\Im G_{2,t}F^{*}G_{1,t}^{*}\mbf{x}\mbf{x}^{*}G_{1,t}^{*}E_{\nu}G_{1,t}\mbf{x}\\
    &+\frac{1}{N}\sum_{\nu=\pm}\nu\mbf{x}^{*}G_{1,t}F\Im G_{2,t}F^{*}G_{1,t}E_{\nu}(G_{1,t}^{*})^{T}\bar{\mbf{x}}\mbf{x}^{T}G_{1,t}^{T}E_{\nu}G_{1,t}^{*}F\Im G_{2,t}F^{*}G_{1,t}^{*}\mbf{x}.
\end{align*}
The terms on even numbered lines can be bounded by the term on the preceeding line by Cauchy-Schwarz. The term on the first and fifth lines are of the same form up to switching $G_{1,t}$ and $G^{*}_{1}$. Thus it is enough to bound the terms on the first and third lines. By a spectral decomposition we have
\begin{align*}
    \mbf{x}^{*}G_{1,t}^{*}F\Im G_{2,t}F^{*}G_{1,t}^{*}E_{b}G_{1,t}F\Im G_{2,t}F^{*}G_{1,t}\mbf{x}&\leq N\mbf{x}^{*}G_{1,t}^{*}F\Im G_{2,t}F^{*}G_{1,t}\mbf{x}\Tr{\Im G_{1,t}F\Im G_{2,t}F^{*}},
\end{align*}
and hence the first line is bounded by
\begin{align*}
    \frac{N^{\xi}\rho_{1,t}\phi^{iso}_{2}(w_{1},w_{2})\phi^{av}_{2}(w_{1},w_{2})}{\eta^{2}_{1,t}}&=\frac{N^{\xi}\rho_{1,t}(\phi^{iso}_{2}(w_{1},w_{2}))^{2}}{\eta_{1,t}}.
\end{align*}
The term on the third line is bounded by
\begin{align*}
    \frac{N^{2\xi}(\phi^{iso}_{2})^{2}}{N\eta_{2,t}^{2}}.
\end{align*}
By the BDG inequality we find that the integral of the stochastic term is bounded by
\begin{align*}
    \sqrt{1+\frac{N^{\xi}}{N\eta_{2}\rho_{2}}}N^{\xi/2}\phi^{iso}_{2}(w_{1},w_{2})&\lesssim N^{\xi/2}\phi^{iso}_{2}(w_{1},w_{2}).
\end{align*}
This completes the bounds for each term in the SDE and hence the proof.
\end{proof}

Combining each of these lemmas we can prove Proposition \ref{prop:gaussLL}.
\begin{proof}[Proof of Proposition \ref{prop:gaussLL}]
From Lemmas \ref{lem:av1Gaussian}--\ref{lem:iso2Gaussian} we deduce that, for any fixed $T>0$, if $R^{av/iso}_{0}<N^{\xi/2}\psi^{av/iso}_{0}$, then $\tau>T$ with very high probability. For the input bounds we use the fact that $|z_{0}|>1+\delta$ for some fixed $\delta>0$ depending on $T$, so with very high probability there is a neighbourhood of 0 containing no eigenvalues of $W_{z_{0}}$. Choosing $\epsilon$ sufficiently small in the definition of $\Omega_{K,\xi,\epsilon}$ and small (but fixed) $T>0$, we can ensure that the parameters $w_{1},w_{2}$ from the initial condition $S^{-1}_{T}(\Omega_{K,\xi,\epsilon})$ are inside this gap. In other words, there is a finite, $N$-independent distance from the spectral parameters $w_{j}$ to the spectrum, meaning that we are in the global regime at $t=0$. In this regime the proof of the local laws by cumulant expansion is much simpler because the trivial bound $\|G_{z}(w)\|\leq\frac{1}{d(w,\text{spec}(W))}$ is affordable, as explained in \cite[Appendix B]{cipolloni_optimal_2022}.
\end{proof}

\subsection{Removal of Gaussian Component}
To remove the Gaussian component we argue along the same lines as \cite[Appendix D]{cipolloni_universality_2024} by bounding the time derivative of high moments and using Gr\"{o}nwall's inequality. Let $X_{t}=e^{-t/2}X+\sqrt{1-e^{-t}}Y$, where $Y\sim Gin_{1}(N)$, and let $W_{t}$ denote the corresponding Hermitisation. For $w_{j}\in\mbb{C}$, let $G_{j}:=G_{z}(w_{j})$. Note that $G_{j}$ depends on $t$ only through $X_{t}$, i.e. the spectral parameters $z$ and $w_{j}$ are time-independent. Throughout this section we assume that $|z|\geq 1-CN^{-1/2}$ for some $C>0$. We consider the time evolution of 
\begin{align}
    R^{av}_{t}(w_{1},B_{1},...,w_{k},B_{k})&:=\Tr{(G_{1}B_{1}\cdots G_{k}B_{k}-M_{z}(w_{1},B_{1},...,w_{k}))B_{k}},\\
    R^{iso}_{t,\mu,\nu}(w_{1},B_{1},...,w_{k})&:=\mbf{e}_{\mu}^{*}(G_{1}B_{1}\cdots G_{k}-M_{z}(w_{1},B_{1},...,w_{k}))\mbf{e}_{\nu}.
\end{align}
We have to prove the local law for isotropic quantities first since this will be used as an input for the proof of the averaged law. In the course of the proof other matrix elements appear besides $(\mu,\nu)$ and so we define
\begin{align}
    \Omega_{m,t}(w_{1},B_{1},...,w_{k})&:=\max_{\mu,\nu\in[2N]}\mbb{E}|R^{iso}_{t,\mu,\nu}|^{m}.
\end{align}
Here and in the sequel we use the following convention on matrix indices: Latin indices take values in $[N]$ and Greek indices take values in $[2N]$. We also define the conjugate index $\hat{\mu}\equiv N+\mu\text{ mod } 2N$.

We begin with the requisite input bounds. Let $G_{\mu,\nu}:=(G_{z}(w))_{\mu,\nu}$. From the isotropic single resolvent local law in Proposition \ref{prop:singleLL} we have
\begin{align}
    |G_{\mu,\nu}|&\prec\rho\delta_{\mu,\nu}+\delta_{\mu,\hat{\nu}}+\sqrt{\frac{\rho}{N\eta}}.\label{eq:isotropicBound}
\end{align}
In addition to this bound, we need the following simple case of fluctuation averaging.
\begin{lemma}\label{lem:fluctuationAveraging}
For any $\mu\in[1,2N]$ we have
\begin{align}
    \left|\sum_{j}G_{\mu,j}\right|&\prec\frac{1}{\eta\rho},\label{eq:fluctuation1}\\
    \left|\sum_{j}G_{\mu,\hat{\jmath}}\right|&\prec\frac{1}{\eta\rho}.\label{eq:fluctuation2}
\end{align}
\end{lemma}
\begin{proof}
We will write the details for $\mu=i\in[N]$; the case $\mu=\hat{\imath}\in N+[N]$ is analogous. Let $P_{j}$ and $P_{\hat{\jmath}}$ denote the expectation with respect to the $j$-th row and column of $X$ respectively, and let $Q_{\mu}=1-P_{\mu}$. Let $G^{(\mu)}$ denote the resolvent of the matrix obtained from $W_{z}$ by removing the $\mu$-th row and column. Thus $G^{(j)}$ and $G^{(\hat{\jmath})}$ are independent of the $j$-th row and column of $X$ respectively, which means in particular that $P_{j}G^{(j)}=G^{(j)}$ and $P_{\hat{\jmath}}G^{(\hat{\jmath})}=G^{(\hat{\jmath})}$. Let
\begin{align*}
    A&:=\sum_{j}^{(i)}P_{j}G_{i,j},\\
    B&:=\sum_{j}^{(i)}Q_{j}G_{i,j},\\
    C&:=\sum_{j}^{(i)}P_{\hat{\jmath}}G_{i,\hat{\jmath}},\\
    D&:=\sum_{j}^{(i)}Q_{\hat{\jmath}}G_{i,\hat{\jmath}},
\end{align*}
where $\sum^{(\mu)}_{\nu}$ denotes the sum over $\nu\neq \mu$. We claim that
\begin{align}
    |B|,|D|&\prec\frac{1}{\eta},\label{eq:fluctuationClaim1}
\end{align}
and
\begin{align}
    A&=\bar{z}m(C+D)+|z|^{2}u(A+B)+O_{\prec}\left(\frac{1}{\eta}\right),\label{eq:fluctuationClaim2}\\
    C&=zm(A+B)+|z|^{2}u(C+D)+O_{\prec}\left(\frac{1}{\eta}\right),\label{eq:fluctuationClaim3}
\end{align}
where we recall that $u=m/(m+w)$. From this we can finish the proof as follows. Using \eqref{eq:fluctuationClaim1} we rewrite \eqref{eq:fluctuationClaim2} and \eqref{eq:fluctuationClaim3} as
\begin{align*}
    C&=\frac{1-|z|^{2}u}{\bar{z}m}A+O_{\prec}\left(\frac{1}{\eta\rho}\right)=-\frac{m+w}{\bar{z}}A+O_{\prec}\left(\frac{1}{\eta\rho}\right),\\
    A&=\frac{1-|z|^{2}u}{zm}C+O_{\prec}\left(\frac{1}{\eta\rho}\right)=-\frac{m+w}{z}C+O_{\prec}\left(\frac{1}{\eta\rho}\right),
\end{align*}
where we have used the identity $1-|z|^{2}u=-m(m+w)$ which follows from the cubic equation for $m$. Now multiply the second equation by $-\frac{m+w}{\bar{z}}$ and add the result to the first to obtain
\begin{align*}
    C-\frac{m+w}{\bar{z}}A&=-\frac{m+w}{\bar{z}}A+\frac{(m+w)^{2}}{|z|^{2}}C+O_{\prec}\left(\frac{1}{\eta\rho}\right),
\end{align*}
which implies that $C=O_{\prec}\left(\frac{1}{\eta\rho}\right)$ since $|z|\gtrsim 1$ and $|m|\lesssim |w|^{1/3}$. The proof for $A$ follows in the same way.

Let us now prove \eqref{eq:fluctuationClaim2} assuming \eqref{eq:fluctuationClaim1}. We use the isotropic law $|G_{i,j}-m\delta_{i,j}|\prec\sqrt{\frac{\rho}{N\eta}}$ to insert a factor of $m/G_{j,j}$:
\begin{align*}
    \sum_{j}^{(i)}P_{j}G_{i,j}&=m\sum_{j}^{(i)}P_{j}\frac{G_{i,j}}{G_{j,j}}+O_{\prec}\left(\frac{1}{\eta}\right).
\end{align*}
Now we use the identity
\begin{align*}
    \frac{G_{i,j}}{G_{j,j}}&=-\sum_{\mu}^{(j)}G_{i,\mu}^{(j)}(W_{z})_{\mu,j}\\
    &=-\sum_{k}G^{(j)}_{i,\hat{k}}(\bar{X}_{j,k}-\bar{z}\delta_{j,k}).
\end{align*}
Taking the partial expectation with respect to the $j$-th row we find
\begin{align*}
    P_{j}\frac{G_{i,j}}{G_{j,j}}&=\bar{z}G^{(j)}_{i,\hat{\jmath}}.
\end{align*}
Now we use
\begin{align*}
    G^{(j)}_{i,\hat{\jmath}}&=G_{i,\hat{\jmath}}-\frac{G_{i,j}G_{j,\hat{\jmath}}}{G_{j,j}}\\
    &=G_{i,\hat{\jmath}}+\frac{zu}{m}G_{i,j}+O_{\prec}\left(\frac{1}{N\eta}\right).
\end{align*}
Multiplying by $\bar{z}m$ and summing over $j$ we have
\begin{align*}
    A&=\bar{z}m\sum_{j}^{(i)}G_{i,\hat{\jmath}}+|z|^{2}u\sum_{j}^{(i)}G_{i,j}+O_{\prec}\left(\frac{1}{\eta}\right)\\
    &=\bar{z}m(C+D)+|z|^{2}u(A+B)+O_{\prec}\left(\frac{1}{\eta}\right),
\end{align*}
which is \eqref{eq:fluctuationClaim2}. Repeating the same steps for $C$ we obtain \eqref{eq:fluctuationClaim3}.

For $B$ and $D$ we can follow \cite[Proof of Proposition 6.1]{benaych-georges_lectures_2016} or \cite[Proof of Theorem 4.8]{erdos_averaging_2013}. The idea is to bound even moments by making repeated use of the identities
\begin{align}
    G_{i,j}&=G_{i,j}^{(k)}+\frac{G_{i,k}G_{k,j}}{G_{k,k}},\label{eq:expansion1}\\
    \frac{1}{G_{i,i}}&=\frac{1}{G^{(j)}_{i,i}}-\frac{G_{i,j}G_{j,i}}{G_{i,i}G^{(j)}_{i,i}G_{j,j}},\label{eq:expansion2}
\end{align}
for $i\neq j$ and $k\neq i,j$, and the fact that $P_{k}\left(G^{(k)}_{i,j}Q_{k}(\cdot)\right)=P_{k}Q_{k}\left(G^{(k)}_{i,j}\cdot\right)=0$. For any $m\in\mbb{N}$ we have
\begin{align*}
    \mbb{E}|B|^{2m}&=\mbb{E}\sum_{j_{1},...,j_{m}}\sum_{k_{1},...,k_{m}}Q_{j_{1}}(G_{i,j_{1}})\cdots Q_{j_{m}}(G_{i,j_{m}})Q_{k_{1}}(\bar{G}_{i,k_{1}})\cdots Q_{k_{m}}(\bar{G}_{i,k_{m}}).
\end{align*}
The complex conjugation plays no role in the remainder and so for simplicity we consider
\begin{align*}
    \mbb{E}\sum_{j_{1},...,j_{2m}}Q_{j_{1}}(G_{i,j_{1}})\cdots Q_{j_{2m}}(G_{i,j_{2m}}).
\end{align*}
We split the sum according to the number $d$ of distinct indices:
\begin{align*}
    \mbb{E}\sum_{d=1}^{2m}\sum_{l}\sum_{j_{1},...,j_{d}}'\left(Q_{j_{1}}(G_{i,j_{1}})\right)^{l_{1}+1}\cdots\left(Q_{j_{d}}(G_{i,j_{d}})\right)^{l_{d}+1}.
\end{align*}
Here the innermost sum is over distinct index tuples and $l$ ranges over partitions of $2m-d$ of length $d$, i.e. $l_{j}\geq0,\,j=1,...,d$ and $l_{1}+\cdots+l_{d}=2m-d$. When $d\leq m$, the result follows immediately from the isotropic law:
\begin{align*}
    \sum_{d=1}^{m}\sum_{l}\sum_{j_{1},...,j_{d}}'\left(Q_{j_{1}}(G_{i,j_{1}})\right)^{l_{1}+1}\cdots\left(Q_{j_{d}}(G_{i,j_{d}})\right)^{l_{d}+1}&\prec N^{d}\cdot\left(\frac{\rho}{N\eta}\right)^{m}\\
    &=\frac{1}{N^{m-d}}\cdot\left(\frac{\rho}{\eta}\right)^{m}.
\end{align*}
When $d>m$, let $I=\{j:l_{j}=0\}$. We use \eqref{eq:expansion1} or \eqref{eq:expansion2} a total of $|I|$ times for each index in $I$. At each step we generate new terms by choosing the first or second term from the right hand sides of \eqref{eq:expansion1} and \eqref{eq:expansion2}. Since $\mbb{E}\left(Q_{j}(G_{i,j})Y\right)=0$ if $P_{j}Y=Y$, any non-zero term must have been obtained by choosing the second terms on the right hand sides of \eqref{eq:expansion1} and \eqref{eq:expansion2} at least once for each of the $|I|$ steps in this procedure. The former increases the number of off-diagonal and (reciprocals of) diagonal elements by 1, while the latter increases this number by 2. The largest terms correspond to those with the minimum number of additional off-diagonal elements, which is equal to $|I|$. Such a term has $2m+|I|$ off-diagonal elements and $|I|$ recicprocals of diagonal elements and is therefore bounded by
\begin{align*}
    \frac{\rho^{m-|I|/2}}{(N\eta)^{m+|I|/2}}.
\end{align*}
Summing over all terms corresponding to a given partition, all tuples $(j_{1},...,j_{d})\in[1,N]$ of distinct elements and all partitions of length $d$, we obtain the bound
\begin{align*}
    \frac{1}{(N\eta)^{m-d+|I|/2}}\cdot\frac{\rho^{2m-d}}{\eta^{d}}.
\end{align*}
Since $|I|\geq2(d-m)$, this takes its maximum value $\eta^{-2m}$ at $d=2m$. This proves \eqref{eq:fluctuationClaim1} for $B$, and the proof for $D$ is analogous.
\end{proof}

The starting point for the GFT is the formula
\begin{align}
    \frac{\mathrm{d}}{\mathrm{d}t}\mbb{E}|R_{t}|^{2m}&=\sum_{p=2}^{P-1}\frac{\kappa_{p+1}}{p!N^{\frac{p+1}{2}}}\sum_{i,j=1}^{N}\mbb{E}\left[\partial_{ij}^{p+1}(R_{t}^{m}\bar{R}_{t}^{m})\right]+O(N^{-D}),\label{eq:dR}
\end{align}
which follows by It\^{o}'s lemma and a cumulant expansion. We have truncated the sum at some large $P$ (depending on $m$) so that the error is bounded by $N^{2-P/2}\eta^{-4m}\leq N^{2+4m-P/2}\leq N^{-D}$ (the term $\eta^{-4m}$ comes from using the single resolvent isotropic local law to bound the extra factors of elements of $G$ that arise from the derivatives). Strictly speaking, since we assume independent but not necessarily identically distributed entries, the cumulants $\kappa_{p}$ of $\sqrt{N}X_{i,j}$ should depend on the indices $i,j$, but this dependence is irrelevant since we assume a uniform upper bound for all cumulants.
\begin{lemma}\label{lem:isoGFT1}
For any $w_{j}\in\mbb{C},\,j=1,2$, $\mu,\nu\in[1,...,2N]$, $t\geq0$ and $\xi>0$, we have
\begin{align}
    \left|\frac{\mathrm{d}}{\mathrm{d}t}\mbb{E}|R|^{2m}\right|&\leq\frac{N^{\xi}}{N^{1/4}\eta^{3/4}}\left[\left(\frac{\phi^{iso}_{1}}{(N\eta)^{1/4}}\right)^{2m}+\frac{\Omega_{t,2m}}{(N\eta)^{m/2}}+|R|^{2m}\right],
\end{align}
with very high probability, where $R:=R^{iso}_{t,\mu,\nu}(w_{1},F,w_{2})$ and $\eta=|\eta_{1}|\wedge|\eta_{2}|$.
\end{lemma}
\begin{proof}
The $p$-th derivative of $R$ is a sums of products of $p$ elements of $G_{j,t}$ and an element of $G_{1,t}FG_{2,t}$. For example,
\begin{align*}
    \partial_{ij}R&=-(G_{1,t})_{\mu,i}(G_{1,t}FG_{2,t})_{\hat{\jmath},\nu}-(G_{1,t})_{\mu,\hat{\jmath}}(G_{1,t}FG_{2,t})_{i,\nu}\\
    &-(G_{1,t}FG_{2,t})_{\mu,i}(G_{2,t})_{\hat{\jmath},\nu}-(G_{1,t}FG_{2,t})_{\mu,\hat{\jmath}}(G_{2,t})_{i,\nu},\\
    \partial_{ij}^{2}R&=(G_{1,t})_{\mu,i}(G_{1,t})_{\hat{\jmath},i}(G_{1,t}FG_{2,t})_{\hat{\jmath},\nu}+(G_{1,t})_{\mu,i}(G_{1,t})_{\hat{\jmath},\hat{\jmath}}(G_{1,t}FG_{2,t})_{i,\nu}+\cdots.
\end{align*}
To generate a term in the $p$-th derivative, we choose $l=0,...,p$, write down a string containing $p-l$ factors of $G_{1,t}$, one factor of $G_{1,t}FG_{2,t}$ and $l$ factors $G_{2,t}$ in that order and assign indices $i$ and $\hat{\jmath}$ such that: i) the first index of the first factor is $\mu$; ii) the second index of the last factor is $\nu$; iii) the second index of a given factor is different to first index of the next (for example, we cannot have a factor $(G_{1,t})_{\alpha,i}(G_{1,t})_{i,\beta}$). The important point to note is that for each copy of $R$ on which a derivative acts we always have at least one $(G_{1,t})_{\mu,\alpha}$ or $(G_{2,t})_{\alpha,\nu}$ and exactly one entry of $G_{1,t}FG_{2,t}$. Using the bound for off-diagonal elements in \eqref{eq:isotropicBound}, we obtain a general estimate
\begin{align*}
    \left|\mbb{E}\sum_{i,j}\sum_{q}\partial_{ij}^{q_{1}}(R)\cdots\partial_{ij}^{q_{k}}(R)\cdot R^{m-k}\bar{R}^{m}\right|&\leq N^{2}\cdot\left(\frac{\rho}{N\eta}\right)^{k}\cdot\max_{\alpha,\beta}\mbb{E}\left(\phi^{iso}_{1}+|R^{iso}_{t,\alpha,\beta}|\right)^{k}\cdot|R|^{2m-k},
\end{align*}
where the sum over $q$ on the left hand side is over partitions of $p+1$ of length $k$. This is enough as soon as $p>2$ or $k>1$. We are left with the terms corresponding to $p=3$ and $k=1$. By using \eqref{eq:isotropicBound} for off-diagonal elements, we deduce that the most critical term is that for which we have the maximum number of diagonal elements of $G_{j,t}$, i.e. terms of the form
\begin{align*}
    \frac{1}{N^{3/2}}\mbb{E}\sum_{i,j=1}^{N}(G_{1,t})_{\mu,i}(G_{1,t})_{\hat{\jmath},\hat{\jmath}}(G_{1,t})_{i,i}(G_{1,t}FG_{2,t})_{\hat{\jmath},\nu}R^{m-1}\bar{R}^{m}.
\end{align*}
For these terms we use the fluctuation averaging result in Lemma \ref{lem:fluctuationAveraging}. In the above example we have the bound
\begin{align*}
    \frac{N^{\xi}}{N^{3/2}\eta}\sum_{j}\mbb{E}\left(\phi^{iso}_{1}+|R^{iso}_{t,\hat{\jmath},\nu}|\right)|R|^{2m-1}&\leq\frac{N^{\xi}}{N^{1/4}\eta^{3/4}}\left[\left(\frac{\phi^{iso}_{1}}{(N\eta)^{1/4}}\right)^{2m}+\frac{\Omega_{t,2m}}{(N\eta)^{m/2}}+|R|^{2m}\right],
\end{align*}
with very high probability.
\end{proof}

With this lemma in hand we can conclude the proof of the first isotropic law in \eqref{eq:isoLL1}.
\begin{proof}[Proof of \eqref{eq:isoLL1}]
Using Lemma \ref{lem:isoGFT1} and Gr\"{o}nwall's inequality, we conclude that we can remove a Gaussian component of variance $t=O(1)$ for $w_{j}$ such that $\eta=\eta_{1}\wedge\eta_{2}\geq N^{-1/3+\xi}$. In other words, for sufficiently small $T>0$, the local law in \eqref{eq:isoLL1} holds in the domain
\begin{align}
    \bigcup_{t\in[0,T]}\left(\Omega_{K,\xi,\epsilon}(t)\cap\{(z,w_{1},w_{2})\in\mbb{C}^{3}:|\eta_{1}|\wedge|\eta_{2}|\geq N^{-1/3+\xi}\}\right).\label{eq:domain}
\end{align}

For $\eta<N^{-1/3+\xi}$, we let $t=N^{1/4}\eta^{3/4}$ and observe that the initial condition $\eta_{j,-t}$ satisfies
\begin{align*}
    \eta_{j,-t}&\simeq \eta_{j}+t\rho_{j}\gtrsim t\eta^{1/3}_{j}.
\end{align*}
If $\eta\geq N^{-7/13+\xi}$ then $\eta_{j,-t}\gtrsim N^{-1/3+\xi}$ and the initial condition belongs to \eqref{eq:domain}. We can now apply Lemma \ref{lem:iso1Gaussian} to propagate this law to $\eta$ for matrices with a Gaussian component of variance $t$. Applying Lemma \ref{lem:isoGFT1} again we remove this Gaussian component and prove the local law for any $\eta$ with $\eta\geq N^{-7/13}$. Repeating this a finite number of times we can redue $\eta$ to the threshold $N\eta\rho\geq N^{\xi}$. The fact that we can reach the threshold follows because at each stage we go from $N^{-x_{n}}$ to $N^{-x_{n+1}}$, where
\begin{align*}
    x_{n+1}&=\frac{x_{n}+1/4}{3/4+1/3}.
\end{align*}
The fixed point of this iteration is $x=3>1$, so we can reach $N^{-1+\xi}$ in a finite number of steps.
\end{proof}

For the second isotropic law in \eqref{eq:isoLL2}, we use \eqref{eq:isoLL1} as an additional input to prove the following.
\begin{lemma}\label{lem:isoGFT2}
For any $w_{j}\in\mbb{C},\,j=1,2$, $\mu\in[2N]$, $\xi>0$ and $t\geq0$, we have
\begin{align}
    \left|\frac{\mathrm{d}}{\mathrm{d}t}\mbb{E}|S|^{2m}\right|&\leq\frac{N^{\xi}}{N^{1/2}\eta}\left[(\phi^{iso}_{2})^{2m}+|S|^{2m}\right],
\end{align}
with very high probability, where $S:=S^{iso}_{t,\mu,\mu}(w_{1},F,\wh{w}_{2},F^{*},\bar{w}_{1})$ and $\eta=\eta_{1}\wedge\eta_{2}$.
\end{lemma}
\begin{proof}
The proof is very similar to the proof of Lemma \ref{lem:isoGFT1}. The only difference is that there are two types of term in a general derivative:
\begin{align*}
    &G_{1}^{q}\cdot G_{1}F\Im G_{2}F^{*}G^{*}_{1}\cdot (G^{*}_{1})^{r},\qquad q+r=p,\\
    &G_{1}^{q}\cdot G_{1}FG_{2}\cdot G_{2}^{1+r}\cdot G_{2}F^{*}G^{*}_{1}\cdot (G^{*}_{1})^{s},\qquad q+r+s=p-1,
\end{align*}
where $A^{q}$ represents $q$ factors of matrix elements of $A$. The first type can be bounded as in the proof of Lemma \ref{lem:isoGFT1}. For the second type we use \eqref{eq:isoLL1} and the bounds
\begin{align*}
    \phi^{iso}_{1}(w_{1},w_{2})&\lesssim\frac{\eta_{1}}{\rho_{1}\wedge\rho_{2}}\phi^{iso}_{2}(w_{1},w_{2}),\\
    \phi^{iso}_{1}(w_{1},w_{2})&\lesssim(\rho_{1}\vee\rho_{2})\left(\frac{\rho_{1}}{\eta_{1}}\wedge\frac{\rho_{2}}{\eta_{2}}\right).
\end{align*}
Again, the most critical terms occur when $p=2$ and the derivative acts on a single copy of $S:=S^{iso}_{t,\mu,\mu}$, and in particular when we have a diagonal element $(G_{2})_{\nu,\nu}$, e.g.
\begin{align*}
    \frac{1}{N^{3/2}}\sum_{i,j}(G_{1})_{\mu,i}(G_{1}FG_{2})_{\hat{\jmath},\hat{\jmath}}(G_{2})_{i,i}(G_{2}F^{*}G^{*}_{1})_{\hat{\jmath},\mu}S^{m-1}\bar{S}^{m}.
\end{align*}
Using Cauchy-Schwarz for the sum over $j$ we find that this is bounded by
\begin{align*}
    \sqrt{\frac{\rho_{1}}{N\eta_{1}\eta_{2}}}\phi^{iso}_{1}|S|^{2m-1/2}&\lesssim\sqrt{\frac{\rho_{1}(\rho_{1}\vee\rho_{2})}{N\eta_{2}(\rho_{1}\wedge\rho_{2})}\left(\frac{\rho_{1}}{\eta_{1}}\wedge\frac{\rho_{2}}{\eta_{2}}\right)}\sqrt{\phi^{iso}_{2}}|S|^{2m-1/2}\\
    &\lesssim\frac{1}{N^{1/2}\eta}\left[(\phi^{iso}_{2})^{2m}+|S|^{2m}\right].
\end{align*}
\end{proof}
Using this lemma the proof of \eqref{eq:isoLL2} is concluded in the same way as the proof of \eqref{eq:isoLL1}. With \eqref{eq:isoLL1} and \eqref{eq:isoLL2} as additional input, the proof of the averaged law in \eqref{eq:avLL2} can be carried out in the same way using the following lemma.
\begin{lemma}\label{lem:av2GFT}
For any $w_{j}\in\mbb{C}$, $t\geq0$ and $\xi>0$ we have
\begin{align}
    \left|\frac{\mathrm{d}}{\mathrm{d}t}\mbb{E}|R|^{2m}\right|&\leq\frac{N^{\xi}}{N^{1/2}\eta}\left[\left(\frac{\phi^{av}_{2}(w_{1},w_{2})}{\sqrt{N\eta}}\right)^{2m}+|R|^{2m}\right],
\end{align}
with very high probability, where $R:=R^{av}_{t}(\wh{w}_{1},F,\wh{w}_{2},F^{*})$ and $\eta=\eta_{1}\wedge\eta_{2}$.
\end{lemma}
\begin{proof}
Derivatives of $R^{av}_{t}(\wh{w}_{1},F,\wh{w}_{2},F^{*})$ generate terms of the form
\begin{align*}
    &\frac{1}{N}G_{1}^{p}\cdot G_{1}F\Im G_{2}F^{*}G_{1},\\
    &\frac{1}{N}G_{1}^{p}\cdot G_{1}FG_{2}\cdot G_{2}^{q}\cdot G_{2}F^{*}G_{1},\\
    &\frac{1}{N}G_{2}F^{*}\Im G_{1}FG_{2}\cdot G_{2}^{p},
\end{align*}
where each factor is understood to be a matrix element. Note that each copy of $R^{av}_{t}$ on which a derivative acts contributes a factor $N^{-1}$. By \eqref{eq:isoLL1} we have $(G_{1}FG_{2})_{\mu,\nu}\prec\phi^{iso}_{1}$ and by \eqref{eq:isoLL2} we have
\begin{align}
    |(G_{1}F\Im G_{2}F^{*}G_{1})_{\mu,\nu}|&\leq\sqrt{(G_{1}F\Im G_{2}F^{*}G_{1}^{*})_{\mu,\mu}}\cdot\sqrt{(G_{1}^{*}F\Im G_{2}F^{*}G_{1})_{\nu,\nu}}\nonumber\\
    &\prec\phi^{iso}_{2}(w_{1},w_{2})\nonumber\\
    &=\frac{\phi^{av}_{2}(w_{1},w_{2})}{\eta_{1}}.\label{eq:av2GFT1}
\end{align}

For the second kind of term, when both $G_{1}FG_{2}$ and $G_{2}F^{*}G_{1}$ are diagonal elements, we use
\begin{align}
    |(G_{1}FG_{2})_{\mu,\mu}||(G_{2}F^{*}G_{1})_{\nu,\nu}|&\prec(\phi^{iso}_{1})^{2}\nonumber\\
    &\leq\frac{\rho_{1}\vee\rho_{2}}{\rho_{1}\wedge\rho_{2}}\left(\frac{\rho_{1}}{\eta_{1}}\wedge\frac{\rho_{2}}{\eta_{2}}\right)\cdot\phi^{av}_{2}\nonumber\\
    &\leq\frac{\phi^{av}_{2}}{\eta}.\label{eq:av2GFT2}
\end{align}
When one of $G_{1}FG_{2}$ and $G_{2}F^{*}G_{1}$ is an off-diagonal element, we use Cauchy-Schwarz and
\begin{align*}
    \sum_{i,j}|(G_{1}FG_{2})_{i,\hat{\jmath}}|^{2}&\leq\frac{N\Tr{\Im G_{1}F\Im G_{2}F^{*}}}{\eta_{1}\eta_{2}}\\
    &\lesssim\frac{N}{\eta_{1}\eta_{2}}\left(\phi^{av}_{2}+R^{av}_{t}(\wh{w}_{1},F,\wh{w}_{2},F^{*})\right).
\end{align*}
Each extra derivative comes with a factor $N^{-1/2}$ and each extra copy of $R^{av}_{t}$ on which a derivative acts contributes a factor $N^{-1}$. Therefore the most critical terms occur when three derivatives act on the same copy of $R^{av}_{t}$. By Cauchy-Schwarz we have
\begin{align*}
    \left|\frac{1}{N^{5/2}}\sum_{i,j}(G_{1})_{i,i}(G_{1})_{\hat{\jmath},\hat{\jmath}}(G_{1}F\Im G_{2}F^{*}G_{1})_{i,\hat{\jmath}}R^{m-1}\bar{R}^{m}\right|&\prec\frac{|R|^{2m}}{N^{1/2}\eta_{1}},
\end{align*}
and
\begin{align*}
    &\left|\frac{1}{N^{5/2}}\sum_{i,j}(G_{1})_{i,i}(G_{1}FG_{2})_{\hat{\jmath},\hat{\jmath}}(G_{2}F^{*}G_{1})_{i,\hat{\jmath}}R^{m-1}\bar{R}^{m}\right|\\
    &\prec\frac{\rho_{1}\phi^{iso}_{1}}{N\sqrt{\eta_{1}\eta_{2}}}\sqrt{\phi^{av}_{2}+|R|}|R|^{2m-1}\\
    &\lesssim\frac{\rho_{1}}{N\sqrt{\eta_{1}\eta_{2}}}\sqrt{\frac{\rho_{1}\vee\rho_{2}}{\rho_{1}\wedge\rho_{2}}\left(\frac{\rho_{1}}{\eta_{1}}\wedge\frac{\rho_{2}}{\eta_{2}}\right)}\sqrt{\phi^{av}_{2}}\sqrt{\phi^{av}_{2}+|R|}|R|^{2m-1}\\
    &\lesssim\frac{1}{N^{1/2}\eta}\left[\left(\frac{\phi^{av}_{2}}{\sqrt{N\eta}}\right)^{2m}+|R|^{2m}\right].
\end{align*}
In all other cases the required bounds follow directly from \eqref{eq:isotropicBound}, \eqref{eq:av2GFT1} and \eqref{eq:av2GFT2}.
\end{proof}

Finally, using \eqref{eq:isoLL1}, \eqref{eq:isoLL2} and \eqref{eq:avLL2}, we can prove the following lemma and ultimately \eqref{eq:avLL1}.
\begin{lemma}
For any $w_{j}\in\mbb{C}$, $t\geq0$ and $\xi>0$ we have
\begin{align}
    \left|\frac{\mathrm{d}}{\mathrm{d}t}\mbb{E}|R|^{2m}\right|&\leq\frac{N^{\xi}}{N^{1/2}\eta}\left[\left(\phi^{iso}_{1}(w_{1},w_{2})+\sqrt{\frac{\phi^{av}_{2}(w_{1},w_{2})}{N\eta_{1}\eta_{2}}}\right)^{2m}+|R|^{2m}\right],
\end{align}
with very high probability, where $R:=R^{av}_{t}(w_{1},F,w_{2})$ and $\eta=\eta_{1}\wedge\eta_{2}$.
\end{lemma}
\begin{proof}
The proof is essentially the same as the proof of Lemma \ref{lem:av2GFT}. Taking derivatives of $R^{av}_{t}(w_{1},F,w_{2})$ generates terms
\begin{align*}
    &G_{1}^{p}\cdot G_{1}FG_{2}\cdot G_{2}^{q}\cdot G_{1}G_{2},\\
    &G_{1}^{p}\cdot G_{1}FG_{2}G_{1},\\
    &G_{2}G_{1}FG_{2}\cdot G_{2}^{p}.
\end{align*}
We have $(G_{1}FG_{2})_{\mu,\nu}\prec\phi^{iso}_{1}$ by \eqref{eq:isoLL1}, and
\begin{align*}
    |(G_{1}FG_{2}G_{1})_{\mu,\nu}|&\leq\sqrt{\frac{(G_{1}F\Im G_{2}F^{*}G_{1}^{*})_{\mu,\mu}}{\eta_{1}}}\cdot\sqrt{\frac{\Im (G_{1})_{\nu,\nu}}{\eta_{1}}}\\
    &\prec\sqrt{\frac{\rho_{1}\phi^{av}_{2}(w_{1},w_{2})}{\eta_{1}^{2}\eta_{2}}},
\end{align*}
by \eqref{eq:isoLL2} and \eqref{eq:isotropicBound}. For off-diagonal elements of $G_{j}$ we use \eqref{eq:isotropicBound}, for off-diagonal elements of $G_{1}FG_{2}$ we use Cauchy-Schwarz and
\begin{align*}
    \sum_{i,j}|(G_{1}FG_{2})_{i,\hat{\jmath}}|^{2}&=\frac{N\Tr{\Im G_{1}F\Im G_{2}F^{*}}}{\eta_{1}\eta_{2}}\\
    &\prec\frac{N\phi^{av}_{2}(w_{1},w_{2})}{\eta_{1}\eta_{2}}.
\end{align*}
and for off-diagonal elements of $G_{1}FG_{2}G_{1}$ we use 
\begin{align*}
    \sum_{i,j}|(G_{1}FG_{2}G_{1})_{i,\hat{\jmath}}|^{2}&=\frac{N}{\eta^{2}_{1}\eta_{2}}\Tr{\Im G_{1}F\Im G_{2}\Im G_{1}F^{*}}\\
    &\prec\frac{N\phi^{av}_{2}(w_{1},w_{2})}{\eta_{1}^{3}\eta_{2}}.
\end{align*}
The last two estimates follow by \eqref{eq:avLL2}.
\end{proof}

\section{Comparison Lemmas and Proofs of Theorems \ref{thm1} and \ref{thm2}}\label{sec:comparison}
The main inputs for the comparison lemmas are the rigidity in \eqref{eq:rigidity}, eigenvector delocalisation and a bound on the inner product between left and right singular vectors. Let $X$ be a non-Hermitian Wigner matrix and $\mbf{w}_{n}(z),\,|n|\in[N]$ be the eigenvectors of the Hermitisation $W_{z}$. We will often suppress the $z$-dependence for notational simplicity. The delocalisation bounds follow directly from the isotropic local law in Proposition \ref{prop:singleLL} and \cite[Theorem 2.6]{cipolloni_optimal_2024}:
\begin{align}
    \left|\mbf{x}^{*}\mbf{w}_{n}(z)\right|&\prec N^{-1/2},\qquad |n|<cN,\label{eq:deloc}
\end{align}
for any $\mbf{x}\in S^{2N-1}$ and $|z|<1+\epsilon$. For bulk eigenvalues, the singular vector bound required is contained in \cite[Theorem 2.7]{cipolloni_optimal_2024}:
\begin{align}
    \left|\mbf{w}_{n}^{*}(z)F\mbf{w}_{m}(z)\right|&\prec N^{-1/2},\qquad |n|<cN\text{ and }|z|<1-\epsilon.\label{eq:svOverlapBulk}
\end{align}
For edge eigenvalues, we use the bound in \eqref{eq:svOverlapEdge}, i.e.
\begin{align*}
    |\mbf{w}_{n}^{*}(z)F\mbf{w}_{m}(z)|&\prec\begin{cases}
    \left(\frac{|n|\wedge |m|}{N}\right)^{1/4}& \quad |n|\wedge|m|\geq c_{1}|n|\vee|m|\\
    \frac{1}{N^{1/4}(|n|\vee |m|)^{1/4}}&\quad |n|\wedge|m|\leq c_{1}|n|\vee|m|
    \end{cases},
\end{align*}
for all $|n|,|m|<c_{2}N$ and $\big||z|-1\big|\lesssim N^{-1/2}$. We also observe that the asymptotics for the density $\rho_{z}(E)$ in \eqref{eq:rhoAsymp1} and \eqref{eq:rhoAsymp2} imply that, when $\big||z|-1\big|\lesssim N^{-1/2}$, we have
\begin{align*}
    |\gamma_{n}|&\simeq\left(\frac{|n|}{N}\right)^{3/4},
\end{align*}
and hence by rigidity
\begin{align}
    |\lambda_{n}|&\simeq\left(\frac{|n|}{N}\right)^{3/4},
\end{align}
for $N^{\xi}<|n|<cN$ with very high probability.

\subsection{Least Non-Zero Singular Value}
To prove a bound for the least non-zero singular value at the edge, we need a comparison lemma for the statistic
\begin{align*}
    \sum_{n=1}^{N}f_{z_{0}}(z_{n})g_{\eta}(z_{n}),
\end{align*}
where $f_{z_{0}}(z)=f(N^{1/2}(z-z_{0}))$ for a smooth function $f$ supported in $B_{r}(0)$ and $g_{\eta}(z)=\eta\tr G_{z}(i\eta)-1$. For this we need the following a priori estimates which follow from the delocalisation and singular vector overlap bounds listed at the beginning of this section.
\begin{lemma}\label{lem:leastSVaPriori}
Let $w=E+i\eta$ with $\eta\in(0,N^{-3/4}]$ and $|E|\in[0,N^{-3/4}]$. Then we have
\begin{align}
    \Tr{G_{z}(i\eta)}&\prec\frac{1}{N\eta},\\
    \partial_{z}\Tr{G_{z}(i\eta)}&\prec\frac{1}{N^{5/4}\eta^{2}},\\
    \nabla^{2}_{z}\Tr{G_{z}(i\eta)}&\prec\frac{1}{N^{3/2}\eta^{3}}.
\end{align}
Let $\epsilon>0$ and $\eta_{0}=N^{-3/4+\epsilon}$. Then for any $\mu,\nu\in[2N]$ we have
\begin{align}
    \left|(G_{z}(w))_{\mu,\nu}-(G_{z}(i\eta_{0}))_{\mu,\nu}\right|&\prec\frac{1}{N\eta},\label{eq:svaPriori4}\\
    (G_{z}(w)FG_{z}(w))_{\mu,\nu}&\prec \frac{1}{N^{5/4}\eta^{2}},\label{eq:svaPriori5}\\
    (G_{z}(w)FG_{z}(w)F^{*}G_{z}(w))_{\mu,\nu}&\prec\frac{|w|}{N^{3/2}\eta^{4}}.\label{eq:svaPriori6}
\end{align}
\end{lemma}
\begin{proof}
The first bound follows from the monotonicity of $\eta\mapsto\eta\Im G_{z}(i\eta)$, while the second and third bounds follow by \cite[Lemma 7.4]{cipolloni_universality_2024} and contour integration. Thus we need only prove the entrywise bounds, for which we use the delocalisation and overlap bounds on the singular vectors (we could also prove the second and third bounds in this way). Since the proofs of \eqref{eq:svaPriori4}, \eqref{eq:svaPriori5} and \eqref{eq:svaPriori6} follow the same pattern we write the details for the latter.

Let $\mbf{x}=\mbf{e}_{\mu}$ and $\mbf{y}=\mbf{e}_{\nu}$. From the spectral decomposition of $G_{z}$ we have
\begin{align*}
    x&:=\mbf{x}^{*}G_{z}(w)FG_{z}(w)F^{*}G_{z}(w)\mbf{y}\\
    &=\sum_{n,m,l}\frac{w(\mbf{x}^{*}\mbf{w}_{n})(\mbf{w}_{n}^{*}F\mbf{w}_{m})(\mbf{w}_{m}^{*}F^{*}\mbf{w}_{l})(\mbf{w}_{l}^{*}\mbf{y})}{(\lambda_{n}-w)(\lambda^{2}_{m}-w^{2})(\lambda_{l}-w)}.
\end{align*}
Let $\xi\in(0,\epsilon)$, $\eta_{0}=N^{-3/4+\xi}$ and
\begin{align*}
    x_{0}&:=\mbf{x}^{*}G_{z}(i\eta_{0})FG_{z}(i\eta_{0})F^{*}G_{z}(i\eta_{0})\mbf{y}.
\end{align*}
Since $\xi>0$ is arbitrarily small, we will routinely absorb factors of $N^{\xi}$ in the symbol $\prec$. By Proposition \ref{prop:ll}, we have
\begin{align*}
    |x_{0}|&=|\mbf{x}^{*}G_{z}(i\eta_{0})F\Im G_{z}(i\eta_{0})F^{*}G_{z}(i\eta_{0})\mbf{y}|\\
    &\leq\sqrt{\mbf{x}^{*}G_{z}(i\eta_{0})F\Im G_{z}(i\eta_{0})F^{*}G^{*}_{z}(i\eta_{0})\mbf{x}}\cdot\sqrt{\mbf{y}^{*}G^{*}_{z}(i\eta_{0})F\Im G_{z}(i\eta_{0})F^{*}G_{z}(i\eta_{0})\mbf{y}}\\
    &\lesssim\frac{\rho_{0}^{3}}{\eta_{0}^{2}}\\
    &\lesssim\frac{1}{\eta_{0}}\\
    &\lesssim\frac{|w|}{N^{3/2}\eta^{4}}.
\end{align*}
We want to compare $x$ and $x_{0}$, so we write
\begin{align*}
    x&=\frac{w}{\eta_{0}}x_{0}+\sum_{i=1}^{7}A_{i},
\end{align*}
where
\begin{align*}
    A_{1}&=w(i\eta_{0}-w)\sum_{n,m,l}\frac{(\mbf{x}^{*}\mbf{w}_{n})(\mbf{w}_{n}^{*}F\mbf{w}_{m})(\mbf{w}_{m}^{*}F^{*}\mbf{w}_{l})(\mbf{w}_{l}^{*}\mbf{y})}{(\lambda_{n}-w)(\lambda_{n}-i\eta_{0})(\lambda_{m}^{2}+\eta_{0}^{2})(\lambda_{l}-i\eta_{0})},\\
    A_{2}&=w(\eta^{2}_{0}-w^{2})\sum_{n,m,l}\frac{(\mbf{x}^{*}\mbf{w}_{n})(\mbf{w}_{n}^{*}F\mbf{w}_{m})(\mbf{w}_{m}^{*}F^{*}\mbf{w}_{l})(\mbf{w}_{l}^{*}\mbf{y})}{(\lambda_{n}-i\eta_{0})(\lambda_{m}^{2}-w^{2})(\lambda_{m}^{2}+\eta_{0}^{2})(\lambda_{l}-i\eta_{0})},\\
    A_{3}&=w(i\eta_{0}-w)\sum_{n,m,l}\frac{(\mbf{x}^{*}\mbf{w}_{n})(\mbf{w}_{n}^{*}F\mbf{w}_{m})(\mbf{w}_{m}^{*}F^{*}\mbf{w}_{l})(\mbf{w}_{l}^{*}\mbf{y})}{(\lambda_{n}-i\eta_{0})(\lambda_{m}^{2}+\eta_{0}^{2})(\lambda_{l}-w)(\lambda_{l}-i\eta_{0})},\\
    A_{4}&=w(i\eta_{0}-w)(\eta_{0}^{2}-w^{2})\sum_{n,m,l}\frac{(\mbf{x}^{*}\mbf{w}_{n})(\mbf{w}_{n}^{*}F\mbf{w}_{m})(\mbf{w}_{m}^{*}F^{*}\mbf{w}_{l})(\mbf{w}_{l}^{*}\mbf{y})}{(\lambda_{n}-w)(\lambda_{n}-i\eta_{0})(\lambda_{m}^{2}-w^{2})(\lambda_{m}^{2}+\eta_{0}^{2})(\lambda_{l}-i\eta_{0})},\\
    A_{5}&=w(i\eta_{0}-w)^{2}\sum_{n,m,l}\frac{(\mbf{x}^{*}\mbf{w}_{n})(\mbf{w}_{n}^{*}F\mbf{w}_{m})(\mbf{w}_{m}^{*}F^{*}\mbf{w}_{l})(\mbf{w}_{l}^{*}\mbf{y})}{(\lambda_{n}-w)(\lambda_{n}-i\eta_{0})(\lambda_{m}^{2}+\eta_{0}^{2})(\lambda_{l}-w)(\lambda_{l}-i\eta_{0})},\\
    A_{6}&=w(i\eta_{0}-w)(\eta_{0}^{2}-w^{2})\sum_{n,m,l}\frac{(\mbf{x}^{*}\mbf{w}_{n})(\mbf{w}_{n}^{*}F\mbf{w}_{m})(\mbf{w}_{m}^{*}F^{*}\mbf{w}_{l})(\mbf{w}_{l}^{*}\mbf{y})}{(\lambda_{n}-i\eta_{0})(\lambda_{m}^{2}-w^{2})(\lambda_{m}^{2}+\eta_{0}^{2})(\lambda_{l}-w)(\lambda_{l}-i\eta_{0})},\\
    A_{7}&=w(i\eta_{0}-w)^{2}(\eta_{0}^{2}-w^{2})\sum_{n,m,l}\frac{(\mbf{x}^{*}\mbf{w}_{n})(\mbf{w}_{n}^{*}F\mbf{w}_{m})(\mbf{w}_{m}^{*}F^{*}\mbf{w}_{l})(\mbf{w}_{l}^{*}\mbf{y})}{(\lambda_{n}-w)(\lambda_{n}-i\eta_{0})(\lambda_{m}^{2}-w^{2})(\lambda_{m}^{2}+\eta_{0}^{2})(\lambda_{l}-w)(\lambda_{l}-i\eta_{0})}.
\end{align*}

Let $c_{2}>0$ be sufficiently small such that the overlap bound in \eqref{eq:svOverlapEdge} holds for indices $|n|,|m|<c_{2}N$. We split the sum over each index into three:
\begin{align*}
    I_{1}&:=\{|n|<N^{2\xi}\},\\
    I_{2}&:=\{N^{2\xi}<|n|<c_{2}N\},\\
    I_{3}&:=\{|n|>c_{2}N\},\\
\end{align*}
In accordance with this decomposition we denote by $A_{i}(I,J,K)$ the absolute value of the term in $A_{i}$ for which $n\in I,\,m\in J,\,k\in K$.

Let us first make some basic observations. In $I_{2}$ we have $|\lambda_{n}|\simeq (|n|/N)^{3/4}\gg\eta_{0}\gg|w|$ and so
\begin{align}
    \frac{1}{|\lambda_{n}-w|}\simeq\frac{1}{|\lambda_{n}-i\eta_{0}|}\simeq\left(\frac{N}{|n|}\right)^{3/4}.\label{eq:p1}
\end{align}
In $I_{j}\times I_{k}$ for $j,k\in\{1,2\}$ we can use the delocalisation and overlap bounds for $\mbf{x}^{*}\mbf{w}_{n}$ and $\mbf{w}_{n}^{*}F\mbf{w}_{m}$ respectively. In $I_{1}\times I_{1}$ we have
\begin{align}
    |\mbf{w}_{n}^{*}F\mbf{w}_{m}|&\prec \frac{1}{N^{1/4}}.\label{eq:p2}
\end{align}
For $(n,m)\in I_{2}\times I_{1}$ we have
\begin{align}
    |\mbf{w}_{n}^{*}F\mbf{w}_{m}|&\prec \frac{1}{(N|n|)^{1/4}},\label{eq:p3}
\end{align}
where we have absorbed factors of $N^{\xi}$ in the symbol $\prec$. For $(n,m)\in I_{2}\times I_{2}$ we have
\begin{align}
    |\mbf{w}_{n}^{*}F\mbf{w}_{m}|&\prec\left(\frac{|m|}{N}\right)^{1/4},\quad |m|\leq|n|.\label{eq:p4}
\end{align}

Consider a general sum of the kind in $A_{i},\,i=1,...,7$ with indices restricted to $I_{1}\cup I_{2}$. Using \eqref{eq:p1} to \eqref{eq:p4} we conclude the following:
\begin{itemize}
\item for each factor of $\mbf{x}^{*}\mbf{w}_{n}$ or $\mbf{w}_{n}^{*}\mbf{y}$ we gain a factor $N^{-1/2}$;
\item for each index $n\in I_{1}$ we gain a factor $1/|\Im x|$ from $1/|\lambda_{n}-x|$;
\item for each index $n\in I_{2}$ we gain a factor $(N/|n|)^{3/4}$ from $1/|\lambda_{n}-x|$;
\item for each pair $(n,m)\in I_{1}\times I_{1}$ we gain a factor $N^{-1/4}$ from $\mbf{w}_{n}^{*}F\mbf{w}_{m}$;
\item for each pair $(n,m)\in I_{2}\times I_{1}$ we gain a factor $(N|n|)^{-1/4}$ from $\mbf{w}_{n}^{*}F\mbf{w}_{m}$;
\item for each pair $(n,m)\in I_{2}\times I_{2}$ we gain a factor $N^{-1/4}(|n|\wedge|m|)^{1/4}$.
\end{itemize}
We claim that a general sum is bounded by the product of each of these factors, i.e. we do not gain any additional factors from the summation. The summation over $I_{1}$ gives a factor $|I_{1}|\prec N^{2\xi}$, which we can absorb in $\prec$ since $\xi>0$ is arbitrary. We thus need to check that the summation over $I_{2}$ does not contribute extra powers of $N$.

Let us begin with an example:
\begin{align*}
    \sum_{n,m\in I_{2}}\frac{|\mbf{w}_{n}^{*}F\mbf{w}_{m}|}{|\lambda_{n}|^{p}||\lambda_{m}|^{q}}&\prec\sum_{n,m\in I_{2}}\left(\frac{N}{|n|}\right)^{3p/4}\left(\frac{N}{|m|}\right)^{3q/4}\left(\frac{|n|\wedge|m|}{N}\right)^{1/4}\\
    &\prec\sum_{n\in I_{2}}\frac{N^{(3(p+q)-1)/4}}{|n|^{(3(p+q)-5)/4}}.
\end{align*}
If we now sum over $n$ then we will not gain any extra powers of $N$ if $p+q\geq3$. If $p+q=3$ then we sum over $|n|^{-1}$ leading to a factor $\log N$ which can be absorbed in the symbol $\prec$.

Consider a second example:
\begin{align*}
    \sum_{n,l\in I_{2}}\sum_{m\in I_{1}}\frac{|\mbf{w}_{n}^{*}F\mbf{w}_{m}||\mbf{w}_{m}^{*}F^{*}\mbf{w}_{l}|}{|\lambda_{n}||\lambda_{m}-w|^{q}|\lambda_{l}|}&\prec\frac{1}{\eta^{q}}\sum_{n\in I_{2}}\frac{N^{1/2}}{|n|}\sum_{l\in I_{2}}\frac{N^{1/2}}{|l|}
\end{align*}
In this case the extra factors of $|n|^{-1/4}$ and $|l|^{-1/4}$ from $\mbf{w}_{n}^{*}F\mbf{w}_{m}$ and $\mbf{w}_{m}^{*}F^{*}\mbf{w}_{l}$ ensure that we do not gain any factors (except $\log N$, which we can ignore) from the summations over $I_{2}$. Note that in the first example this extra gain is not needed if $p+q$ is large enough.

Since we have at least five factors of singular values in the denomniator of any $A_{i}$, we always have sufficiently many inverse powers of $|n|$ to bound the summation over $I_{2}$. This is the reason for comparing $x$ and $x_{0}$ rather than directly estimating $x$, since it introduces additional factors of singular values in the denominator. We can now bound any $A_{i}(I,J,K)$ with $I,J,K\in\{I_{1},I_{2}\}$ by applying the rules listed above; we obtain
\begin{align*}
    A_{i}(I,J,K)&\prec\frac{|w|}{N^{3/2}\eta^{4}},\quad I,J,K\in \{I_{1},I_{2}\}.
\end{align*}

When an index is in $I_{3}$, we use Cauchy-Schwarz: if $m\in I_{3}$ then by Cauchy-Schwarz we have
\begin{align*}
    A_{1}(I_{j},I_{3},I_{k})&\prec|w|\eta_{0}\sum_{n}\frac{|\mbf{x}^{*}\mbf{w}_{n}|}{|\lambda_{n}-w||\lambda_{n}-i\eta_{0}|}\sum_{l}\frac{|\mbf{w}_{l}^{*}\mbf{y}|}{|\lambda_{l}-i\eta_{0}|};
\end{align*}
if $n\in I_{3}$ then
\begin{align*}
    A_{1}(I_{3},I_{j},I_{k})&\prec|w|\eta_{0}\sum_{m,l}\frac{|\mbf{w}_{m}^{*}F^{*}\mbf{w}_{l}||\mbf{w}_{l}^{*}\mbf{y}|}{(\lambda_{m}^{2}+\eta_{0}^{2})|\lambda_{l}-i\eta_{0}|};
\end{align*}
if $l\in I_{3}$ then
\begin{align*}
    A_{1}(I_{j},I_{k},I_{3})&\prec|w|\eta_{0}\sum_{n,m}\frac{|\mbf{x}^{*}\mbf{w}_{n}||\mbf{w}_{n}^{*}F\mbf{w}_{m}|}{|\lambda_{n}-w||\lambda_{n}-i\eta_{0}|(\lambda_{m}^{2}+\eta_{0}^{2})}.
\end{align*}
We can bound each of these resulting expressions using the above rules.
\end{proof}
We remark that the above proof would be simpler if we were able to capture the sharp decay in $||n|-|m||$ of $|\mbf{w}_{n}^{*}F\mbf{w}_{m}|$. We would then be able to directly estimate each expression at $w$, without the need to compare with the corresponding expression at $\eta_{0}$.

For $\epsilon>0$ and $f:\mbb{F}_{\beta}\to\mbb{C}$ define
\begin{align}
    f_{z_{0}}(z)&:=f(N^{1/2}(z-z_{0})),\\
    g_{\eta}(z)&:=\eta\Im\tr G_{z}(i\eta)-1.
\end{align}
Theorem \ref{thm2} will follow directly from the result for Gauss-divisible matrices and the comparison lemma below. For this we introduce the notion of $t$-matching matrices.

\begin{definition}\label{def2}
We say that $A$ and $B$ are $t$-matching for some $t\geq0$ if 
\begin{align}
    \mbb{E}\left[(\Re a_{jk})^{p}(\Im a_{jk})^{3-p}\right]&=\mbb{E}\left[(\Re b_{jk})^{p}(\Im a_{jk})^{3-p}\right],\qquad p=0,1,2,3,
\end{align}
and
\begin{align}
    \left|\mbb{E}\left[(\Re a_{jk})^{p}(\Im a_{jk})^{4-p}\right]-\mbb{E}\left[(\Re b_{jk})^{p}(\Im b_{jk})^{4-p}\right]\right|&\lesssim \frac{t}{N^{2}},\qquad p=0,...,4.
\end{align}
\end{definition}

\begin{lemma}\label{lem:svComparison}
Let $\tau>0$ and $A$ and $B$ be $t$-matching non-Hermitian Wigner matrices for $t\leq N^{-\tau}$. Let $\epsilon\in(0,1/24)$ and $\eta=N^{-3/4-\epsilon}$. Then for any $z_{0}\in\mbb{T}$ and $f\in C^{2}(\mbb{F}_{\beta})$ with compact support we have
\begin{align}
    \left|\left(\mbb{E}_{A}-\mbb{E}_{B}\right)\left[\sum_{n=1}^{N_{\beta}}f_{z_{0}}(z_{n})g_{\eta}(z_{n})\right]\right|&\prec N^{3/4}\eta+\frac{t}{N(N^{3/4}\eta)^{36}}+\frac{1}{N^{7/4}(N^{3/4}\eta)^{42}}.\label{eq:svMatching}
\end{align}
\end{lemma}
\begin{proof}
We only give a sketch since the argument is standard and the details in this particular case are essentially the same as in \cite[Proof of Theorem 2.1]{osman_least_2024}. By a union bound we have
\begin{align*}
    P\left(\min_{|z_{n}-z_{0}|<rN^{-1/2}}s_{2}(z_{n})<\eta\right)&\leq\mbb{E}\left[\sum_{n=1}^{N}f_{\eta}(z_{n})g_{z_{0}}(z_{n})\right].
\end{align*}
By Girko's formula we have for large $D>0$
\begin{align*}
    \sum_{n=1}^{N}f_{z_{0}}(z_{n})g_{\eta}(z_{n})&=-\frac{1}{4\pi}\int_{\mbb{C}}\nabla^{2}_{z}(f_{z_{0}}g_{\eta})\int_{0}^{N^{D}}\Im\tr G_{z}(i\sigma)\,\mathrm{d}\sigma \,\mathrm{d}^{2}z+O(N^{2-D}),
\end{align*}
where the error is in the sense of stochastic domination.

We define 
\begin{align*}
    I(\eta_{1},\eta_{2})&:=\int_{\mbb{C}}\nabla^{2}_{z}(f_{z_{0}}g_{\eta})\int_{\eta_{1}}^{\eta_{2}}\Im\tr G_{z}(i\sigma)\,\mathrm{d}\sigma,
\end{align*}
and split the integral over $\sigma$ into $I(0,\eta_{0})+I(\eta_{0},N^{D})$ for $\eta_{0}=N^{-3/4-\nu}$. By \cite[Theorem 3.2]{tao_smooth_2010} we have, for any fixed $A>0$,
\begin{align*}
    P\left(s_{N}(z)<N^{-A-1}\right)&\prec N^{-A}.
\end{align*}
By \cite[Proposition 2.1]{cipolloni_edge_2021} we have
\begin{align*}
    \mbb{E}\left[\left|\left\{n:s_{n}(z)\leq N^{\nu/2}\eta_{0}\right\}\right|\right]&\prec N^{3/4+\nu/2}\eta_{0},
\end{align*}
for $\big||z|-1\big|<CN^{-1/2}$. Using these bounds and a Riemann sum approximation (see \cite[Section 5]{osman_least_2024}, which is based on \cite[Proof of Lemma 4]{cipolloni_edge_2021} and \cite[Proof of Theorem 1.1]{he_edge_2023}) we can obtain
\begin{align}
    \mbb{E}\left[I(0,\eta_{0})\right]&\prec \frac{N^{-\nu/2}}{N^{3/2}\eta^{2}}.
\end{align}

For the region $\sigma\geq\eta_{0}$ we integrate by parts:
\begin{align}
    I(\eta_{0},N^{D})&=\Re\int_{\mbb{C}}f_{z_{0}}(z)g_{\eta}(z)h_{\eta_{0}}(z)\,\mathrm{d}^{2}z+O(N^{-D}),
\end{align}
where
\begin{align}
    h_{\eta_{0}}(z)&:=\tr G_{z}(i\eta_{0})FG_{z}(i\eta_{0})F^{*}.
\end{align}
The rest of the proof is a straightforward application of the Lindeberg method, using the a priori bounds in Lemma \ref{lem:leastSVaPriori} to estimate the matrix elements that arise from the resolvent expansion
\begin{align*}
    G^{(1)}&=\sum_{p=0}^{4}(G^{(0)}\Delta_{ij})^{p}G^{(0)}+(G^{(0)}\Delta_{ij})^{5}G^{(1)},
\end{align*}
where $G^{(0)}$ and $G^{(1)}$ are resolvents of Hermitisations of matrices that differ in one element.

The contribution from the first three moments cancels after taking expectations due to the moment matching assumption. For $p\geq4$, we use
\begin{align*}
    |G_{i,i}|,|G_{\hat{\jmath},\hat{\jmath}}|&\prec\frac{1}{N\eta_{0}}\\
    |G_{i,\hat{\jmath}}|,|G_{\hat{\jmath},i}|&\prec\delta_{i,j}+\frac{1}{N\eta_{0}},
\end{align*}
and Lemma \ref{lem:leastSVaPriori} to bound the contribution from the $p$-th moment by
\begin{align*}
    &N^{2}\cdot\frac{1}{N^{p/2}}\cdot\frac{1}{(N\eta_{0})^{p-1}}\cdot\frac{1}{N^{3/2}\eta_{0}^{3}}\cdot\frac{1}{N}\\
    &+N\cdot\frac{1}{N^{p/2}}\cdot\frac{1}{N^{3/2}\eta_{0}^{3}}\cdot\frac{1}{N}.
\end{align*}
In the first line, the first factor comes from summing over $O(N^{2})$ off-diagonal elements, the second from the bound $|X_{i,j}|\prec N^{-1/2}$, the third from the $p-1$ factors of off-diagonal resolvent entries, the fourth from an entry of $GFGF^{*}G$ and the fifth from $\|f_{z_{0}}\|_{1}\lesssim N^{-1}$. The second line comes from the sum over diagonal elements, where we have $|G_{i,\hat{\imath}}|\prec1$. When $p=4$ we also have a factor $t$ from the moment matching assumption. We can rewrite this as
\begin{align}
    \frac{1+t\delta_{p,4}}{N^{\frac{3p-8}{4}}}\cdot\frac{1}{(N^{3/4}\eta)^{6(p+2)}}+\frac{1+t\delta_{p,4}}{N^{\frac{2p-3}{4}}}\cdot\frac{1}{(N^{3/4}\eta)^{18}},\label{eq:pBound}
\end{align}
where we recall that $N^{3/4}\eta_{0}=(N^{3/4}\eta)^{6}$. If $\epsilon<1/24$ then this decreases with $p$ and so we obtain \eqref{eq:svMatching} from the $p=4$ and $p=5$ terms.
\end{proof}

Note that \eqref{eq:pBound} we observe that only two moments need to match exactly. We conclude with the proof of Theorem \ref{thm2}.
\begin{proof}[Proof of Theorem \ref{thm2}]
By standard arguments, we find a non-Hermitian Wigner matrix $\wt{X}$ such that $X$ and $M:=\frac{1}{\sqrt{1+t}}(\wt{X}+\sqrt{t}Y)$ are $t$-matching for $Y\sim Gin_{\beta}(N)$ and $t=N^{-1/3+\xi}$. By Lemma \ref{lem:svComparison} and Proposition \ref{prop:leastSVGauss}, for $\eta=N^{-3/4-\epsilon}$ and $\epsilon\in(0,1/24)$ we have
\begin{align*}
    P_{X}\left(\min_{N^{1/2}|z_{n}-z_{0}|<r}s_{2}(z_{n})<\eta\right)&\lesssim\mbb{E}_{M}\left[\sum_{n=1}^{N_{\beta}}f_{z_{0}}(z_{n})g_{\eta}(z_{n})\right]+N^{3/4}\eta\\
    &\lesssim N^{3/2}\eta^{2}|\log N^{3/4}\eta|^{2-\beta}+N^{3/4}\eta+N^{-D}\\
    &\lesssim N^{3/4}\eta,
\end{align*}
where $f_{\eta}$ and $g_{\eta}$ are as in \eqref{eq:f} and \eqref{eq:g}.
\end{proof}

\subsection{Overlap}
We begin with the construction of an approximate overlap in terms of the resolvent $G_{z}$. For $\eta,\zeta>0$ and an eigenvalue $z_{n}$, define $O_{\eta,\zeta}(z_{n})$ by
\begin{align}
    \frac{1}{O_{\eta,\zeta}(z_{n})}&=\frac{2}{\pi}\int_{0}^{N^{\zeta}\eta}\eta\tr\Im G_{z_{n}}(E+i\eta)F\Im G_{z_{n}}(E+i\eta)F^{*} \,\mathrm{d}E.\label{eq:Ohat}
\end{align}
\begin{lemma}\label{lem:approximateOverlap}
Let $z_{n}$ be an eigenvalue of $X$ such that $|z_{n}-z_{0}|\lesssim N^{-1/2}$. For sufficiently small $\epsilon>3\zeta>0$ and any $\xi>0$ the following holds with probability at least $1-N^{-\epsilon}$:
\begin{enumerate}[i)]
\item if $z_{0}\in\mbb{D}_{\beta}$ and $\eta=N^{-1-\epsilon}$ then
\begin{align}
    \frac{N}{O_{\eta,\zeta}(z_{n})}&=\left[1+O(N^{-3\zeta})\right]\frac{N}{O_{nn}}+O(N^{\xi-\epsilon});
\end{align}
\item if $z_{0}\in\mbb{T}_{\beta}$ and $\eta=N^{-3/4-\epsilon}$ then
\begin{align}
    \frac{N^{1/2}}{O_{\eta,\zeta}(z_{n})}&=\left[1+O(N^{-3\zeta})\right]\frac{N^{1/2}}{O_{nn}}+O(N^{\xi-\epsilon}).
\end{align}
\end{enumerate}
\end{lemma}
\begin{proof}
Since $z_{n}$ is an eigenvalue of $X$ we have $s_{1}(z_{n})=0$. By the spectral decomposition of $G_{z}$ we have
\begin{align*}
    \eta\tr\Im G_{z_{n}}(w)F\Im G_{z_{n}}(w)F^{*}&=\frac{2\eta^{3}}{(E^{2}+\eta^{2})^{2}}\cdot\frac{1}{O_{nn}}\\
    &+\sum_{m,l\neq \pm(1,1)}\frac{\eta^{3}|\mbf{w}_{m}^{*}F\mbf{w}_{l}|^{2}}{|\lambda_{m}(z_{n})-w|^{2}|\lambda_{l}(z_{n})-w|^{2}}.
\end{align*}
Observe that
\begin{align*}
    \frac{1}{\pi}\int_{0}^{N^{\zeta}\eta}\frac{4\eta^{3}}{(E^{2}+\eta^{2})^{2}}\,\mathrm{d}E&=1+O(N^{-3\zeta}).
\end{align*}

Consider $z_{0}\in\mbb{D}_{\beta}$ and $\eta=N^{-1-\epsilon}$. By \cite[Theorem 2.1]{osman_least_2024} (see also \cite[Lemma 2.4]{dubova_gaussian_2024}), for sufficiently small $\epsilon$ we have $|\lambda_{m}(z_{n})|>N^{\epsilon/2}\eta\gg N^{\zeta}\eta$ for $|m|>1$ with probability at least $1-N^{-\epsilon}$. On this event we have $|\lambda_{m}-w|^{2}\simeq\lambda_{m}^{2}+N^{\epsilon}\eta^{2}$ for $|E|\lesssim N^{\zeta}\eta$ and so
\begin{align*}
    \int_{0}^{N^{\zeta}\eta}\sum_{|m|>1}\frac{\eta^{3}|\mbf{w}_{1}^{*}F\mbf{w}_{m}|^{2}}{(E^{2}+\eta^{2})|\lambda_{m}-w|^{2}}\,\mathrm{d}E&\lesssim\sum_{m=2}^{N}\frac{\eta^{2}|\mbf{w}_{1}^{*}F\mbf{w}_{m}|^{2}}{\lambda_{m}^{2}+N^{\epsilon}\eta^{2}}\\
    &\prec\frac{1}{N}\sum_{m=2}^{N}\frac{\eta^{2}}{\lambda_{m}^{2}+N^{\epsilon}\eta^{2}}\\
    &\lesssim N^{-1-\epsilon},
\end{align*}
where in the second line we used the overlap bound $|\mbf{w}_{1}^{*}F\mbf{w}_{m}|^{2}\prec N^{-1}$ for $|m|<cN$ and the trivial bound $|\lambda_{m}|>c$ for $|m|>cN$. Similarly, we have
\begin{align*}
    \int_{0}^{N^{\zeta}\eta}\sum_{|m|,|l|>1}\frac{\eta^{3}|\mbf{w}_{m}^{*}F\mbf{w}_{l}|^{2}}{|\lambda_{m}-w|^{2}|\lambda_{l}-w|^{2}}\,\mathrm{d}E&\prec N^{\zeta}\eta^{4}\left(\frac{1}{N}\sum_{m=2}^{N}\frac{1}{\lambda_{m}^{2}+N^{\epsilon}\eta^{2}}\right)^{2}\\
    &\lesssim N^{-1-2\epsilon+\zeta}.
\end{align*}
for some $\delta>0$ depending on $\epsilon$ and $\zeta$.

Now consider $z_{0}\in\mbb{T}_{\beta}$ and $\eta=N^{-3/4-\epsilon}$. By Theorem \ref{thm2}, for sufficiently small $\epsilon$ we have $|\lambda_{m}(z_{n})|>N^{\epsilon/2}\eta\gg N^{\zeta}\eta$ for $|m|>1$ with probability at least $1-N^{-\epsilon}$. Thus
\begin{align*}
    \int_{0}^{N^{\zeta}\eta}\sum_{|m|>1}\frac{\eta^{3}|\mbf{w}_{1}^{*}F\mbf{w}_{m}|^{2}}{(E^{2}+\eta^{2})|\lambda_{m}-w|^{2}}\,\mathrm{d}E&\lesssim\sum_{m=2}^{N}\frac{\eta^{2}|\mbf{w}_{1}^{*}F\mbf{w}_{m}|^{2}}{\lambda_{m}^{2}+N^{\epsilon}\eta^{2}}.
\end{align*}
Splitting the sum into $m<N^{\xi}$ and $m>N^{\xi}$ and using \eqref{eq:svOverlapEdge} and rigidity we find
\begin{align*}
    \sum_{m<N^{\xi}}\frac{\eta^{2}|\mbf{w}_{1}^{*}F\mbf{w}_{m}|^{2}}{\lambda_{m}^{2}+N^{\epsilon}\eta^{2}}&\prec N^{-1/2-\epsilon},
\end{align*}
and
\begin{align*}
    \sum_{m>N^{\xi}}\frac{\eta^{2}|\mbf{w}_{1}^{*}F\mbf{w}_{m}|^{2}}{\lambda_{m}^{2}+N^{\epsilon}\eta^{2}}&\prec\sum_{N^{\xi}<m<cN}\frac{\eta^{2}}{(Nm)^{1/2}}\left(\frac{N}{m}\right)^{3/2}+\sum_{m>cN}C\eta^{2}\\
    &\prec N\eta^{2}\\
    &\prec N^{-1/2-2\epsilon}.
\end{align*}

For the double sum over $|m|,|l|>1$, we note that by \eqref{eq:aPriori5} below we have
 \begin{align*}
    \int_{0}^{N^{\zeta}\eta}\sum_{|m|,|l|>1}\frac{\eta^{3}|\mbf{w}_{m}^{*}F\mbf{w}_{l}|^{2}}{|\lambda_{m}-w|^{2}|\lambda_{l}-w|^{2}}\,\mathrm{d}E&\lesssim N^{\zeta}\sum_{m,l=2}^{N}\frac{\eta^{4}|\mbf{w}_{m}^{*}F\mbf{w}_{l}|^{2}}{(\lambda_{m}^{2}+N^{\epsilon}\eta^{2})(\lambda_{l}^{2}+N^{\epsilon}\eta^{2})}\\
    &\prec N^{-1/2-2\epsilon+\zeta}.
\end{align*}
\end{proof}

We can prove a comparison lemma for $O_{\eta,\zeta}(z_{n})$ by the Lindeberg method. The a priori estimates that we need are contained in the following.
\begin{lemma}\label{lem:aPrioriOverlap}
Let $w=E+i\eta$. Then we have the following bounds:
\begin{enumerate}[i)]
\item if $|z|<1-\epsilon$ for some $\epsilon>0$ and $|w|\in(0,N^{-1}]$,
\begin{align}
    \Tr{G_{z}(w)FG_{z}(w)F^{*}}&\prec\frac{1}{N^{2}\eta^{2}},\label{eq:aPriori1}\\
    \partial_{z}\Tr{G_{z}(w)FG_{z}(w)F^{*}}&\prec\frac{1}{N^{5/2}\eta^{3}},\label{eq:aPriori2}\\
    \nabla^{2}_{z}\Tr{G_{z}(w)FG_{z}(w)F^{*}}&\prec\frac{1}{N^{3}\eta^{4}},\label{eq:aPriori3}
\end{align}
and, for any $\mu,\nu\in[2N]$,
\begin{align}
    |(G_{z}(w)FG_{z}(w)F^{*}G_{z}(w))_{\mu,\nu}|&\prec\frac{1}{N^{2}\eta^{3}};\label{eq:aPriori4}
\end{align}
\item if $\big||z|-1\big|\lesssim N^{-1/2}$ and $|w|\in(0,N^{-3/4}]$,
\begin{align}
    \Tr{\Im G_{z}(w)F\Im G_{z}(w)F^{*}}&\prec\frac{1}{N^{3/2}\eta^{2}},\label{eq:aPriori5}\\
    \partial_{z}\Tr{\Im G_{z}(w)F\Im G_{z}(w)F^{*}}&\prec\frac{1}{N^{7/4}\eta^{3}},\label{eq:aPriori6}\\
    \nabla^{2}_{z}\Tr{\Im G_{z}(w)F\Im G_{z}(w)F^{*}}&\prec\frac{1}{N^{2}\eta^{4}}.\label{eq:aPriori7}
\end{align}
\end{enumerate}
\end{lemma}
\begin{proof}
We prove the bounds in the edge regime; the proof in the bulk regime is similar. Let $\xi>0,\,\eta_{0}=N^{-3/4+\xi}$ and $w_{0}=E+i\eta_{0}$. For \eqref{eq:aPriori5} we use the fact that $\eta\Im G_{z}(E+i\eta)\leq\eta_{0}\Im G_{z}(w_{0})$ and the local law in Proposition \ref{prop:ll}:
\begin{align*}
    \Tr{\Im G_{z}(w)F\Im G_{z}(w)F^{*}}&\leq\frac{\eta_{0}^{2}}{\eta^{2}}\Tr{\Im G_{z}(w_{0})F\Im G_{z}(w_{0})F^{*}}\\
    &\prec\frac{\eta_{0}^{2}}{\eta^{2}}\\
    &\leq\frac{N^{2\xi}}{N^{3/2}\eta^{2}}.
\end{align*}

For \eqref{eq:aPriori7}, we use the bound
\begin{align*}
    |\Tr{G_{z}(w)FG_{z}(w)FG_{z}(w)F^{*}G_{z}(w)F^{*}}|&\leq\frac{1}{2N}\sum_{n,m,l,k}\frac{|\mbf{w}_{n}^{*}F\mbf{w}_{m}||\mbf{w}_{m}^{*}F\mbf{w}_{l}||\mbf{w}_{l}^{*}F^{*}\mbf{w}_{k}||\mbf{w}_{k}^{*}F^{*}\mbf{w}_{n}|}{|\lambda_{n}-w||\lambda_{m}-w||\lambda_{l}-w||\lambda_{k}-w|}\\
    &\leq\frac{1}{2N}\sum_{n,m,l,k}\frac{|\mbf{w}_{n}^{*}F\mbf{w}_{m}|^{2}|\mbf{w}_{l}^{*}F^{*}\mbf{w}_{k}|^{2}}{|\lambda_{n}-w||\lambda_{m}-w||\lambda_{l}-w||\lambda_{k}-w|}\\
    &+\frac{1}{2N}\sum_{n,m,l,k}\frac{|\mbf{w}_{m}^{*}F\mbf{w}_{l}|^{2}|\mbf{w}_{k}^{*}F^{*}\mbf{w}_{n}|^{2}}{|\lambda_{n}-w||\lambda_{m}-w||\lambda_{l}-w||\lambda_{k}-w|}\\
    &=N\Tr{|G_{z}(w)|F|G_{z}(w)|F^{*}}^{2}.
\end{align*}
Since $\eta|G_{z}(w)|\leq\eta_{0}|G_{z}(w_{0})|$ we have
\begin{align*}
    \Tr{|G_{z}(w)|F|G_{z}(w)|F^{*}}&\leq\frac{\eta_{0}^{2}}{\eta^{2}}\Tr{|G_{z}(w_{0})|F|G_{z}(w_{0})|F^{*}}.
\end{align*}
For any $D>0$, there is an $L>0$ such that (see \cite[(A.19)]{cipolloni_universality_2024})
\begin{align*}
    |G_{z}(w_{0})|&=\frac{2}{\pi}\int_{0}^{N^{L}}\frac{\Im G_{z}(w_{x})}{\eta_{x}}\,\mathrm{d}x+O_{\prec}(N^{-D}),
\end{align*}
where $\eta_{x}=\sqrt{\eta^{2}_{0}+x^{2}}$ and $w_{x}=E+i\eta_{x}$, and so
\begin{align*}
    \Tr{|G_{z}(w_{0})|F|G_{z}(w_{0})|F^{*}}&=\frac{4}{\pi^{2}}\int_{0}^{N^{L}}\int_{0}^{N^{L}}\frac{\Tr{\Im G_{z}(w_{x})F\Im G_{z}(w_{y})F^{*}}}{\eta_{x}\eta_{y}}\,\mathrm{d}x\,\mathrm{d}y+O_{\prec}(N^{-D})
\end{align*}
Using the local law in Proposition \ref{prop:ll} we have
\begin{align*}
    \Tr{\Im G_{z}(w_{x})F\Im G_{z}(w_{y})F^{*}}&\prec\frac{\rho_{x}\rho_{y}}{\frac{\eta_{x}}{\rho_{x}}+\frac{\eta_{y}}{\rho_{y}}}\\
    &\lesssim\frac{(\eta_{x}\eta_{y})^{1/3}}{(\eta_{x}\vee\eta_{y})^{2/3}},
\end{align*}
when $\eta_{x},\eta_{y}\lesssim1$. The contribution of this region to the integral is 
\begin{align*}
    \int_{0}^{1}\int_{0}^{1}\left(\frac{1}{(\eta_{x}\vee\eta_{y})\eta_{x}\eta_{y}}\right)^{2/3}\,\mathrm{d}x\,\mathrm{d}y&\lesssim\int_{0}^{1}\frac{1}{\eta_{x}}\,\mathrm{d}x\lesssim-\log\eta_{0}\prec1.
\end{align*}
When $\eta_{y}\gtrsim1$, we use
\begin{align*}
    \Tr{\Im G_{z}(w_{x})F\Im G_{x}(w_{y})F^{*}}&\leq\frac{1}{\eta_{y}}\Im\Tr{G_{z}(w_{x})}.
\end{align*}
The contribution of this region to the integral is
\begin{align*}
    \int_{0}^{1}\frac{1}{\eta_{x}^{2/3}}\,\mathrm{d}x+\int_{1}^{N^{L}}\frac{1}{\eta_{x}^{2}}\,\mathrm{d}x&\lesssim1.
\end{align*}
Thus we find
\begin{align*}
    \Tr{|G_{z}(w_{0})|F|G_{z}(w_{0})|F^{*}}&\prec1,
\end{align*}
and hence
\begin{align*}
    \left|\Tr{G_{z}(w)FG_{z}(w)FG_{z}(w)F^{*}G_{z}(w)F^{*}}\right|&\leq N\Tr{|G_{z}(w)|F|G_{z}(w)|F^{*}}^{2}\\
    &\prec N\left(\frac{\eta_{0}^{2}}{\eta^{2}}\right)^{2}\\
    &\prec \frac{N^{4\xi}}{N^{2}\eta^{4}}.
\end{align*}
We obtain the same bound for the remaining combinations of $G_{z}$ and $F$ in $\nabla^{2}_{z}\Tr{\Im G_{z}(w)F\Im G_{z}(w)F^{*}}$. Since $\xi$ is arbitrary we obtain \eqref{eq:aPriori7}.

To prove \eqref{eq:aPriori6} we argue as in the proof of Lemma \ref{lem:leastSVaPriori} using the bound on singular vector overlaps. It is enough to prove
\begin{align*}
    \Tr{\partial_{z}(\Im G_{z}(w))F\Im G_{z}(w)F^{*}}&\prec\frac{1}{N^{7/4}\eta^{3}}.
\end{align*}
Observe that
\begin{align*}
    \partial_{z}\Im G_{z}(w)&=\frac{1}{2i}(G_{z}FG_{z}-G^{*}_{z}FG^{*}_{z})\\
    &=G_{z}F\Im G_{z}+\Im(G_{z})FG^{*}_{z},
\end{align*}
so it is enough to prove
\begin{align*}
    \Tr{G_{z}F\Im G_{z}F\Im G_{z}F^{*}}&\prec\frac{1}{N^{7/4}\eta^{3}}.
\end{align*}
By the spectral decomposition we have
\begin{align*}
    \Tr{G_{z}F\Im G_{z}F\Im G_{z}F^{*}}&=\frac{1}{N}\sum_{n,m,l}\frac{\eta^{2}\mbf{w}_{n}^{*}F\mbf{w}_{m}\mbf{w}_{m}^{*}F\mbf{w}_{l}\mbf{w}_{l}^{*}F^{*}\mbf{w}_{n}}{(\lambda_{n}-w)|\lambda_{m}-w|^{2}||\lambda_{l}-w|^{2}}.
\end{align*}
The bound for this expression follows from the argument in the proof of Lemma \ref{lem:leastSVaPriori}. It is important that we have two factors of $\Im G_{z}$ so that we have enough powers of $\lambda_{n}$ in the denominator to bound the term for which $n,m,l\in I_{2}$:
\begin{align*}
    \frac{1}{N}\sum_{n,m,l\in I_{2}}\frac{\eta^{2}|\mbf{w}_{n}^{*}F\mbf{w}_{m}||\mbf{w}_{m}^{*}F\mbf{w}_{l}||\mbf{w}_{l}^{*}F^{*}\mbf{w}_{n}|}{|\lambda_{n}-w||\lambda_{m}-w|^{2}||\lambda_{l}-w|^{2}}&\lesssim\frac{\eta^{2}}{N}\sum_{n>N^{2\xi}}n^{2}\left(\frac{N}{n}\right)^{3}\\
    &\lesssim N^{2}\eta^{2}\log N\\
    &\lesssim\frac{1}{N^{7/4}\eta^{3}}.
\end{align*}
\end{proof}

By first comparing $O_{nn}$ with $O_{\eta,\zeta}(z_{n})$ and then following the argument of Lemma \ref{lem:svComparison}, we can obtain the following comparison result for bulk overlaps.
\begin{lemma}\label{lem:overlapComparison}
Let $\epsilon>0$ and $A$ and $B$ be $t$-matching non-Hermitian Wigner matrices for $t\leq N^{-\epsilon}$. Let $r>0$ and $f\in C^{2}(\mbb{F}_{\beta})$ be supported in the ball of radius $r$ about the origin. Let $L>0$ and $g\in C^{5}([0,\infty])$ such that
\begin{align}
    \left|\frac{\mathrm{d}^{p}g}{\mathrm{d}x^{p}}\right|&\lesssim(1+x)^{L},\quad p=0,...,5.
\end{align}
Then there is a $\delta>0$ such that the following holds for sufficiently large $N$:
\begin{enumerate}
\item if $z_{0}\in\mbb{D}_{\beta}$,
\begin{align}
    \left|(\mbb{E}_{A}-\mbb{E}_{B})\left[\sum_{n=1}^{N_{\beta}}f_{z_{0}}(z_{n})g\left(\frac{N}{O_{nn}}\right)\right]\right|&\leq N^{-\delta};
\end{align}
\item if $z_{0}\in\mbb{T}_{\beta}$,
\begin{align}
    \left|(\mbb{E}_{A}-\mbb{E}_{B})\left[\sum_{n=1}^{N_{\beta}}f_{z_{0}}(z_{n})g\left(\frac{N^{1/2}}{O_{nn}}\right)\right]\right|&\leq N^{-\delta}.
\end{align}
\end{enumerate}
\end{lemma}
We remark that this extends without difficulty to finite collections of diagonal overlaps.
\begin{proof}
Assume that $z_{0}\in\mbb{D}_{\beta}$; the details in the case $z_{0}\in\mbb{T}_{\beta}$ are analogous. Let $\epsilon>3\zeta>\xi>0$ and $O_{\eta,\zeta}(z_{n})$ be defined by \eqref{eq:Ohat}. Let
\begin{align*}
    \mc{L}&=\sum_{n=1}^{N_{\beta}}f_{z_{0}}(z_{n})g\left(\frac{N}{O_{nn}}\right),\\
    \wh{\mc{L}}&=\sum_{n=1}^{N_{\beta}}f_{z_{0}}(z_{n})g\left(\frac{N}{O_{\eta,\zeta}(z_{n})}\right).
\end{align*}
Since $O_{nn}\geq1$, we have $g(N/O_{nn})\lesssim N^{L}$ and hence the deterministic bounds
\begin{align}
    |\mc{L}|&\lesssim N^{L+1},\label{eq:detL}\\
    |\wh{\mc{L}}|&\lesssim N^{(2+\zeta)L+1}.\label{eq:detLhat}
\end{align}
Define the events
\begin{align*}
    \mc{E}_{1}&=\left\{|\{n:N^{1/2}|z_{n}-z_{0}|\leq r\}|\leq N^{\xi}\right\},\\
    \mc{E}_{2}&=\bigcap_{N^{1/2}|z_{n}-z_{0}|<r}\left\{\Tr{\Im G_{z}(w)F\Im G_{z}(w)F^{*}}\leq \frac{N^{\xi}}{N^{2}(\Im w)^{2}},\,|w|\in(0,N^{-1}]\right\},\\
    \mc{E}_{3}&=\bigcap_{N^{1/2}|z_{n}-z_{0}|<r}\left\{\left|\frac{N}{O_{nn}}-\frac{N}{O_{\eta,\zeta}(z_{n})}\right|\lesssim\frac{N}{O_{nn}}\cdot N^{-3\zeta}+N^{\xi-\epsilon}\right\}.
\end{align*}
Note that on $\mc{E}_{2}$ we have
\begin{align*}
    O_{nn}&\geq\frac{2}{N\eta^{2}\Tr{\Im G_{z}(i\eta)F\Im G_{z}(i\eta)F^{*}}}\geq 2N^{1-\xi}.
\end{align*}
By the single resolvent averaged local law and \eqref{eq:aPriori1}, for any $D>0$ we have
\begin{align*}
    P(\mc{E}_{1}\cap\mc{E}_{2})&\geq1-N^{-D},
\end{align*}
and hence by the deterministic bounds in \eqref{eq:detL} and \eqref{eq:detLhat} we have
\begin{align*}  
    \mbb{E}\mc{L}&=\mbb{E}1_{\mc{E}_{1}\cap\mc{E}_{2}}\mc{L}+O(N^{-D}),\\
    \mbb{E}\wh{\mc{L}}&=\mbb{E}1_{\mc{E}_{1}\cap\mc{E}_{2}}\wh{\mc{L}}+O(N^{-D}).
\end{align*}
On the event $\mc{E}_{1}\cap\mc{E}_{2}$, we have the improved bounds
\begin{align*}
    1_{\mc{E}_{1}\cap\mc{E}_{2}}|\mc{L}|&\lesssim N^{(L+1)\xi},\\
    1_{\mc{E}_{1}\cap\mc{E}_{2}}|\wh{\mc{L}}|&\lesssim N^{(L+1)\zeta}.
\end{align*}
By Lemma \ref{lem:approximateOverlap}, for sufficiently small $\epsilon$ we have
\begin{align*}
    P(\mc{E}_{3})&\geq1-N^{-\epsilon},
\end{align*}
and so with the improved bounds we obtain
\begin{align*}
    \mbb{E}1_{\mc{E}_{1}\cap\mc{E}_{2}}\mc{L}&=\mbb{E}1_{\mc{E}_{1}\cap\mc{E}_{2}\cap\mc{E}_{3}}\mc{L}+O(N^{-\epsilon/2+(L+1)\xi}),\\
    \mbb{E}1_{\mc{E}_{1}\cap\mc{E}_{2}}\wh{\mc{L}}&=\mbb{E}1_{\mc{E}_{1}\cap\mc{E}_{2}\cap\mc{E}_{3}}\wh{\mc{L}}+O(N^{-\epsilon/2+L\zeta}).
\end{align*}
Now we can compare $\mc{L}$ and $\wh{\mc{L}}$ using the definition of $\mc{E}_{3}$. If $\xi$ is sufficiently smaller than $\zeta$ which itself is sufficiently smaller than $\epsilon$, then
\begin{align*}
    \left|\mbb{E}1_{\mc{E}_{1}\cap\mc{E}_{2}\cap\mc{E}_{3}}(\mc{L}-\wh{\mc{L}})\right|&\lesssim N^{-\delta},
\end{align*}
for some $\delta>0$ depending on $L,\epsilon,\zeta$ and $\xi$. By the triangle inequality we obtain
\begin{align*}
    \left|\mbb{E}(\mc{L}-\wh{\mc{L}})\right|&\lesssim N^{-\delta}.
\end{align*}
    
At this point we simply follow the same steps as in the proof of Lemma \ref{lem:svComparison} with
\begin{align*}
    g_{\eta}(z)&:=g\left(\frac{N}{O_{\eta,\zeta}(z)}\right)
\end{align*}
using the a priori bounds in Lemma \ref{lem:aPrioriOverlap}. For real matrices we replace $f_{z_{0}}$ with $\wt{f}_{z_{0}}$ which is supported in
\begin{align*}
    \left\{x+iy\in\mbb{C}:N^{1/2}|x-z_{0}|\leq r,\,|y|\leq N^{-1/2-\tau}\right\},
\end{align*}
for some $\tau>0$. Using the universality of correlation functions for Gauss-divisible matrices in \cite[Theorem 2.1]{osman_bulk_2024} and a moment matching argument (this is in fact part of the proof of \cite[Theorem 2.2]{osman_bulk_2024}) we can deduce that for sufficiently small $\tau>0$,
\begin{align*}
    P\left(\left|\{n:|z_{n}-z_{0}|<N^{-1/2-\tau}\}\right|>1\right)\lesssim N^{-\tau},
\end{align*}
and so
\begin{align*}
    \mbb{E}\left[\sum_{n=1}^{N_{\beta}}f_{z_{0}}(z_{n})g_{\eta}(z_{n})\right]&=\mbb{E}\left[\sum_{n=1}^{N}f_{z_{0}}(z_{n})g_{\eta}(z_{n})\right]+O(N^{-\delta}).
\end{align*}
The sum on the right hand side is over the whole spectrum, so we can apply Girko's formula and argue as in the complex case.
\end{proof}

We are now in a position to prove Theorem \ref{thm1}.
\begin{proof}[Proof of Theorem \ref{thm1}]
We give the argument for $z_{0}\in\mbb{D}_{\beta}$; the case $z_{0}\in\mbb{T}_{\beta}$ is analogous. We find a non-Hermitian Wigner matrix $\wt{X}$ such that $X$ and $M=\frac{1}{\sqrt{1+t}}(\wt{X}+\sqrt{t}Y)$ are $t$-matching. By Proposition \ref{prop:gaussDivisible}, for $t\geq N^{-1/3+\epsilon}$ and bounded $g$ we have
\begin{align}
    \mbb{E}_{\sqrt{1+t}M}\left[\frac{1}{N^{\beta/2}}\sum_{n=1}^{N_{\beta}}f_{\beta,z_{0}}(z_{n})g(1/S_{n})\right]&=\int_{\mbb{F}_{\beta}\times[0,\infty]}f_{\beta,z_{0}}(z)g(s)\rho_{\beta,bulk}(z,1/s)\,\mathrm{d}m(z,1/s)\nonumber\\&+O(N^{-\delta}),\label{eq:1/S}
\end{align}
where the rescaled overlap $S_{n}$ is defined in \eqref{eq:SBulkGauss}. Using the local law in Proposition \ref{prop:ll} and the explicit expression for the deteriminstic approximation of the resolvent (see Section \ref{sec:deterministicBounds}) we can obtain
\begin{align*}
    \sigma_{z,t}&=1+O(N^{-\delta}),\\
    \frac{\eta_{z,t}^{2}}{t^{2}}&=1-|z|^{2}+O(N^{-\delta}),
\end{align*}
and hence
\begin{align*}
    S_{n}&=\left[1+O(N^{-\delta})\right]\frac{O_{nn}}{N(1-|z|^{2})}.
\end{align*}
Using the comparison result in Lemma \ref{lem:overlapComparison} we obtain \eqref{eq:1/S} for general non-Hermitian Wigner matrices and bounded $g$ satisfying the assumptions of Lemma \ref{lem:overlapComparison}. Taking $g(x)$ to be equal to 1 on $x<N^{-\epsilon}$ and 0 on $x>2N^{-\epsilon}$, we deduce that there is a $\delta>0$ such that
\begin{align}
    P\left(O_{nn}>N^{1+\epsilon},\,z_{n}\in\mbb{D}_{\beta}\right)\leq N^{-\delta}.
\end{align}
Note that strictly speaking such a $g$ does not satisfy the conditions in Lemma \ref{lem:overlapComparison} but the same proof works since extra factors of $N^{\epsilon}$ from the derivatives of $g$ can be absorbed in the error for sufficiently small $\epsilon>0$. Using this bound we have
\begin{align*}
    \frac{O_{\eta,\zeta}(z_{n})}{N}&=\frac{O_{nn}}{N}+O(N^{-\delta}),\quad z_{n}\in\mbb{D}_{\beta},
\end{align*}
with probability at least $1-N^{-\epsilon}$ for some $\epsilon,\delta>0$. We can now repeat the argument of Lemma \ref{lem:overlapComparison} with $g(S_{n})$ in place of $g(1/S_{n})$ (here we have to assume $g$ and $g'$ are bounded since we do not have a very high probability upper bound on $S_{n}$) and hence conclude the proof.
\end{proof}

\paragraph{Acknowledgements}
This work was supported by the Royal Society, grant number \\RF/ERE210051.

\appendix
\section{Bounds on the Deterministic Approximation: Proof of Lemma \ref{lem:MBounds}}\label{sec:deterministicBounds}
From the definition of $M_{z}(w_{1},F,w_{2})$ in \cite[Eq. (5.7) and (5.8)]{cipolloni_central_2023}, we have the explicit formulae
\begin{align*}
    M_{z}(w_{1},F,w_{2})&=\begin{pmatrix}A_{11}&A_{12}\\A_{21}&A_{22}\end{pmatrix},
\end{align*}
where
\begin{align*}
    A_{11}&=-\frac{\bar{z}m_{1}(1-u_{1})u_{2}}{(1-|z|^{2}u_{1}u_{2})^{2}-m_{1}^{2}m_{2}^{2}},\\
    A_{12}&=\frac{m_{1}m_{2}(1-u_{1}u_{2})}{(1-|z|^{2}u_{1}u_{2})^{2}-m_{1}^{2}m_{2}^{2}},\\
    A_{21}&=\bar{z}^{2}\left(u_{1}u_{2}+\frac{m_{1}^{2}(1-u_{1})u_{2}^{2}+m_{2}^{2}(1-u_{2})u_{1}^{2}}{(1-|z|^{2}u_{1}u_{2})^{2}-m_{1}^{2}m_{2}^{2}}\right),\\
    A_{22}&=-\frac{\bar{z}m_{2}(1-u_{2})u_{1}}{(1-|z|^{2}u_{1}u_{2})^{2}-m_{1}^{2}m_{2}^{2}}.
\end{align*}
In particular,
\begin{align*}
    \Tr{M_{z}(w_{1},F,w_{2})F^{*}}&=A_{12}.
\end{align*}

Recall that $w_{j}=E_{j}+i\eta_{j}$ and $m_{j}=sigma_{j}+i\rho_{j}$. We begin with the claim that
\begin{align}
    \left|(1-|z|^{2}u_{1}u_{2})^{2}-m_{1}^{2}m_{2}^{2}\right|&\geq\frac{1}{4}\phi(w_{1},w_{2})\sum_{j=1,2}\left(\frac{\eta_{j}}{\rho_{j}+\eta_{j}}+\frac{E_{j}}{\sigma_{j}+E_{j}}\right).\label{eq:denom}
\end{align}
Taking the real and imaginary parts of the cubic equation for $m_{j}$ we obtain
\begin{align*}
    1-|z|^{2}|u_{j}|^{2}-|m_{j}|^{2}&=(|z|^{2}|u_{j}|^{2}+|m_{j}|^{2})\cdot\frac{\eta}{\rho_{j}}=\frac{\eta}{\rho_{j}+\eta},\\
    1-|z|^{2}|u_{j}|^{2}+|m_{j}|^{2}&=(|z|^{2}|u_{j}|^{2}-|m_{j}|^{2})\cdot\frac{E_{j}}{\sigma_{j}}=\frac{E_{j}}{\sigma_{j}+E_{j}}.
\end{align*}
Note that since $1-|z|^{2}|u_{j}|^{2}+|m_{j}|^{2}>0$ we have $\frac{E_{j}}{\sigma_{j}+E_{j}}>0$. Now observe that since $|z|^{2}|u_{1}u_{2}|<1$ we have
\begin{align*}
    \left|(1-|z|^{2}u_{1}u_{2})^{2}-m_{1}^{2}m_{2}^{2}\right|&\geq\left|1-|z|^{2}|u_{1}u_{2}|\right|\cdot\left|1-|z|^{2}|\Re(u_{1}u_{2})|-|\Re(m_{1}m_{2})|\right|\nonumber\\
    &=\left|1-|z|^{2}|u_{1}u_{2}|\right|\cdot\phi(w_{1},w_{2}).
\end{align*}
For the first factor in the right hand side above we obtain
\begin{align*}
    \left|1-|z|^{2}|u_{1}u_{2}|\right|&\geq\frac{1}{2}\left(1-|z|^{2}|u_{1}|^{2}+1-|z|^{2}|u_{2}|^{2}\right)\\
    &=\frac{1}{4}\left(\frac{\eta_{1}}{\rho_{1}+\eta_{1}}+\frac{E_{1}}{\sigma_{1}+E_{1}}+\frac{\eta_{2}}{\rho_{2}+\eta}+\frac{E_{2}}{\sigma_{2}+E_{2}}\right),
\end{align*}
which proves \eqref{eq:denom}.

Now we note that
\begin{align*}
    1-u_{1}u_{2}&=1-u_{1}+1-u_{2}-(1-u_{1})(1-u_{2}),
\end{align*}
and
\begin{align*}
    |1-u_{j}|&=\left|\frac{w_{j}}{m_{j}+w_{j}}\right|\\
    &\leq\sqrt{2}\cdot\frac{|E_{j}|+\eta_{j}}{|\Re m_{j}+E_{j}|+|\Im m_{j}+\eta|}\\
    &\leq\sqrt{2}\left(\frac{E_{j}}{\sigma_{j}+E_{j}}+\frac{\eta_{j}}{\rho_{j}+\eta_{j}}\right).
\end{align*}
Thus we find
\begin{align*}
    |1-u_{1}u_{2}|&\leq2\left(|1-u_{1}|+|1-u_{2}|\right)\\
    &\leq2\sqrt{2}\left(\frac{\eta_{1}}{\rho_{1}+\eta_{1}}+\frac{E_{1}}{\sigma_{1}+E_{1}}+\frac{\eta_{2}}{\rho_{2}+\eta_{2}}+\frac{E_{2}}{\sigma_{2}+E_{2}}\right).
\end{align*}
In the first inequality we used $|u_{j}|\leq1$. Combining this with the bound on the denominator in \eqref{eq:denom} we obtain \eqref{eq:M12FF*} and \eqref{eq:M12F}.

We prove \eqref{eq:phiBound} as follows:
\begin{align*}
    \phi(w_{1},w_{2})&=1-|z|^{2}|\Re(u_{1}u_{2})|-|\Re(m_{1}m_{2})|\\
    &\geq1-\frac{|z|^{2}|u_{1}|^{2}+|z|^{2}|u_{2}|^{2}}{2}-\frac{|m_{1}|^{2}+|m_{2}|^{2}-(|m_{1}|-|m_{2}|)^{2}}{2}\\
    &=\frac{1-|z|^{2}|u_{1}|^{2}-|m_{1}|^{2}}{2}+\frac{1-|z|^{2}|u_{2}|^{2}-|m_{2}|^{2}}{2}+\frac{(|m_{1}|-|m_{2}|)^{2}}{2}\\
    &=\frac{1}{2}\left(\frac{\eta_{1}}{\rho_{1}+\eta_{1}}+\frac{\eta_{2}}{\rho_{2}+\eta_{2}}+(|m_{1}|-|m_{2}|)^{2}\right)\\
    &\gtrsim\frac{\eta_{1}}{\rho_{1}}+\frac{\eta_{2}}{\rho_{2}}+(|m_{1}|-|m_{2}|)^{2},
\end{align*}
where the last line follows from the bound $|\rho_{j}|\gtrsim|\eta_{j}|$.

To estimate the norm of higher order deterministic approximations, we observe that all deterministic approximations are in the span of $\{E_{+},E_{-},F,F^{*}\}$ (i.e. the $2\times2$ block matrices whose blocks are multiples of the identity). Thus
\begin{align}
    \left\|M(w_{1},A_{1},w_{2},A_{2},w_{3})\right\|&\lesssim\max_{A_{3}\in\{E_{+},E_{-},F,F^{*}\}}\left|\Tr{M(w_{1},A_{1},w_{2},A_{2},w_{3})A_{3}}\right|.\label{eq:Mnorm}
\end{align}
In principle we can use the explicit definition of $M_{z}(w_{1},A_{1},...,w_{k})$ in \cite[Definition 4.1]{cipolloni_optimal_2024} to obtain the required bounds, but it is simpler to use the so-called ``meta model" argument (in fact the explicit definition itself is obtained in \cite{cipolloni_optimal_2024} by such an argument; see also the earlier work \cite{najim_gaussian_2016} for this argument in a different context). The idea is to transfer relations satisfied by resolvent chains to the deterministic equivalents. One constructs the random matrix $\mbf{X}$ of size $Nd\times Nd$ whose entries have mean zero and variance $1/Nd$ and considers the deterministic equivalent $\mbf{M}(w_{1},\mbf{B}_{1},...,w_{k})$ of $\mbf{G}_{z}(w_{1})\mbf{B}_{1}\cdots\mbf{G}_{z}(w_{k})$, where
\begin{align*}
    \mbf{G}_{z}(w)&:=\begin{pmatrix}-w&\mbf{X}-z\\\mbf{X}^{*}-\bar{z}&-w\end{pmatrix}^{-1}.
\end{align*}
For $\mbf{B}_{j}=B_{j}\otimes 1_{M}$ we have $\mbf{M}_{z}(w_{1},\mbf{B}_{1},...,w_{k})=M_{z}(w_{1},B_{1},...,w_{k})\otimes1_{M}$. If the spectral parameters $w_{j}\in\mbb{C}\setminus\mbb{R}$ are independent of $M$, then when $N$ is fixed and $M\to\infty$ we are in the global regime and can obtain
\begin{align*}
    \lim_{M\to\infty}\mbb{E}\Tr{\mbf{G}_{1}\mbf{B}_{1}\cdots\mbf{G}_{k}\mbf{B}_{k}}&=\Tr{M_{z}(w_{1},B_{1},...,w_{k})B_{k}}
\end{align*}
by arguing as in \cite[Lemma D.1]{cipolloni_optimal_2024}. In our case we want to bound the right hand side of \eqref{eq:Mnorm}, which we do by transferring an inequality satisfied by resolvents to the deterministic approximation:
\begin{align*}
    \left|\Tr{M_{z}(w_{1},F,\wh{w}_{2},F^{*},\bar{w}_{1})A}\right|&=\left|\lim_{M\to\infty}\mbb{E}\Tr{\mbf{G}_{1}\mbf{F}\Im\mbf{G}_{2}\mbf{F}^{*}\mbf{G}^{*}_{1}\mbf{A}}\right|\\
    &\leq\lim_{M\to\infty}\mbb{E}\left|\Tr{\mbf{G}_{1}\mbf{F}\Im\mbf{G}_{2}\mbf{F}^{*}\mbf{G}^{*}_{1}\mbf{A}}\right|\\
    &\leq\lim_{M\to\infty}\mbb{E}\frac{\|\mbf{A}\|}{\eta_{1}}\Tr{\Im\mbf{G}_{1}\mbf{F}\Im\mbf{G}_{2}\mbf{F}^{*}}\\
    &=\frac{\|A\|}{\eta_{1}}\Tr{M_{z}(\wh{w}_{1},F,\wh{w}_{2})F^{*}}\\
    &\leq\frac{\|A\|\phi^{av}_{2}(w_{1},w_{2})}{\eta_{1}}.
\end{align*}
In the third line we used Cauchy-Schwarz.

\section{Resolvent Estimates in the Spectral Gap}\label{sec:gap}
The analysis of Gauss-divisible matrices at the edge requires understanding the behaviour of the resolvent of the Hermitisation of $X-z$ for $|z-z_{0}|\lesssim N^{-1/2}$ and $|z_{0}|=\sqrt{1+t}$. In this appendix we prove the following.
\begin{lemma}\label{lem:edge}
Let $X$ be a non-Hermitian Wigner matrix, $\epsilon>0$ and $t\geq N^{-2/5+\epsilon}$. Then $X\in\mc{E}_{edge}(t)$ with very high probability.
\end{lemma}

To prepare for the proof, recall that $\rho_{z}(w)=\Im m_{z}(w)$ and $m_{z}(w)$ satisfies the cubic equation \eqref{eq:cubic}. When $|z|>1$, $\rho_{z}$ is zero in an interval $[-\Delta_{z}/2,\Delta_{z}/2]$, where $\Delta_{z}\simeq(|z|-1)^{3/2}$. When $|z|\gtrsim 1+N^{-1/2+\epsilon}$, we have the following two-resolvent local law from \cite[Theorem 3.5]{cipolloni_universality_2024}.
\begin{proposition}[Theorem 3.5 in \cite{cipolloni_universality_2024}]\label{prop:twoLL}
Let $\xi>0$ and $c>0$ be sufficiently small. Let $z\in\mbb{C}$ and $\eta>0$ such that $1\leq|z|\leq c$ and $N\eta\rho\geq N^{\xi}$, where $\rho:=\rho_{z}(i\eta)$. Then
\begin{align}
    \left|\Tr{G_{z}(i\eta)}-m_{z}(i\eta)\right|&\prec\frac{1}{N\eta},\\
    \left|\Tr{G_{z}(i\eta)FG_{z}(i\eta)-M_{z}(i\eta,F,i\eta)}\right|&\prec\frac{\rho}{N\eta^{2}},\label{eq:twoLL1}\\
    \left|\Tr{(G_{z}(i\eta)FG_{z}(i\eta)-M_{z}(i\eta,F,i\eta))F^{*}}\right|&\prec\frac{\rho^{5/2}}{N^{1/2}\eta^{3/2}}.\label{eq:twoLL2}
\end{align}
\end{proposition}
In the following lemma we extend this to negative values of $\eta^{2}$.
\begin{lemma}\label{lem:ext}
Let $\Delta_{z}\gtrsim t^{3/2}$ and $\eta_{0}=N^{\xi}\sqrt{\frac{t}{N}}$. Then we have
\begin{align}
    \left|\Tr{H^{n}_{z}(\eta_{1})}-\Tr{H^{n}_{z}(\eta_{0})}\right|&\prec\frac{\eta_{0}^{2}}{t^{3}}\Tr{H^{n}_{z}(\eta_{0})},\label{eq:ext1}\\
    \left|\Tr{H_{z}(\eta_{1})\wt{H}_{z}(\eta_{2})}-\Tr{H_{z}(\eta_{0})\wt{H}_{z}(\eta_{0})}\right|&\prec\frac{\eta_{0}^{2}}{t^{3}}\Tr{H_{z}(\eta_{0})\wt{H}_{z}(\eta_{0})},\label{eq:ext2}\\
    \left|\Tr{H_{z}(\eta_{1})X_{z}H_{z}(\eta_{2})}-\Tr{H_{z}(\eta_{0})X_{z}H_{z}(\eta_{0})}\right|&\prec\frac{\eta_{0}^{2}}{t^{11/2}},\label{eq:ext3}
\end{align}
uniformly in $\eta^{2}_{1},\eta^{2}_{2}\in[-\eta^{2}_{0},\eta^{2}_{0}]$.
\end{lemma}
\begin{proof}
The bounds in \eqref{eq:ext1} and \eqref{eq:ext2} follow directly from the fact that $\|H_{z}\|\prec t^{-3}$ and
\begin{align*}
    H_{z}(\eta)&=\left(1+(\eta^{2}-\eta_{0}^{2})H_{z}(\eta_{0})\right)^{-1}H_{z}(\eta_{0}).
\end{align*}

Now consider \eqref{eq:ext3}. Let $\mbf{u}_{n},\mbf{v}_{n}$ be the left and right singular vectors of $X_{z}$. We have the spectral decomposition 
\begin{align*}
    G_{z}(w)&=\sum_{n=1}^{N}\frac{1}{s_{n}^{2}(z)-w^{2}}\begin{pmatrix}w\mbf{u}_{n}\mbf{u}_{n}^{*}&s_{n}(z)\mbf{u}_{n}\mbf{v}_{n}^{*}\\s_{n}(z)\mbf{v}_{n}\mbf{u}_{n}^{*}&w\mbf{v}_{n}\mbf{v}_{n}^{*}\end{pmatrix},
\end{align*}
from which we obtain
\begin{align*}
    x(\eta_{1},\eta_{2})&:=\Tr{H_{z}(\eta_{1})X_{z}H_{z}(\eta_{2})}\\
    &=\frac{1}{i(\eta_{1}+\eta_{2})}\Tr{G_{z}(i\eta_{1})F^{*}G_{z}(i\eta_{2})}\\
    &=\frac{8}{N}\sum_{n=1}^{N}\frac{s_{n}(z)\mbf{u}_{n}^{*}\mbf{v}_{n}}{(s_{n}^{2}(z)+\eta^{2}_{1})(s_{n}^{2}(z)+\eta_{2}^{2})}.
\end{align*}
Comparing $x(\eta_{1},\eta_{2})$ and $x(\eta_{0},\eta_{0})$ using the facts that $0\leq\eta\leq\eta_{0}$ and $s_{1}(z)\geq ct^{3/2}>\eta_{0}$ with probability $1-N^{-D}$ we find
\begin{align*}
    \left|x(\eta_{1},\eta_{2})-x(\eta_{0},\eta_{0})\right|&\prec\frac{1}{N}\sum_{n=1}^{N}\frac{\eta_{0}^{2}|\mbf{u}_{n}^{*}\mbf{v}_{n}|}{s_{n}^{5}}+\frac{1}{N}\sum_{n=1}^{N}\frac{\eta_{0}^{4}|\mbf{u}_{n}^{*}\mbf{v}_{n}|}{s_{n}^{7}}.
\end{align*}
We use the bound $|\mbf{u}_{n}^{*}\mbf{v}_{n}|\leq1$ and \eqref{eq:ext1} to obtain
\begin{align*}
    \frac{1}{N}\sum_{n=1}^{N}\frac{\eta_{0}^{2}|\mbf{u}_{n}^{*}\mbf{v}_{n}|}{s_{n}^{5}}&\prec\frac{\eta_{0}^{2}}{t^{11/2}}.
\end{align*}
and
\begin{align*}
    \frac{1}{N}\sum_{n=1}^{N}\frac{\eta_{0}^{4}|\mbf{u}_{n}^{*}\mbf{v}_{n}|}{s^{7}_{n}}&\prec\frac{\eta_{0}^{4}}{t^{17/2}}.
\end{align*}
\end{proof}

\begin{proof}[Proof of Lemma \ref{lem:edge}]
We have to verify the conditions \eqref{B1} to \eqref{B6}. The norm bound $\|X\|\leq e^{\log^{2}N}$ with very high probability is much weaker than the standard result $\|X\|\lesssim1$ with very high probability (which follows by e.g. the moment method). By the rigidity in \eqref{eq:rigidity}, we have
\begin{align*}
    \left|s_{1}(z)-\frac{\Delta_{z}}{2}\right|&\prec\max\left\{\frac{1}{N^{3/4}},\frac{\Delta_{z}^{1/9}}{N^{2/3}}\right\}.
\end{align*}
When $\Delta_{z}\gtrsim t^{3/2}$ and $t\geq N^{-1/2+\epsilon}$, this becomes
\begin{align*}
    \left|s_{1}(z)-\frac{\Delta_{z}}{2}\right|&\prec\frac{\Delta_{z}^{1/9}}{N^{2/3}},
\end{align*}
which implies \eqref{B2}.

From the cubic equation for $m_{z}(w)$ we can obtain the asymptotic series
\begin{align*}
    \frac{m_{z}(w)}{w}&=\frac{1}{|z|^{2}-1}+\frac{|z|^{4}w^{2}}{(|z|^{2}-1)^{4}}+O\left(\frac{w^{4}}{(|z|^{2}-1)^{7}}\right),\quad\frac{w^{2}}{(|z|^{2}-1)^{3}}\leq1.
\end{align*}
Combining this with Proposition \ref{prop:twoLL} and Lemma \ref{lem:ext} we obtain \eqref{B3} to \eqref{B6}.
\end{proof}

\newpage
\section{Numerical Study of Singular Vector Overlap}\label{sec:numerics}
We generate the data $|\mbf{w}_{m}^{*}F\mbf{w}_{m}|$ and $|\lambda_{n}|$ from a single realisation of $X-1$, where $X\sim Gin_{2}(1000)$, and plot them in the figures below. We can see that the decay of the overlap becomes sharper as $|m|$ increases.
\begin{figure}[htbp]
\centering
\captionsetup{width=0.8\columnwidth}
\begin{subfigure}{0.4\columnwidth}
    \includegraphics[width=\columnwidth,keepaspectratio=true]{{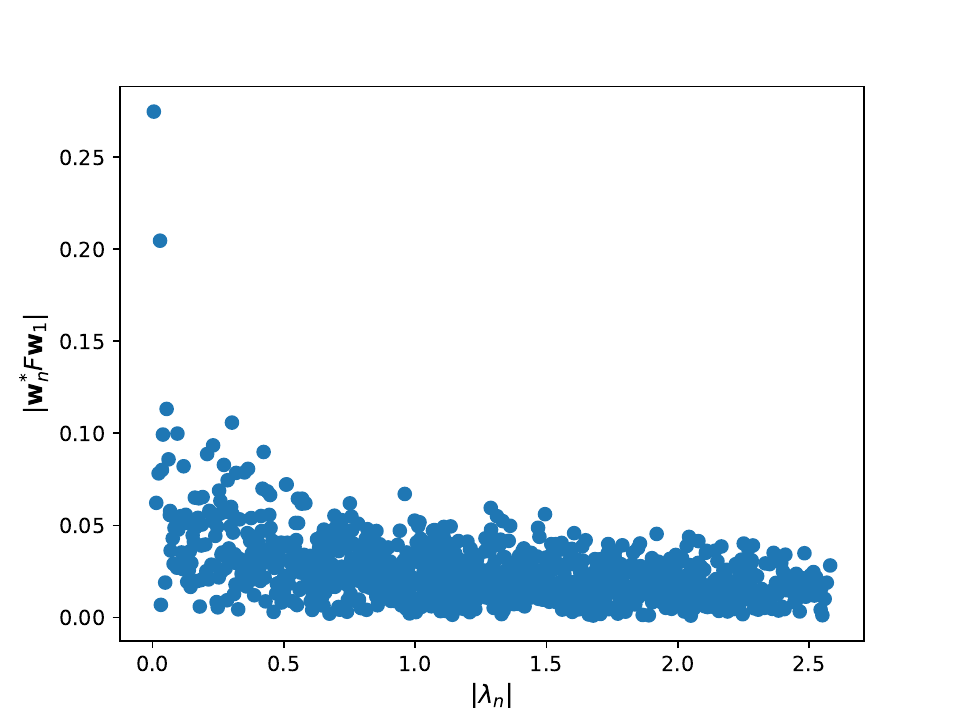}}
\end{subfigure}%
\begin{subfigure}{0.4\columnwidth}
    \includegraphics[width=\columnwidth,keepaspectratio=true]{{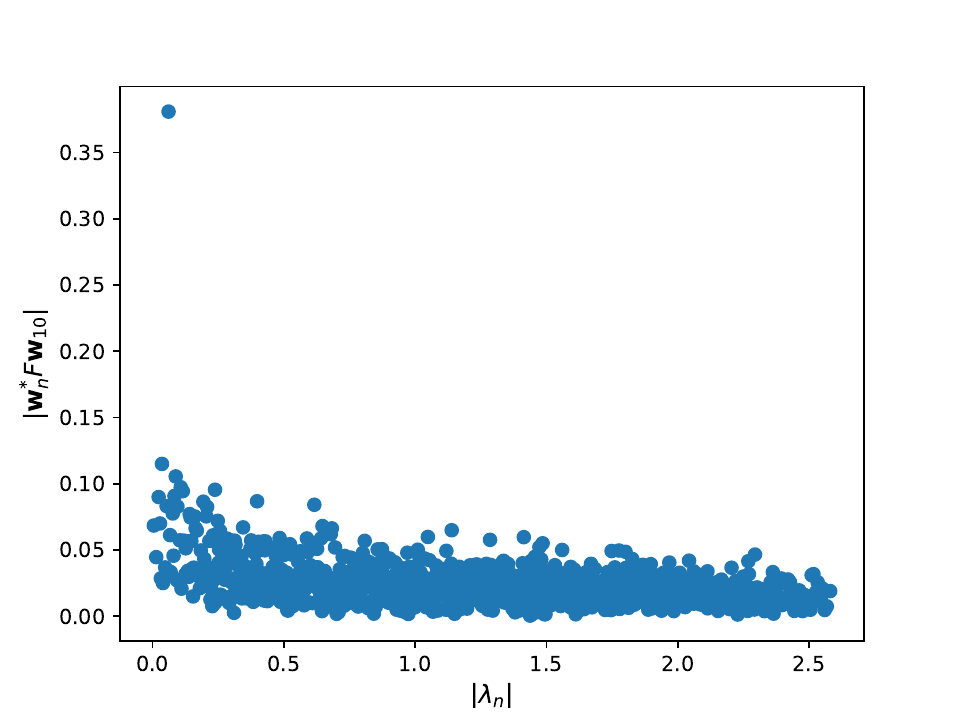}}
\end{subfigure}\\
\begin{subfigure}{0.4\columnwidth}
    \includegraphics[width=\columnwidth,keepaspectratio=true]{{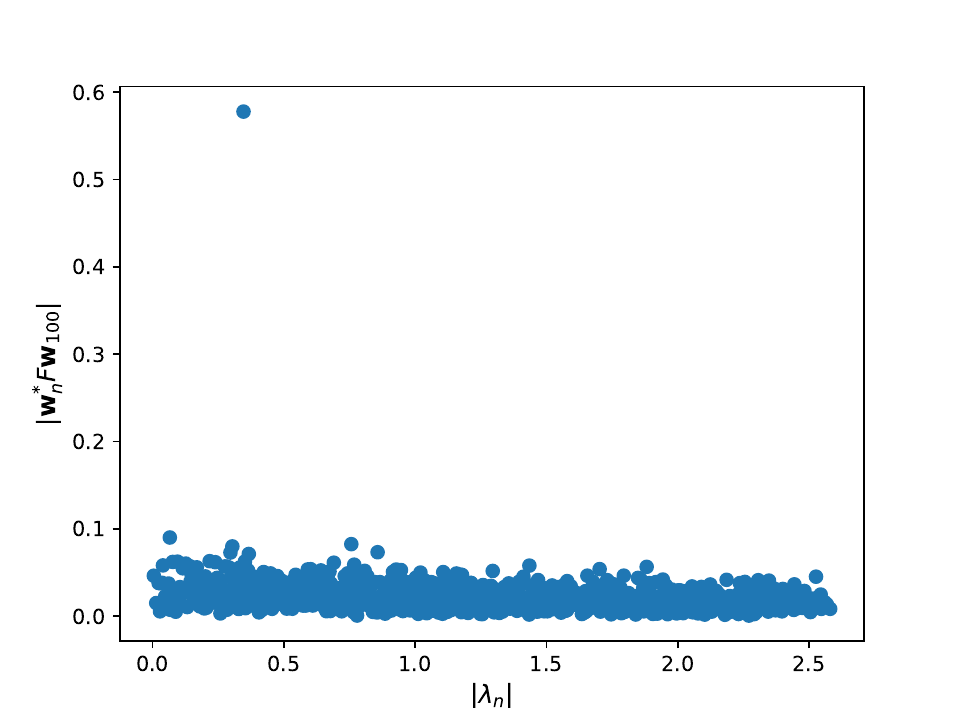}}
\end{subfigure}%
\begin{subfigure}{0.4\columnwidth}
    \includegraphics[width=\columnwidth,keepaspectratio=true]{{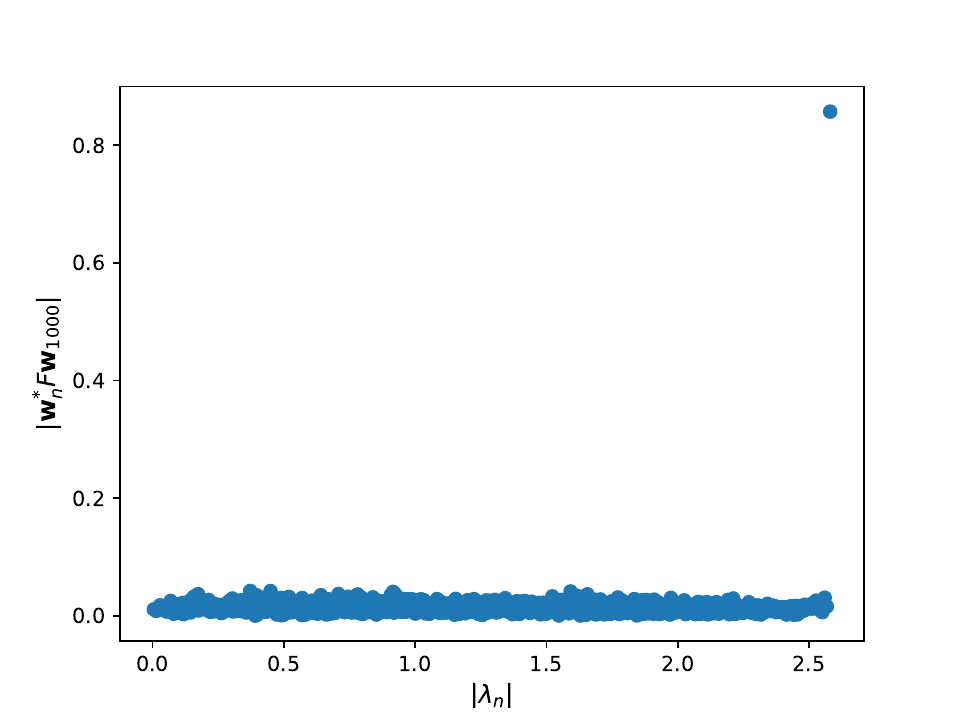}}
\end{subfigure}
\begin{caption}{\small Clockwise from the top left: the overlap $|\mbf{w}_{n}^{*}F\mbf{w}_{m}|$ against the singular value $|\lambda_{n}|$ for different values of $m=1,10,100,1000$. The data is obtained from diagonalising $X-1$ where $X\sim Gin_{2}(1000)$.}
\end{caption}
\end{figure}

\end{document}